\documentclass{amsart}

\usepackage{amsmath}
\usepackage{amsfonts}
\usepackage{amssymb}
\usepackage{amsthm}
\usepackage{comment}
\usepackage{epsfig}
\usepackage{psfrag}
\usepackage{mathrsfs}
\usepackage{amscd}
\usepackage[all]{xy}
\usepackage{rotating}
\usepackage{lscape}
\usepackage{amsbsy}
\usepackage{verbatim}
\usepackage{moreverb}
\usepackage{color}
\usepackage{bbm}
\usepackage{eucal}
\usepackage{mathtools}
\usepackage{tikz}
\usetikzlibrary{patterns,shapes.geometric,arrows,decorations.markings}
\usepackage{tikz-3dplot}
\usepackage{tikz-cd}
\usepackage{caption}
\usepackage{subcaption}
\usepackage[bottom]{footmisc}
\usepackage[bookmarks=true, linktocpage=true,
bookmarksnumbered=true, breaklinks=true,
pdfstartview=FitH, hyperfigures=false,
plainpages=false, naturalnames=true,
colorlinks=true, pagebackref=true,
pdfpagelabels]{hyperref}
\hypersetup{
	linkcolor = blue,
	citecolor = blue,
	urlcolor = blue,
	colorlinks = true,
}
\colorlet{lightgray}{black!15}

\tikzset{->-/.style={decoration={
  markings,
  mark=at position .5 with {\arrow{>}}},postaction={decorate}}}
\tikzset{midarrow/.style={decoration={
    markings,
    mark=at position {#1} with {\arrow{>}}},postaction={decorate}}}

\newtheorem{theorem}{Theorem}[subsection]
\newtheorem{prop}[theorem]{Proposition}
\newtheorem{lemma}[theorem]{Lemma}
\newtheorem{cor}[theorem]{Corollary}
\newtheorem{conj}[theorem]{Conjecture}

\theoremstyle{definition}
\newtheorem{definition}[theorem]{Definition}

\newtheorem{note}[theorem]{Note}

\newtheorem{observation}[theorem]{Observation}
\newtheorem{construction}[theorem]{Construction}
\newtheorem{terminology}[theorem]{Terminology}
\newtheorem{remark}[theorem]{Remark}
\newtheorem{example}[theorem]{Example}

\newtheorem{notation}[theorem]{Notation}

\theoremstyle{remark}

\definecolor{orange}{rgb}{.95,0.5,0}
\definecolor{light-gray}{gray}{0.75}
\definecolor{brown}{cmyk}{0, 0.8, 1, 0.6}
\definecolor{plum}{rgb}{.5,0,1}

\newcommand{\Hom}{{\rm{Hom}}}

\newcommand{\R}{\mathbb{R}}

\newcommand{\N}{\mathbb{N}}
\newcommand{\ds}{\displaystyle}

\newcommand{\s}{\mathbb{S}}
\newcommand{\lra}{\longrightarrow}
\newcommand{\tx}{\text}
\newcommand{\fin}{{\sf{Fin}_*}}
\newcommand{\Fin}{{\sf{Fin}}}

\newcommand{\surj}{{\sf{surj}}}
\newcommand{\op}{{\sf{op}}}
\DeclareMathOperator{\pr}{{\sf{pr}}}
\DeclareMathOperator{\Ran}{{\sf{Ran}}}
\DeclareMathOperator{\Exit}{{\sf{Exit}}}   
\newcommand{\exit}{{\sf{exit}}}
\newcommand{\un}{\underline}

\newcommand{\Act}{{\sf{act}}}
\newcommand{\G}{\mathrm{G}}

\DeclareMathOperator{\Ranu}{\sf{Ran}^{\sf{u}}}

\newcommand{\ra}{\rightarrow}

\newcommand{\Strat}{\text{Strat}}

\newcommand{\cylr}{{ \sf cylr}}
\newcommand{\xra}{\xrightarrow}
\newcommand{\hra}{\hookrightarrow}
\newcommand{\xhra}{\xhookrightarrow}
\newcommand{\tr}{{{\sf tr}}}
\newcommand{\F}{\mathcal{F}}
\newcommand{\C}{\mathcal{C}}

\newcommand{\Id}{\tx{Id}}
\newcommand{\Cat}{{\sf{Cat}}}

\newcommand{\D}{\mathcal{D}}

\newcommand{\Ar}{{\sf{Ar}}}
\newcommand{\Fun}{{{\sf Fun}}}
\newcommand{\h}{{\sf{hlt}}}
\newcommand{\Conf}{{\sf{Conf}}}

\newcommand{\shs}{\hspace{.1cm}}
\newcommand{\la}{\langle}
\newcommand{\ran}{\rangle}
\newcommand{\E}{\mathcal{E}}
\newcommand{\ev}{\text{ev}}

\newcommand{\pb}{\ar[dr, phantom, "\lrcorner", very near start]}

\newcommand{\Ner}{\mathcal{N}}
\newcommand{\mor}{{\sf{mor}}}

\DeclareMathOperator{\fib}{\sf fib}
\newcommand{\DDelta}{\mathbf{\Delta}}
\newcommand{\B}{\mathcal{B}}

\newcommand{\Alg}{{\sf Alg}}
\newcommand{\bl}{\bigl}
\newcommand{\br}{\bigr}

\DeclareMathOperator{\Bun}{\B{\sf{un}}}
\DeclareMathOperator{\spaces}{\cS\mathsf{paces}}
\DeclareMathOperator{\Spaces}{\cS\mathsf{paces}}
\DeclareMathOperator{\sk}{\sf sk}
\def\cS{\mathcal S}
\DeclareMathOperator{\RRef}{\mathcal{R}{\sf ef}}
\DeclareMathOperator{\V}{\mathcal{V}}
\DeclareMathOperator{\A}{\mathcal{A}}
\DeclareMathOperator{\cB}{\mathcal{B}}
\DeclareMathOperator{\uno}{\mathbbm{1}}
\DeclareMathOperator{\fB}{\mathfrak{B}}

\makeatletter
\@namedef{subjclassname@2020}{\textup{2020} Mathematics Subject Classification}
\makeatother
\pdfoutput=1

\vbadness=\maxdimen

\begin{document}

\title{Higher-categorical combinatorics of configuration spaces of Euclidean space}
\author{Anna Cepek}


\address{University of Oregon \\ Eugene, OR USA}
\email{cepek@uoregon.edu, anna.cepek@gmail.com}
\thanks{This work was supported by NSF award 1507704, IBS-R003-D1 and NSF-RTG grant DMS-2039316}


\begin{abstract}
We examine configurations of finite subsets of manifolds within 
the homotopy-theoretic context of $\infty$-categories by way of stratified spaces. 
Through these higher categorical means, 
we identify the homotopy types of such configuration spaces 
in the case of $n$-dimensional Euclidean space in terms of the category $\mathbf{\Theta}_n$. 
\end{abstract}





\maketitle

\tableofcontents

\section{Introduction}
Configuration spaces (of finite subsets of topological spaces) 
play fruitful roles across algebraic topology, geometric 
and differential topology, and applied topology
as Knudsen exposits in \cite{BK}.
For example, the embedding calculus of Goodwillie-Weiss 
uses them as fundamental building blocks to study differential topology \cite{W1, GW2}. 
More basically, the homotopy type of a fixed configuration space 
often detects subtle topological features of the background manifold.
For example, in \cite{SL}, Salvatore-Longoni identify two homotopy equivalent 
Lens spaces of distinct homeomorphism type whose configuration spaces 
of just two points are not homotopy equivalent.

We are motivated by Ayala-Francis's factorization homology 
which uses configuration spaces to define manifold invariants 
within the homotopy-theoretic framework of $\infty$-categories \cite{AF1}.
More specifically, motivated by the theory of stratified spaces 
and exit-path $\infty$-categories thereof \cite{Lu1, AFT, AFR}, 
we codify the homotopy type of configuration spaces of finite subsets of manifolds.
Our main result identifies such configuration spaces of $n$-dimensional Euclidean space 
$\R^n$ combinatorially in terms of the category $\mathbf{\Theta}_n$. 
Joyal introduced the `theta-categories' $\mathbf{\Theta}_n$ to present 
the first model of $(\infty, n)$-categories \cite{Joy3}. 
This was a direct generalization of his theory of quasi-categories 
\cite{Joy} as $\mathbf{\Theta}_1$ is defined to be the simplex category $\mathbf{\Delta}$. 
We use Berger's definition of $\mathbf{\Theta}_n$ as the n-fold 
wreath product of $\mathbf{\Delta}$ with itself \cite{Ber}. 

The relationship between configuration spaces of $\R^n$ 
and the category $\mathbf{\Theta}_n$ is a natural one as Ayala-Hepworth show in \cite{AH}
and has a simple idea at its core.  
With the following discussion, we hope to demonstrate this simple idea
before introducing any technical machinery.

\subsubsection{Finite subsets of Euclidean space and trees}
Let $S$ be a finite, nonempty subset of $\R^n$.
Consider a sequence of its projections
$$S \twoheadrightarrow S_{n-1} \twoheadrightarrow \cdots \twoheadrightarrow S_1$$
where $S_{n-i}$ is the finite subset of $\R^{n-i}$ 
obtained by projecting off the last $i$~coordinates of $S$ for each $1 \leq i \leq n-1$, 
and each surjection is induced by the projection map $\R^{n-i+1}~\twoheadrightarrow~\R^{n-i}$.
Consider the following two observations about $S$:
\begin{itemize}
\item \textit{Linear order on $S_1$:}
The projection of $S$ onto its set of first coordinates $S_1$ is a subset of $\R$ 
and therefore inherits a linear order from $\R$ equipped with the standard linear order. 
\item \textit{Linear order on each fiber:}
For each point $s \in S_{n-i}$, its fiber in $S_{n-i+1}$
naturally inherits a linear order from $\R$.
Indeed, it is a subset of the fiber over $s$ of the projection map 
$\R^{n-i+1} \twoheadrightarrow \R^{n-i}$ which canonically identifies with $\R$.
\end{itemize}

We call this linear ordering on the set of first coordinates of $S$
and each fiber of its projections an \textit{$n$-order} on $S$. 
More generally, an $n$-ordering of a finite set $S$ is a sequence of surjective maps 
$$S \xra{\sigma_{n-1}} S_{n-1} \ra \cdots \xra{\sigma_1} S_1$$ 
among finite sets together with a linear order on $S_1$ 
and a linear order on each fiber of each map in the sequence, namely, 
for each $1 \leq i \leq n-1$, a linear order on $\sigma_i^{-1}\left(s\right)$ for each $s \in S_i$. 
A $1$-ordered set $S$, then, is just a linearly ordered set. 

It is not hard to see that finite n-ordered sets are 
combinatorially codified by the following type of trees. 

\begin{definition}[Defs. 1 \& 2, \cite{AH}] \label{tree} 
The following compile to define \emph{planar level rooted trees}.
\begin{enumerate}
\item A \emph{level tree} is a finite, rooted tree, the \emph{root} of which is a choice of vertex thereof. 
The choice of root uniquely determines a direction to each edge 
such that there is a unique directed path from each vertex to the root. 
\item For each vertex, we may equip the set of edges directed toward the vertex with a linear order. 
A tree is called a \emph{planar level tree} if such an order is specified with respect to each vertex. 
\item A vertex is at \emph{level} $i$ if the directed path from the vertex to the root counts $i$ edges.
\item A tree has \emph{height} $n$ if the maximum level of all the vertices is $n$.  
\item A vertex is a \emph{leaf} if it has no edges directed towards it. 
\item A planar level tree of height $n$ is \emph{healthy} if all of its leaves are at level $n$. 
Otherwise, it is \textit{unhealthy}.
\end{enumerate}
We will say \emph{tree} to mean planar level rooted tree. 
\end{definition}
 
The following are depictions of trees, 
where the root is at the bottom of each diagram; 
on the left, $T_1$ is a healthy tree of height one, in the middle, 
$T_2$ and $T_2'$ are a healthy trees of height two, and on the far right, 
$T_2^{\sf uh}$ is an unhealthy tree of height two.

\begin{figure}[ht]
\begin{center}
\begin{tikzpicture}[scale=.4] 
 
\filldraw[black] (-10, -2) circle (4pt) {};
\filldraw[black] (-9,-4) circle (4pt);
\filldraw[] (-8.9,-5) node {$T_1$};
\filldraw[black] (-8,-2) circle (4pt) {};

\filldraw[black] (-1,0) circle (4pt) {};
\filldraw[black] (1,0) circle (4pt) {};
\filldraw[black] (0,-2) circle (4pt) { };
\filldraw[black] (0,-4) circle (4pt) { };
\filldraw[] (.1,-5) node{$T_2$};

\filldraw[black] (9,-4) circle (4pt){};
\filldraw[] (9.1,-5) node{$T_2'$};
\filldraw[black] (10,-2) circle (4pt){};
\filldraw[black] (8,0) circle (4pt) {};
\filldraw[black](8,-2) circle (4pt){};
\filldraw[black] (10,0) circle (4pt) {};

\filldraw[black] (17,-2) circle (4pt) {};
\filldraw[black] (18,-4) circle (4pt){};
\filldraw[] (18.1,-5) node{$T_2^{\sf uh}$};
\filldraw[black] (19,-2) circle (4pt) {};
\filldraw[black] (18,0) circle (4pt) {};
\filldraw[black] (20,0) circle (4pt){};

\draw 
(-10,-2) -- (-9,-4)
(-9,-4) -- (-8,-2)

(-1,0) -- (0,-2)
(1,0) -- (0,-2)
(0,-2) -- (0,-4)

(8,0) -- (8,-2)
(8,-2) -- (9,-4)
(9,-4) -- (10,-2)
(10,-2) -- (10,0)

(17,-2) -- (18,-4)
(18,-4) -- (19,-2)
(19,-2) -- (20,0)
(19,-2) -- (18,0);

\end{tikzpicture}
\caption{}
\label{trees}
\end{center}
\end{figure}

A healthy tree $T$ of height $n$ encodes an $n$-order on its set of leaves as follows. 
The $i$th set of the sequence is the set of vertices at level $i$. 
The assignment from the set of vertices at level $i$ to the set of vertices at level $i-1$ is canonical, 
given by assigning to each vertex at level $i$ the vertex at level $i-1$ that is adjacent to it. 
Such a sequence of maps is surjective precisely because $T$ is healthy. 
The \emph{planar} condition (2) of Definition~\ref{tree} 
fixes the set of incoming edges of each vertex with a linear order; 
this condition induces a linear order on each fiber 
in the sequence and on the set of vertices at level $1$. 
As a convention, we will fix these linear orders to be read from left to right on trees. 

A natural assignment from the set of finite, nonempty subsets of $\R^n$ 
to the set of healthy trees of height $n$ is given by the $n$-order.
Namely, each set is assigned to the tree that encodes its $n$-order. 
Although this assignment is not particularly strong in the setting of sets
(it is not even injective), it shines in the setting of topological spaces.
In \cite{AH}, Ayala-Hepworth show that this assignment is a homotopy equivalence. 
The set of finite subsets of $\R^n$ is naturally upgraded 
to the configuration space of $k$ ordered points in $\R^n$.
On the other hand, planar level trees depict the objects 
of Joyal's $\mathbf{\Theta}_n$ categories (Def.~\ref{theta}) 
as the authors show, and they codify these categories 
as topological spaces by taking their classifying spaces.
The reason this assignment strengthens to a homotopy 
equivalence in the context of topological spaces
is that the Fox-Neuwirth cells of the configuration space of 
$\R^n$ \cite{FN}, which are determined by $n$-orders,
are contractible \cite{GS} and form a partially ordered set.

\subsubsection*{Main results}
In this paper, we examine the natural assignment between 
finite subsets of $\R^n$ and planar level trees
in the context of $\infty$-categories, and our main result 
articulates the sense in which it is an equivalence.
Motivated by the theory of stratified spaces and 
exit-path $\infty$-categories thereof \cite{Lu1, AFT, AFR},
we introduce an $\infty$-category $\Exit\bl(\Ranu(\R^n)\br)$ 
which codifies the homotopy type of 
all the finite unordered configurations of points in $\R^n$ at once, 
including the empty set.
It also codifies the maps between the configuration spaces of different cardinalities 
by witnessing anticollision and vanishing of points, but not collision (Def.~\ref{exitdef2}). 
Within this framework, the relationship between configuration spaces of $\R^n$ 
and the category $\mathbf{\Theta}_n$ is most naturally phrased 
as a localization of $\infty$-categories. 
Heuristically, this means that upon formally inverting certain 
morphisms of (a subcategory of) $\mathbf{\Theta}_n$ we obtain our 
$\infty$-category of configurations of points in $\R^n$. 
We articulate an introductory version of our main result as follows. 

\begin{theorem}[Thm.~\ref{loc}]  \label{introloc}
For each $n \geq 1$, there is a localization of $\infty$-categories
\begin{equation}\label{Gn}
\mathbf{\Theta}_n^\Act \lra \Exit\bl(\Ranu(\R^n)\br)
\end{equation}
from the subcategory of $\mathbf{\Theta}_n$ consisting 
of active morphisms (Def.~\ref{act})
to the exit-path $\infty$-category of the unital Ran space of $\R^n$. 
\end{theorem}

Our main result is an enhancement of \cite{AH} because 
the subset of $\R^n$ of cardinality zero is included,
leading to a much larger and better-behaved equivalence.
For example, the source of (\ref{Gn}), $\mathbf{\Theta}_n^\Act$, 
decomposes as a wreath product 
because it has the empty tree (the tree of height $0$), 
and all of the morphisms that come along with that, 
namely all of the active morphisms.
However, without the unit the resulting subcategory, 
$\mathbf{\Theta}_n^\exit$ (Def.~\ref{exit}), 
does not decompose as a wreath product. 

As a corollary to Theorem~\ref{introloc}, we recover the infinity-categorical generalization 
of \cite{AH}, yielding all of Ayala-Hepworth's homotopy equivalences 
(as cardinality varies) at once, in the unordered setting. 
The configuration space of $k$ points in $\R^n$ is upgraded to an $\infty$-subcategory 
of $\Exit\bl(\Ranu(\R^n)\br)$ denoted $\Exit\bigl(\Ran(\R^n)\bigr)$ 
whose objects are \emph{nonempty} finite subsets of 
$\R^n$, and whose morphisms may witness anticollision of points, 
but neither collision, nor the vanishing of points (Def.~\ref{little exit}).
In other words, this corollary is the nonunital version of Theorem~\ref{introloc}. 

\begin{cor}[Cor. \ref{locexit}] \label{introlocexit} 
For each $n \geq 1$, there is a localization of $\infty$-categories
$$
\mathbf{\Theta}_n^\exit \lra \Exit\bigl(\Ran(\R^n)\bigr)
$$
from the subcategory of $\mathbf{\Theta}_n$ consisting of healthy trees and exiting morphisms
to the exit-path $\infty$-category of the Ran space of $\R^n$. 
\end{cor}

An immediate consequence of Theorem~\ref{introloc} is that the wreath product 
decomposition of $\mathbf{\Theta}_2^\Act$ induces a likewise 
decomposition of $\Exit\left(\Ranu\left(\R^2\right)\right)$. 

\begin{cor}[Cor.~\ref{cor 2}] \label{introcor 2}
There is a localization of $\infty$-categories
$$
\Exit\left(\Ranu\left(\R\right)\right) \wr \Exit\left(\Ranu\left(\R\right)\right) 
\lra \Exit\left(\Ranu\left(\R^2\right)\right).
$$
\end{cor}

\subsubsection*{Approach}
In \cite{AFR}, the authors define an `exit-path $\infty$-category' functor 
which takes in a (nice) stratified space and outputs an $\infty$-category 
whose objects are points of the input space and whose morphisms are `exit' paths 
in the space; see \S\ref{AFRT} for more details.
Because the output $\infty$-categories of this construction are akin to our main 
$\infty$-category $\Exit\bl(\Ranu(\R^n)\br)$,
our approach to the proof of Theorem~\ref{introloc} is motivated by their result that 
the exit-path $\infty$-category functor carries \emph{refinements} of stratified spaces 
(Def. \ref{refinement}) to \emph{localizations} (Def.~\ref{localization}) of their exit-path $\infty$-categories 
(Thm.~3.3.12, \cite{AFR}). 
However, for technical reasons, our exit-path $\infty$-category  
$\Exit\bl(\Ranu(\R^n)\br)$
does not have a stratified space in the background; see \S\ref{nottrivial} for more details.
Therefore, we directly construct an $\infty$-category 
$\Exit\bl(\Ranu(\un{\R}^n)\br)$ (Def. \ref{underline})
to interpolate between $\mathbf{\Theta}_n^\Act$ and $\Exit\bl(\Ranu(\R^n)\br)$.
Its construction is motivated by a (nonexisting) stratification 
of the unital Ran space of $\R^n$ by
cardinality and the Fox-Neuwirth cells (\S\ref{FNRanu}); see \S\ref{exitFNRanu} for more details.
We show that $\mathbf{\Theta}_n^\Act$ is equivalent to $\Exit\bl(\Ranu(\un{\R}^n)\br)$, 
and then prove that $\Exit\bl(\Ranu(\un{\R}^n)\br)$ 
localizes to $\Exit\bl(\Ranu(\R^n)\br)$.
Key to our approach for proving the localization is 
Theorem~\ref{BcSS} of \cite{AMG} which gives a method for 
identifying localizations in favorable cases.

Much of the proof of Corollary~\ref{introlocexit} is extrapolated from the proof of 
Theorem~\ref{introloc}.
Indeed, both the domain $\mathbf{\Theta}_n^\exit$ 
and codomain $\Exit\bigl(\Ran(\R^n)\bigr)$ are $\infty$-subcategories (respectively) of the 
domain $\mathbf{\Theta}_n^\Act$ and codomain 
$\Exit\bl(\Ranu(\R^n)\br)$ of Theorem~\ref{introloc}.

\subsubsection*{Motivation and conjectures}
The main advantage of codifying \cite{AH} in a higher-categorical framework is that it yields 
a direct connection between $\E_n$-algebras and $\left(\infty,n\right)$-categories 
as Ayala-Hepworth suggest in \cite{AH}. 
In particular, our work gives an approximation to the folklore conjecture that 
relates $\E_n$-algebras to $\left(\infty,n\right)$-categories. 
Since we have not found this conjecture explicitly stated anywhere, we will articulate it here. 
Let $\V$ be a symmetric monoidal $\infty$-category.
Consider the $\infty$-category $\Alg_{\E_n}\left(\V\right)$ of $\E_n$-algebras in $\V$.
Consider the $\infty$-category $\Cat_{\left(\infty,n\right)}\left(\V\right)$ of $\V$-enriched 
$\left(\infty,n\right)$-categories.  
Consider the object $\un{\uno} \in \Cat_{\left(\infty,n\right)}\left(\V\right)$ whose underlying 
$\left(\infty,n-1\right)$-category 
$\un{\uno}_{<n}\simeq \ast$ is terminal, 
and whose ${\sf nEnd}_{\un{\uno}}\left(\ast\right) = \uno \in \V$.

Say a morphism $\A \xra{F} \cB$ between $\left(\infty,n\right)$-categories is \emph{$n$-connected} 
if, for each $0\leq k\leq n$, each diagram among $\left(\infty,n\right)$-categories

\[ \begin{tikzcd}
\partial c_k \ar{r} \ar{d} & \A \ar{d}{F} \\
c_k \ar{r}\ar[dashrightarrow]{ur} & \cB 
\end{tikzcd} \]
admits a filler.
Here, $c_k$ is the $\left(\infty,n\right)$-category corepresenting the space of $k$-morphisms, 
and $\partial c_k \to c_k$ is the inclusion of its maximal $\left(\infty,k-1\right)$-subcategory.

\begin{conj}
Let $\V$ be a symmetric monoidal $\infty$-category.
There is a fully faithful left adjoint
\[
\mathfrak{B}^n\colon \Alg_{\E_n}\left(\V\right)
\longrightarrow
\Cat_{\left(\infty,n\right)}\left(\V\right)^{\uno/}~;
\]
the image consists of those morphisms $\un{\uno} \to \C$ between $\V$-enriched 
$\left(\infty,n\right)$-categories whose underlying morphism $\ast \to \C_{<n}$ between 
$\left(\infty,n-1\right)$-categories is $\left(n-1\right)$-connected. 
\end{conj}

If we specialize to the case that $\left(\V,\otimes\right) = \left(\spaces,\times\right)$, 
our results Theorem~\ref{introloc} and Corollary~\ref{introlocexit} 
give an approximation to this conjecture. 
Let us review some key results from the literature before we explain how. 

The collection of unordered configuration spaces of finite subsets of $\R^n$ exhibit 
the algebraic structure of the $\E_n$-operad. Lurie identifies this relationship 
as a fully faithful embedding
$${\Alg}_{\E_n}\left(\spaces\right) \hra {{\sf coShv}}^{\tx{cbl}}_{\spaces} \left(\Ran\left(\R^n\right)\right)$$ 
from $\E_n$-algebras valued in spaces to constructible cosheaves (locally constant on each stratum) 
on the Ran space of $\R^n$ (Thm.~5.5.4.10, \cite{Lu1}). 

Corollary~3.3.11 of \cite{AFR} yields an equivalence of $\infty$-categories
$${{\sf Shv}}^{\tx{cbl}}_{\spaces} \left(\Ran\left(\R^n\right)\right) 
\simeq \Fun\left(\Exit\bigl(\Ran(\R^n)\bigr), \spaces\right)$$
between space-valued constructible sheaves on $\Ran\left(\R^n\right)$ 
and space-valued functors on $\Exit\bigl(\Ran(\R^n)\bigr)$.
Thus, there is a fully faithful embedding
$$\Alg_{\E_n}\left(\spaces\right) \hra  
\Fun\left(\Exit\bigl(\Ran(\R^n)\bigr)^\op, \spaces\right).$$
We conjecture that this functor factors through space-valued functors 
on $\Exit\bl(\Ranu(\R^n)\br)^\op$ as an equivalence of $
\infty$-categories. 

Lastly, Theorem~\ref{introloc} and Corollary~\ref{introlocexit} induce fully faithful functors 
from space-valued functors on $\Exit\bl(\Ranu(\R^n)\br)$ 
and $\Exit\bigl(\Ran(\R^n)\bigr)$ respectively, 
to space-valued functors out of the appropriate subcatgories of $\mathbf{\Theta}_n$. 
These functors fit into the following diagram as indicated. 

\begin{equation} \label{ff3} \begin{tikzcd}[row sep=28, column sep=28]
\Alg_{\E_n}\left(\Spaces\right) \ar[dashed]{rr}{\fB^n}[swap]{\tx{fully faithful}} 
\ar[dashed]{d}[swap]{\simeq} 
&&\Cat_{\left(\infty,n\right)}\left(\Spaces\right)^{\uno/}  \ar[hookrightarrow]{d}[swap]{\tx{forget}} 
\ar[hookrightarrow]{d}{\tx{(fully faithful, by def)}} \\
\Fun\left(\Exit\bl(\Ranu(\R^n)\br)^\op, \Spaces \right) 
\ar{d}[swap]{\tx{forget}} 
\ar[hookrightarrow]{rr}{\tx{Thm.~\ref{introloc}}} 
\ar[hookrightarrow]{rr}[swap]{\tx{fully faithful}} 
&& \Fun\left(\left(\mathbf{\Theta}_n^\Act\right)^\op, \Spaces\right)  
\ar[hookrightarrow]{d}{\tx{forget}} \\ 
\Fun\left(\Exit\bigl(\Ran(\R^n)\bigr)^\op, \Spaces \right) 
\ar{d}[swap]{\simeq} \ar[hookrightarrow]{rr}{\tx{Cor.~\ref{introlocexit}}} 
\ar[hookrightarrow]{rr}[swap]{\tx{fully faithful}} && 
\Fun\left(\left(\mathbf{\Theta}_n^\exit\right)^\op, \Spaces\right) \\
\sf{coShv_{\Spaces}}\left(\Ran\left(\R^n\right)\right) && 
\end{tikzcd}
\end{equation}

Our paper is also motivated by Dwyer-Hess-Knudsen's decomposition 
of the configuration space of a product of parallelizable manifolds into the 
factors according to the Boardman-Vogt tensor product \cite{DHK}.
Indeed, Corollary~\ref{introcor 2} is a closely related decomposition and
we conjecture two natural generalizations of it. 

\begin{conj} 
For $n \geq 2$, there is a localization of $\infty$-categories
$$
\Exit\left(\Ranu\left(\R\right)\right) \wr \Exit\left(\Ranu\left(\R^{n-1}\right)\right) 
\lra \Exit\bl(\Ranu(\R^n)\br).
$$
\end{conj}

\begin{conj} 
For connected, parallelizable manifolds $M$ and $N$ there is a localization of $\infty$-categories
$$
\Exit\bl(\Ranu(M)\br) \wr \Exit\left(\Ranu\left(N\right)\right) 
\lra \Exit\left(\Ranu\left(M \times N\right)\right).
$$
\end{conj}

\begin{remark}
The $\infty$-category $\Exit\bigl(\Ran(\R^n)\bigr)$ does not admit a wreath product decomposition 
because the subcategory $\mathbf{\Theta}_n^\exit$ does not admit a wreath product decomposition.
\end{remark}

\subsubsection*{Use of \texorpdfstring{$\infty$}{Lg}-categories} 
In this paper, we use $\left(\infty, 1\right)$-categories, or simply, `$\infty$-categories', 
to package the homotopy type of configuration spaces. 
There are many models for $\infty$-categories; 
notable for the scope of this paper is the work of Lurie in \cite{Lu1} on the theory of 
\emph{quasi-categories} and the work of Rezk in \cite{Rezk2} on the theory of 
\emph{complete Segal spaces}. 

\textbf{Quasi-categories:} 
We use Joyal's quasi-category model of $\infty$-categories from \cite{Joy}, 
wherein a quasi-category is defined to be a \emph{simplicial set}, 
that is, a functor from the opposite of the simplex category $\mathbf{\Delta}^\op$ 
(Def.~\ref{simplex}) 
to the category of sets, $\tx{Set}$, 
 that satisfies a certain condition called the \emph{inner-horn filling condition}. 
For the definition of this condition, together with other basic notions regarding 
quasi-catgeories see Rezk's friendly exposition \cite{Rezk3}.

\textbf{Complete Segal spaces:} 
We use Rezk's complete Segal spaces to model $\infty$-categories \cite{Rezk2}. 
The $\infty$-category of spaces $\spaces$ is the localization 
of the category of topological spaces that admit a CW structure 
and continuous maps thereof, localized on (weak) homotopy equivalences. 
We call its objects \emph{spaces}, i.e., CW complexes. 
A complete Segal space is a \emph{simplicial space}, that is, 
a functor from the opposite of the simplex category $\mathbf{\Delta}^\op$ 
to the $\infty$-category of spaces $\spaces$ that satisfies the \emph{completeness} 
and \emph{Segal} conditions (Def.~\ref{cSegspc}).
We refer to a point in the $[p]$-space of such a functor as a $[p]$-point. 
For any additional information about complete Segal spaces, see \cite{Rezk2}.

\textbf{The nerve functor:} 
There is a construction which takes an ordinary category $\C$ 
and produces an $\infty$-category $\Ner \C$ called the (ordinary) \emph{nerve} of $\C$. 
This construction is explicated by a fully faithful functor from the category of categories 
to the category of simplicial sets, through which each category is carried to a quasi-category. 
In light of the fully faithfulness of this functor, 
we refer to an ordinary category as  an $\infty$-category 
without any reference to its nerve, whenever appropriate within the context.  
For a definition of the nerve, see Definition~3.1 in \cite{Rezk3}.

\textbf{Model independence:} 
In this paper, we work model independently, which, by the work of Joyal and Tierney, 
is a valid approach, 
since quasi-categories are shown to be equivalent to complete Segal spaces \cite{JT}. 
Model independence is exercised in this paper, for example, 
in that the hom-$\infty$-groupoid with fixed source and target of 
a quasi-category is equivalent to a space (i.e., CW complex) 
by way of the equivalence between quasi-categories and complete Segal spaces. 
Throughout this work, we are liberal with our use of model independence 
and typically do not give forewarning of its implementation. 
Additionally, we frequently construct $\infty$-categories by taking finite limits 
of $\infty$-categories, an allowable maneuver precisely because 
the category of complete Segal spaces (or any other equivalent model) 
admits all finite limits and colimits \cite{RV}.

\subsubsection*{Linear overview} 

\S\ref{unital Ran}:
We introduce the main character of this work, the $\infty$-category $\Exit\bigr(\Ranu(M)\bigr)$, 
to encode the configuration spaces of finite, possibly empty subsets of $M$. 
After defining it as a simplicial space, we prove that it is a complete Segal space 
by identifying it as one derived through formal constructions among complete Segal spaces 
from the $\infty$-category of constructible bundles from \cite{AFR}.
Along the way we define the $\infty$-subcategory $\Exit\bigl(\Ran(M)\bigr)$ 
and show that consequently, it, too, is a complete Segal space.
We conclude this section by restating our main result in its full generality as Theorem~\ref{loc}, 
which states that $\mathbf{\Theta}_n^\Act$ localizes to $\Exit\bigl(\Ranu(\R^n)\bigr)$ 
over the (opposite) category of finite sets. 

\S\ref{step 1}:
We introduce the $\infty$-category $\Exit\bigl(\Ranu(\un{\R}^n)\bigr)$ to interpolate 
between $\mathbf{\Theta}_n^\Act$ and $\Exit\bigl(\Ranu(\R^n)\bigr)$.
It encodes the configuration spaces of finite, possibly empty subsets of $\R^n$ 
each stratified by its Fox-Neuwirth cell decomposition. 
We show that it is a complete Segal space using techniques similar to those in \S\ref{unital Ran}. 

\S\ref{ah}:
This section is the first step toward proving Theorem~\ref{loc}.
We prove that the $\infty$-category $\Exit\bigl(\Ranu(\un{\R}^n)\bigr)$
is equivalent to the category $\mathbf{\Theta}_n^\Act$ over the (opposite) category of finite sets. 
We construct the functor inductively and
show that it is essentially surjective and fully faithful by induction. 

\S\ref{step2}:
This section is the second and final step towards the proof of Theorem~\ref{loc} 
and is the technical heart of our result. 
We show that the natural forgetful functor from $\Exit\bigl(\Ranu(\un{\R}^n)\bigr)$ 
to $\Exit\bigl(\Ranu(\R^n)\bigr)$ is a localization over the (opposite) category of finite sets. 
It localizes on all those morphisms that induce bijections between their underlying finite sets, 
namely, all those morphisms that get sent to isomorphism in $\Exit\bigl(\Ranu(\R^n)\bigr)$.
Our argument is built around Theorem~\ref{BcSS} from \cite{AMG}, 
which identifies localizations of $\infty$-categories in favorable cases. 
We implement Theorem~\ref{BcSS} by proving two lemmas, 
the first of which holds the technical heart of the paper. 
Central to the first lemma is that the unordered configuration space 
$\Conf_r(\R^n)_{\Sigma_r}$ stratified by its
Fox-Neuwirth cell decomposition is conically smooth (Cor. \ref{conically smooth}). 
This result allows us to apply Ayala-Francis-Tanaka-Rozenblyum's 
exit-path $\infty$-category functor from \cite{AFT, AFR} to $\Conf_r(\un{\R}^n)_{\Sigma_r}$.
The main advantage in employing their exit-path functor is that we gain hands-on access 
to the background stratified space, namely, the configuration space $\Conf_r(\R^n)_{\Sigma_r}$, 
which allows us to run some key homotopy equivalences.
Note that we do not have this luxury with our main two exit-path 
$\infty$-categories $\Exit\bigl(\Ranu(M)\bigr)$ 
and $\Exit\bigl(\Ranu(\un{\R}^n)\bigr)$ as they technically 
do not have stratified spaces underlying them. 

\S\ref{bigcor}:
This section is devoted proving Corollary~\ref{introlocexit}
which is the nonunital version of Theorem~\ref{loc}. 
Our method of proof mainly extrapolates the proof of Theorem~\ref{loc}, 
in that it too is built from Theorem~\ref{BcSS} of \cite{AMG}.

\S\ref{stratapp}:
We review the notions of topologically and Whitney stratified spaces. 
We develop two key examples of stratified spaces. 
The first is the configuration space 
of points in $\R^n$ labeled by a set $A$ stratified by its Fox-Neuwirth cell decomposition, 
which we eventually show is Whitney stratified.
The second is the Ran space of a smooth, nonempty, connected manifold $M$ stratified by cardinality. 
We also consider the unital Ran space of $M$ and note why it is not a stratified space. 

\S\ref{exit sec}:
This section is a review of exit-path $\infty$-category constructions notable to this work, 
namely Lurie's \cite{Lu1} and Ayala-Francis-Tanaka-Rozenblyum's \cite{AFR, AFT}. 
Our main example is the latter's exit-path construction applied to 
the unordered configuration space of $r$ points in $\R^n$ stratified by its Fox-Neuwirth cells.
Key here is that Whitney stratified spaces are conically smooth \cite{NV}.
Thus, the Whitney stratified space $\Conf_r(\un{\R}^n)_{\Sigma_r}$ is conically smooth, 
and is, therefore, input for Ayala-Francis-Tanaka-Rozenblyum's exit-path functor.
We conclude this section with an explanation for why we use neither
exit-path construction to define the three main exit-path $\infty$-categories of this work. 

\S\ref{Theta sec}:
We recall the wreath product of categories and Segal's gamma functor
in order to define the category $\mathbf{\Theta}_n$ 
as the n-fold wreath product of the simplex category $\mathbf{\Delta}$ with itself.
We also review an important class of morphisms in $\mathbf{\Theta}_n$, namely active morphisms,
which are characterized by their trivial interaction with the basepoint in pointed finite sets.

\subsubsection*{Acknowledgements}
I would like to thank my PhD advisor David Ayala for sharing his time, energy and ideas for this project, 
contributing significantly to my fantastic time as a graduate student in Bozeman, Montana.
Additionally, I would like to thank Eric Berry, Greg Friedman, Ben Knudsen, 
Damien Lejay, Peter May, Ben Moldstad, Dan Perry, and Dev Sinha for many helpful conversations
and unmeasurable support.

\section{Exit-paths in the unital Ran space} \label{unital Ran}
Unless otherwise stated, let $M$ denote a connected, nonempty, smooth manifold 
for the entirety of the paper. 
In this section, we construct two $\infty$-categories to codify configurations of finite subsets of $M$.
The first encodes the configuration spaces of unordered points in $M$ of all finite cardinalities, 
including cardinality zero.
Heuristically, its objects are finite, possibly empty subsets $S$ of $M$, 
and its morphisms witness anticollision and vanishing of points in $M$, but not collision.
Before defining this $\infty$-category, we need the following construction.

\begin{definition}[Def.~6.6.12, \cite{AFR}]  \label{rev} 
The \emph{reversed cylinder} of a map between finite sets $T \to S$ is 
$$\ds {\sf cylr}\left(T \to S\right) := S \coprod_{T \times \{0\}} \left(T  \times \Delta^1\right).$$
More generally, the reversed cylinder of a composable sequence of maps 
between finite sets \\ $S_p \to S_{p-1} \to \cdots \to S_0$ is
$$ \ds  \cylr\left(S_p \to S_{p-1} \to \cdots \to S_0\right) := 
S_0 \coprod_{S_1 \times \{0\}} S_1 \times \Delta^1 \coprod_{S_2 \times \Delta^1} ...
 \coprod_{S_p \times \Delta^{p-1}} \left( S_p \times \Delta^p\right) .$$
\end{definition}

\begin{example} \label{cylr ex} 
Consider the map of sets $T=\{t, t', t''\} \ra S=\{s, s', s''\}$ 
given by $t, t' \mapsto s$ and $t'' \mapsto s'$. 
The coproduct $ \ds S \coprod_{T \times \{0\}}\left(T \times \Delta^1\right) $ 
is depicted by Figure~\ref{cylr pic},
along with its natural projection to $\Delta^1$.

\begin{figure}[ht] 
\begin{center}
\begin{tikzpicture}[scale=.5] 
\draw[->] (0,-1) -- (0, -3);
\filldraw[black]  (-2,2) circle (4pt)  node[anchor=east] {s};
\filldraw[black]  (-2,1) circle (4pt)  node[anchor=east]  {s'};
\filldraw[black]  (-2,0) circle (4pt)  node[anchor=east]  {s''};
\filldraw[black]  (-2,-4) circle (4pt)  node[anchor=east]  {0};
\filldraw[black]  (2,2) circle (4pt)  node[anchor=west]  {t};
\filldraw[black]  (2,1) circle (4pt)  node[anchor=west]  {t'};
\filldraw[black]  (2,0) circle (4pt)  node[anchor=west]  {t''};
\filldraw[black]  (2,-4) circle (4pt)  node[anchor=west]  {1};

\draw
(-2,2) -- (2,2)
(-2,2) -- (2,1)
(-2,1) -- (2,0)
(-2,-4) -- (2,-4);

\end{tikzpicture}
\caption{}
\label{cylr pic}
\end{center}
\end{figure}
\end{example}

\begin{observation} \label{preshff} 
Recall that by definition, the nerve of $\Fin^\op$ is a presheaf on $\mathbf{\Delta}$. 
 Thus, we may define the $\infty$-category $\mathbf{\Delta}$ slice over $\Fin^\op$ 
 as the pullback of $\infty$-categories
\[ \begin{tikzcd}
\mathbf{\Delta}_{/\Fin^\op} \pb \ar{r} \ar{d} & 
{\sf PShv}\left(\mathbf{\Delta}\right)_{/ \Fin^\op} \ar{d} \\
\mathbf{\Delta} \ar{r}{\tx{Yoneda}} & \sf{PShv}\left(\mathbf{\Delta}\right) \end{tikzcd} \] 
where ${\sf PShv}\left(\mathbf{\Delta}\right)$ is the category of presheaves 
on the simplex category $\mathbf{\Delta}$ (Def. \ref{simplex})
and $\Fin^\op$ is the opposite of the category of finite sets (Def. \ref{Fin}). 
 \end{observation}

\begin{definition}\label{exitdef2} 
For a smooth, connected manifold $M$, 
the \textit{exit-path $\infty$-category of the unital Ran space of $M$}, 
$\Exit\bl(\Ranu(M)\br)$, is the simplicial space over $\Fin^\op$ 
representing the presheaf on $\mathbf{\Delta}_{/\Fin^\op}$ whose value on an object 
$$ [p] \xra{<\sigma>} \Fin^\op   $$ 
which selects out a sequence of maps among finite sets 
$\ds \sigma: S_p \to \cdots \to S_0$, is the space of embeddings 
$$\cylr\left(\sigma\right) \hookrightarrow M \times \Delta^p$$ 
over $\Delta^p$ equipped with the compact-open topology; 
the structure maps are evident. \end{definition}

\begin{observation} 
Explicitly, an object of $\Exit\bl(\Ranu(M)\br)$ in the fiber 
over the finite (possibly empty) set $S$ is an embedding, 
$$S\hookrightarrow M.$$  
A morphism from $S \xhookrightarrow{e} M$ to $T \xhra{d} M$ 
over the map of finite sets $T \xra{\sigma} S$ is an embedding,
$$ \cylr\left(T \xra{\sigma} S\right) \xhookrightarrow{E} M \times \Delta^1 $$
over $\Delta^1$ such that $E_{|S}=e$ and $E_{|T\times {1}}= d$. 
\end{observation}

\begin{example}
Any embedding of Figure~\ref{cylr pic} into $M \times \Delta^1$ 
over $\Delta^1$ is a morphism in $\Exit\bl(\Ranu(M)\br)$. 
Such a morphism is a degeneracy morphism 
since the point $s''$ vanishes. 
Also note that the point $s$ anticollides into two points. 
\end{example}

\begin{remark} 
Definition~\ref{exitdef2} can be viewed as a generalization of 
Lurie's exit-path $\infty$-category (Def. \ref{exitdef}) 
of the Ran space of $M$ (Def. \ref{Ran}). 
In particular, it encodes the empty configuration in a nontrivial way, 
namely the space of morphisms from any configuration to the empty 
configuration is always nonempty. 
Definition~\ref{exitdef} cannot be directly extended to 
accommodate this essential feature
which presented one of the main technical difficulties of this work;
see \S\ref{nottrivial} for more details. 
\end{remark}

\begin{remark}
Ayala-Francis-Tanaka-Rozenblyum's exit-path $\infty$-category construction (Def. \ref{AFRexit})
cannot be used to encode $\Exit\bl(\Ranu(M)\br)$ because the 
unital Ran space of $M$ is not a stratified space; 
see \S\ref{unitalRan} and \S\ref{nottrivial} for more details. 
\end{remark}

\begin{observation} \label{exittofin} 
There is a natural forgetful functor
$$\Exit\bl(\Ranu(M)\br) \xra{\phi} \Fin^\op$$ 
the value of which on an object $\ds S \hra M $ is $S$ 
and on a morphism $\ds \cylr\left(J \xra{\sigma} S\right) \hra M \times \Delta^1$ is $\sigma$. 
\end{observation} 

By removing the empty set from $\Exit\bl(\Ranu(M)\br)$ 
and all morphisms that involve the vanishing of points, 
we obtain the other $\infty$-category introduced in this section,
which is defined as follows. 

\begin{definition} \label{little exit} 
The \textit{exit-path $\infty$-category of the Ran space of $M$} 
denoted $\Exit\bigl(\Ran(M)\bigr)$
is the subsimplicial space of $\Exit\bigl(\Ranu(M)\bigr)$, 
obtained as the following pullback of simplicial spaces
\begin{equation} \label{nonunital} 
\begin{tikzcd}
\Exit\bigl(\Ran(M)\bigr) \ar[hookrightarrow]{r} \pb \ar{d} & 
\Exit\bl(\Ranu(M)\br)\ar{d}{\phi} \\
\left(\Fin_{\neq \emptyset}^\surj\right)^\op \ar[hookrightarrow]{r} &  \Fin^\op
\end{tikzcd} 
\end{equation}
where $\Fin^\surj_{\neq \emptyset}$ is the subcategory of finite sets 
consisting of nonempty sets and all those morphisms that are surjections. 
\end{definition}

Informally, an object of $\Exit\bigl(\Ran(M)\bigr)$ is a finite, nonempty subset of $M$ 
and a morphism is a path in $\Ran\left(M\right)$ (Def. \ref{Ran}) 
that may witnesses anticollision of points in $M$, but not collision. 

\begin{remark}
Lurie's exit-path $\infty$-category construction (Def. \ref{exitdef}) 
applied to the Ran space (Def. \ref{Ran}) of $M$ is equivalent to Definition~\ref{little exit}. 
\end{remark}

\subsection{\texorpdfstring{$\Exit\bigl(\Ranu(M)\bigr)$}{Lg}
is an \texorpdfstring{$\infty$}{Lg}-category} \label{unital} 
This subsection is devoted to proving the technical result
that the exit-path $\infty$-category of the unital Ran space 
of a connected smooth manifold $M$, $\Exit\bl(\Ranu(M)\br)$, 
and consequently the exit-path $\infty$-category of the Ran space of $M$, 
$\Exit\bigl(\Ran(M)\bigr)$, is in fact an $\infty$-category.
In particular we show that it is a complete Segal space.
We allow ourselves to freely use notation and results from~\cite{AFR} and 
begin with the definition of a complete Segal space after Rezk in \cite{Rezk2}. 

\begin{definition} \label{cSegspc} 
A simplicial space $\mathbf{\Delta}^\op \xra{F} \spaces$ 
is a \emph{complete Segal space} if it satisfies the following two conditions:
\begin{enumerate}
\item  (Segal Condition) For each $p >1$, the diagram of spaces 
\[ \begin{tikzcd} 
F[p] \ar{r} \pb \ar{d} & F\{p-1<p\} \ar{d} \\
F\{0<\cdots < p-1\} \ar{r} & F\{p-1\} \end{tikzcd} \]
is a pullback.

\item  (Completeness Condition) The diagram of spaces
\[ \begin{tikzcd}
&  & F\left(\ast\right) \ar{dll} \ar{d} \ar{drr} & & \\
F\left(\ast\right) \ar{dr} & & F[3] \ar{dl} \ar{dr} & & F\left(\ast\right) \ar{dl} \\
& F\{0< 2\} & &  F\{1 <3\} & 
\end{tikzcd} \]
is a limit. 
\end{enumerate}
\end{definition}

\begin{prop}\label{t1}
The simplicial space $\Exit\bl(\Ranu(M)\br)$ 
satisfies the Segal and completeness conditions.	
\end{prop}

\begin{cor} \label{wow}
The simplicial space $\Exit\bigl(\Ran(M)\bigr)$ 
satisfies the Segal and completeness conditions.
\end{cor}

\begin{proof}
By Definition~\ref{little exit}, $\Exit\bigl(\Ran(M)\bigr)$ is the pullback of the diagram
\[ \begin{tikzcd}
\Exit\bigl(\Ran(M)\bigr) \ar[hookrightarrow]{r} \pb \ar{d} & \Exit\bl(\Ranu(M)\br)\ar{d} \\
\left(\Fin_{\neq \emptyset}^\surj\right)^\op \ar[hookrightarrow]{r} &  \Fin^\op
\end{tikzcd}\]
of simplicial spaces.
The result follows because the full $\infty$-subcategory 
of simplicial spaces consisting of the complete Segal spaces 
is closed under the formation of pullbacks.  
\end{proof}

The idea for the proof of Proposition~\ref{t1} is to witness the simplicial space 
$\Exit\bl(\Ranu(M)\br)$ as one derived through formal constructions 
among complete Segal spaces from the complete Segal space of 
\emph{constructible bundles} $\Bun$ (Def. 6.3.6, \cite{AFR}).
$\Bun$ is a complete Segal space 
whose value on $[p]$ is the moduli space of 
constructible bundles over $\Delta^p$ with the 
standard stratification (Ex. \ref{standard}).  
The simplicial structure maps are implemented by base change of constructible bundles.  

The space of objects in $\Bun$ is the moduli space of constructible bundles over 
$\Delta^0=\ast$, which, in fact, is the moduli space of stratified spaces.
In other words, an object of $\Bun$ is a stratified space.
Thus, a finite set is an example of an object in $\Bun$, 
and a smooth manifold is an example of an object in $\Bun$ as well.  
Lemma~6.3.11 of~\cite{AFR} uses the reverse cylinder 
construction (Def.~\ref{rev}) to construct a fully faithful functor
\[
\Fin_\ast^{\op}
\longrightarrow
\Bun
\]
whose image consists of finite sets.  
In particular, there is a composite monomorphism 
\begin{equation}\label{e1}
\Fin^{\op}
\xra{~\left(-\right)_+~}
\Fin_\ast^{\op}
\longrightarrow
\Bun
\end{equation}
of $\infty$-categories.
Note that a $[p]$-point $X\to \Delta^p$ of $\Bun$ factors through (\ref{e1}) 
if and only if $X\to \Delta^p$ is a finite proper constructible bundle.  

Lemma~3.31 of~\cite{AFR2} constructs the \emph{$k$-skeleton} functor
\[
\sk_k\colon \Bun
\longrightarrow
\Bun
\]
for each dimension $k$.
Explicitly, the value on a stratified space $X$ is the proper constructible 
stratified subspace $\sk_k\left(X\right) \subset X$ that is the union of the strata whose dimension is at most $k$.
The value of $\sk_k$ on a $[p]$-point $X\to \Delta^p$ of $\Bun$ is the constructible bundle 
\[
\sk_k^{\sf fib}\left(X\right)
\longrightarrow 
\Delta^p
\]
which is the \emph{fiberwise $k$-skeleton} of $X$: 
the union of those strata of $X$ whose projection to $\Delta^p$ have fiber-dimension at most $k$.

Note that $\sk_0$ factors through $\Fin_\ast$:
\[
\sk_0\colon \Bun
\longrightarrow
\Fin_\ast^{\op}
\hookrightarrow
\Bun.
\]

Consider the $\infty$-category $\cS{\sf trat}^{\sf ref}$ underlying 
the topological category in which an object is a stratified space 
and the space of morphisms is that of \emph{refinements}.

\begin{definition}[Def. 3.6.1, \cite{AFT}] \label{refinement} 
A map of stratified spaces $\left(X \to P\right)  \xra{f} \left(Y \to Q\right)$ is a \emph{refinement} 
if $f$ is a homeomorphism between the underlying topological spaces, 
and if for each $p \in P$ the restriction of $f$ to each stratum $X_p$  is an embedding into $Y$. 
\end{definition}
  
Section~\S6.6 of~\cite{AFR} constructs the \emph{open cylinder} functor between $\infty$-categories
\begin{equation} \label{e}
{\sf Cylo}\colon \cS{\sf trat}^{\sf ref}
\longrightarrow
\Bun
\end{equation}
which is an equivalence on spaces of objects.  
Theorem~6.6.15 of~\cite{AFR} verifies that this functor is a monomorphism.  
So each refinement between stratified spaces defines a morphism in $\Bun$.
As a matter of notation, a morphism in $\Bun$ 
that is in the image of this functor is called a \emph{refinement}; 
the $\infty$-category of \emph{refinement arrows} in $\Bun$ is the full $\infty$-subcategory
\begin{equation} \label{ref arrows}
\Ar^{\sf ref}\left(\Bun\right)
~\subset~
\Ar\left(\Bun\right)
\end{equation}
consisting of the refinements arrows.  
Evaluation at source-target defines a functor
\[
\left({\sf ev}_s,{\sf ev}_t\right)\colon
\Ar^{\sf ref}\left(\Bun\right)
\longrightarrow
\Bun\times \Bun .
\]
Denote the pullback $\infty$-category:
\begin{equation}  \begin{tikzcd} \label{Ref}
\RRef^0(M) \ar{rr} \ar{dd}   & & \Ar^{\sf ref}\left(\Bun\right) \ar{d}{\left(\ev_s, \ev_t\right)} \\
&& \Bun \times \Bun \ar{d}{\sk_{n-1} \times \Id} \\
\Fin^\op \ar{r}{=} & \Fin^\op \times \ast \ar{r}{\left(\ref{e1}\right) \times \langle M \rangle} & \Bun \times \Bun.
\end{tikzcd} 
\end{equation}
Unpacking this definition and using the \emph{open cylinder} construction
(\ref{e}), we see that $\RRef^0(M)$ is a simplicial space whose value 
on $[p]\in \mathbf{\Delta}$ is the moduli space of 
\begin{itemize}
\item
constructible bundles
\[
X
\longrightarrow
\Delta^p
\]
for which the $\left(n-1\right)$-skeleton of each fiber of which is a finite set,
\item
together with a refinement 
\[
X
\xra{~\rm refinement~}
M\times \Delta^p
\]
over $\Delta^p$.
\end{itemize}

We will denote such a $[p]$-point of $\RRef^0(M)$ 
simply as $\left(X \xra{\sf ref} M\times \Delta^p\right)$. 
Example~2.1.7 of \cite{AFT} shows that the product of stratified spaces 
is naturally a stratified space.
We consider $M \times \Delta^p$ as a product stratified space, 
where $M$ is trivially stratified over the poset with a singleton, 
and $\Delta^p$ is given the standard stratification (Ex.~\ref{standard}).

Informally, an object in $\RRef^0(M)$ is a refinement of $M$ 
in which the $\left(n-1\right)$-skeleton of the domain is a finite set, 
and a morphism in $\RRef^0(M)$ is a path of such refinements of $M$ 
witnessing anticollision and disappearances of strata.  

\begin{observation} 
There is a natural forgetful functor 
\[ \RRef^0(M) \lra \Fin^\op
\]
to the opposite of the category of finite sets, the value of which on an object 
$X \xra{{\sf ref}} M$ is the underlying set of the $\left(n-1\right)$-skeleton of $X$, 
and on a morphism $X \xra{{\sf ref}} M \times \Delta^1$ 
is the canonical assignment between sets, 
from the $\left(n-1\right)$-skeleton of the fiber of $X \to \Delta^1$ over $\{1\} \in \Delta^1$ 
to the $\left(n-1\right)$-skeleton of the fiber over $\{0\} \in \Delta^1$, 
implemented by taking connected components of the $\left(n-1\right)$-skeleton of $X$. 
In other words, taking connected components of the 
fiberwise $\left(n-1\right)$-skeleton of $X$ induces a canonical assignment of sets
\[ \sk_{n-1}^{{\fib}}\left(X_{|1}\right) \lra \sk_{n-1}^{\fib}\left(X_{|0}\right) 
\]
from the $\left(n-1\right)$-skeleton of the target of $X$ to the $\left(n-1\right)$-skeleton of the source of $X$.
\end{observation}

\begin{lemma}\label{l1}
There is a canonical equivalence between simplicial spaces
\[
\RRef^0(M)
{~\simeq~}
\Exit\bl(\Ranu(M)\br)
\]
over $\Fin^\op$.
\end{lemma}

\begin{proof}
A rightward morphism is implemented by, for each $[p]\in \mathbf{\Delta}$, the assignment,
\[
\left(X\xra{\sf ref} M\times \Delta^p \right)
\mapsto
\left( \sk^{\sf fib}_{n-1}\left(X\right) \hookrightarrow X \to  M\times \Delta^p \right)
~,
\]
whose value is the embedding over $\Delta^p$ from the fiberwise $\left(n-1\right)$-skeleton, 
which maps to $\Delta^p$ as a finite proper constructible bundle.
A leftward morphism is implemented by, for each $[p]\in \mathbf{\Delta}$, the assignment,
\[
\left( {\sf Cylr}\left(\sigma\right) \hookrightarrow M\times \Delta^p \right)
\mapsto 
\left( \left( {\sf Cylr}\left(\sigma\right) \subset M\times \Delta^p \right)
\xra{~\sf ref~}
M\times \Delta^p \right)
~,
\]
whose value is the coarsest refinement of $M\times \Delta^p$ 
for which the embedding from ${\sf Cylr}\left(\sigma\right)$ is a proper and constructible.  
(Such a refinement exists because the image of this embedding is, 
by definition, a properly embedded stratified subspace.)

It is straightforward to verify that these two assignments are mutually inverse to one another, 
and further, that they are both over $\Fin^\op$. 
Furthermore, it is evident that the structure maps are equivalent. 
\end{proof}

\begin{proof}[Proof of Proposition~\ref{t1}]
Being an $\infty$-category, the simplicial space $\RRef^0(M)$ 
satisfies the Segal and completeness conditions.
Through the equivalence of Lemma~\ref{l1}, 
then so does the simplicial space $\Exit\bl(\Ranu(M)\br)$.
\end{proof}

\subsection{Statement of main result}
We restate our main result in its full generality,
which is a combinatorial identification of the $\infty$-category 
$\Exit\bl(\Ranu(M)\br)$ in the case of $M=\R^n$
in terms of the subcategory of 
$\mathbf{\Theta}_n$ consisting of active morphisms (Def. \ref{act}). 

\begin{theorem} \label{loc}  
For $n \geq 1$ there is a localization of $\infty$-categories
$$ \mathbf{\Theta}_n^\Act \to \Exit\bl(\Ranu(\R^n)\br)$$
which is contravariant over the category of finite sets. 
\end{theorem}

\section{Exit-paths in the Fox-Neuwirth unital Ran space of \texorpdfstring{$\R^n$}{Lg}} \label{step 1}
The goal of this section is to construct an $\infty$-category
to interpolate between the category $\mathbf{\Theta}_n^\Act$ 
and the $\infty$-category $\Exit\bl(\Ranu(\R^n)\br)$,
the source and target, respectively, of Theorem~\ref{loc}.
This $\infty$-category will
encode all of the configuration spaces of $\R^n$ (including the empty one)
each stratified by their Fox-Neuwirth cell decompositions (\S\ref{FNRanu}).

\begin{observation} 
Let $\Fun\left(\{1<\cdots<n\}, \Fin^\op\right)$ denote the category of functors  
from the linearly ordered set $\{1<\cdots<n\}$ regarded as a category (Term. \ref{poset})
to the opposite of the category of finite sets $\Fin^\op$ (Def. \ref{Fin}).
Recall that by definition the nerve of 
$\Fun\left(\{1<\cdots<n\}, \Fin^\op\right)$
 is a presheaf on $\mathbf{\Delta}$. 
 Thus, we may define the $\infty$-category 
 $\mathbf{\Delta}$ slice over $\Fun\left(\{1<\cdots<n\}, \Fin^\op\right)$ 
 as the following pullback
\[ \begin{tikzcd}
\mathbf{\Delta}_{/ \Fun\left(\{1<\cdots<n\}, \Fin^\op\right)} \pb \ar{r} \ar{d} 
& {\sf PShv}\left(\mathbf{\Delta}\right)_{/ \Fun\left(\{1<\cdots<n\}, \Fin^\op\right)} \ar{d} \\
\mathbf{\Delta} \ar{r}{\tx{Yoneda}} & \sf{PShv}\left(\mathbf{\Delta}\right) 
\end{tikzcd} \] 
where ${\sf PShv}\left(\mathbf{\Delta}\right)$ is the category of presheaves on the simplex category.
 \end{observation}

\begin{definition} \label{underline} 
The \textit{exit-path $\infty$-category of the Fox-Neuwirth unital Ran space of $\R^n$} denoted
$\Exit\bl(\Ranu(\un{\R}^n)\br)$  is the simplicial space over  
$\ds \Fun\left(\{1<\cdots<n\}, \Fin^\op\right)$ representing the presheaf on 
$ \ds \mathbf{\Delta}_{/ \Fun\left(\{1<\cdots<n\}, \Fin^\op\right)}$ whose value on an object 
$$ [p] \to  \Fun\left(\{1<\cdots<n\}, \Fin^\op\right)   $$ 
which selects a diagram of finite sets
\begin{equation} \label{ppointover} 
\begin{tikzcd} 
 \sigma_n:  &[-30pt]  S^p_n \arrow{r} \ar{d}  &  \cdots  \arrow{r} & S^0_n \ar{d} \\
& \vdots \ar{d} &  \vdots & \vdots \ar{d} \\
\sigma_1: & S^p_1 \arrow{r}   &  \cdots  \arrow{r} & S^0_1
\end{tikzcd} 
\end{equation}
is the space of compatible embeddings 
\begin{equation} \label{ppoint} \ds
\begin{tikzcd}
{\sf cylr}\left(\sigma_n\right) \ar[hookrightarrow]{r}{E_n} \ar{d}  
& \R^n \times \Delta^p \ar[twoheadrightarrow]{d}{\tx{pr}_{< n} \times \tx{id}_{\Delta^p}} \\
{\sf cylr}\left(\sigma_{n-1}\right) \ar[hookrightarrow]{r}{E_{n-1}} \ar{d}  
& \R^{n-1} \times \Delta^p \ar[twoheadrightarrow]{d}{\tx{pr}_{< n-1} \times \tx{id}_{\Delta^p}} \\
\vdots \ar{d} &  \vdots \ar[twoheadrightarrow]{d} \\
{\sf cylr}\left(\sigma_1\right) \ar[hookrightarrow]{r}{E_1}   & \R \times \Delta^p
\end{tikzcd}
\end{equation} 
where each embedding is over $\ds \Delta^p$ 
and the downward arrows on the lefthand side 
are induced by the downward arrows of (\ref{ppointover}). 
This embedding space is given the compact-open topology; 
the structure maps are evident. 

Observation~\ref{phii} defines a canonical functor from 
$\Exit\bl(\Ranu(\un{\R}^n)\br)$ to  $\ds \Fun\left(\{1<\cdots<n\}, \Fin^\op\right)$.   
\end{definition}

\begin{observation} 
Explicitly, an object of $\Exit\bl(\Ranu(\un{\R}^n)\br)$ 
over the sequence  of finite sets,
\begin{equation} \label{ss} 
S_n \xra{\tau_{n-1}} S_{n-1} \xra{\tau_{n-2}}  \cdots \ra S_1 
\end{equation}
 is a sequence of embeddings,
\begin{equation} \label{objRn} \ds
\begin{tikzcd}
S_n \ar[hookrightarrow]{r}{e_n} \ar{d}[swap]{\tau_{n-1}}  
& \R^n  \ar[twoheadrightarrow]{d}{\tx{pr}_{< n}}  \\
S_{n-1} \ar[hookrightarrow]{r}{e_{n-1}} \ar{d}[swap]{\tau_{n-2}}  
& \R^{n-1}  \ar[twoheadrightarrow]{d}{\tx{pr}_{< n-1} } \\
\vdots \ar{d} &  \vdots \ar[twoheadrightarrow]{d} \\
S_1 \ar[hookrightarrow]{r}{e_1}  & \R .
\end{tikzcd}
\end{equation}
\noindent When the context is clear, we denote (\ref{objRn}) 
by $\un{S} \xhra{\un{e}} \un{\R}^n$ or just $\un{e}$. 

A morphism from $\ds \un{S} \xhra{\un{e}} \un{\R}^n$ to 
$\ds \un{T} \xhra{\un{d}} \un{\R}^n$ over the diagram of finite sets,
\begin{equation} \label{seqofmaps} \begin{tikzcd} 
 T_n \ar{d}[swap]{\omega_{n-1}} \ar{r}{\sigma_n} & S_n \ar{d}{\tau_{n-1}} \\
 T_{n-1} \ar{d}[swap]{\omega_{n-2}} \ar{r}{\sigma_{n-1}} & S_{n-1} \ar{d}{\tau_{n-2}} \\ 
\vdots \ar{d} & \vdots \ar{d} \\
  T_1 \ar{r}{\sigma_1} & S_1 
\end{tikzcd}
\vspace{.7cm}
\end{equation}
is a sequence of embeddings,
\begin{equation} \label{mor} \begin{tikzcd}
{\sf cylr}\left(\sigma_n\right) \ar[hookrightarrow]{r}{E_n} \ar{d}  
& \R^n \times \Delta^1 \ar[twoheadrightarrow]{d}{\tx{pr}_{< n} \times \tx{id}_{\Delta^1}}  \\
{\sf cylr}\left(\sigma_{n-1}\right) \ar[hookrightarrow]{r}{E_{n-1}} \ar{d}  
& \R^{n-1} \times \Delta^1 \ar[twoheadrightarrow]{d}{\tx{pr}_{< n-1} \times \tx{id}_{\Delta^1}} \\
\vdots \ar{d} &  \vdots \ar[twoheadrightarrow]{d} \\
{\sf cylr}\left(\sigma_1\right) \ar[hookrightarrow]{r}{E_1}   
& \R \times \Delta^1 \\
\end{tikzcd}
\end{equation}
\noindent over $\Delta^1$ such that ${E_i}_{|S_i}=e_i$ 
and ${E_i}_{|T_i\times \{1\}}= d_i$, for each $1\leq i \leq n$. 
When the context is clear, we denote (\ref{mor}) 
by $\cylr\left(\un{\sigma}\right) \xhra{\un{E}} \un{\R}^n \times \Delta^1$ or $ \un{E}$. 
\end{observation}

Heuristically an object consists of a finite, possibly 
empty subset of $\R^i$ for each $1 \leq i \leq n$
together with maps between them
which agree with the projection maps $\R^i \to \R^{i-1}$, 
but are not limited to projection only. 
In particular, if in the sequence $S_n \xra{\tau_{n-1}} S_{n-1} \xra{\tau_{n-2}}  \cdots \ra S_1$,
$\tau_{i-1}$ is not a surjection, then there are points in $S_{i-1} \subset \R^{i-1}$
that are not from the projection of $S_i \subset \R^i$ to $\R^{i-1}$, 
namely those points in $S_{i-1}$ whose preimage under $\tau_{i-1}$ is empty. 

Similarly a morphism in $\Exit\bl(\Ranu(\un{\R}^n)\br)$ 
consists of a morphism in $\Exit\bigl(\Ranu(\R^i)\bigr)$ 
for each $1 \leq i \leq n$, the collection of which are compatible under, 
but not limited to projection.
The data of a morphism for all $i < n$ captures 
the fact that morphisms in $\Exit\bl(\Ranu(\un{\R}^n)\br)$ 
distinguish between the Fox-Neuwirth cells in each 
configuration space of $\R^n$ of fixed cardinality.
In other words, its morphisms witness sets of points moving 
from more coordinate coincidence to less, but not vice-versa, 
in addition to anticollision and vanishing of points. 

\begin{notation} 
We denote a point in the $[p]$-space of 
$\Exit\bl(\Ranu(\un{\R}^n)\br)$ over (\ref{seqofmaps}) by 
$$\cylr\left(\un{\sigma}\right) \xhra{\un{E}} \un{\R}^n \times \Delta^p.$$ 
\end{notation}

\begin{observation} \label{neat}
A morphism in $\Exit\bl(\Ranu(\R^n)\br)$ is an isomorphism if and only if 
it is entirely contained in a configuration space of fixed cardinality. 
In contrast a morphism in $\Exit\bl(\Ranu(\un{\R}^n)\br)$ is an isomorphism if and only if 
it is entirely contained in a single Fox-Neuwirth cell of a configuration space of fixed cardinality. 
\end{observation}

\begin{remark}
Ayala-Francis-Rozenblyum-Tanaka's exit-path $\infty$-category
of the Fox-Neuwirth cell stratification $\Conf_r\left(\un{\R}^n\right)_{\Sigma_r}$ (\S\ref{exitFN})
is an $\infty$-subcategory of $\Exit\bl(\Ranu(\un{\R}^n)\br)$.
\end{remark}

\begin{remark}
Ayala-Francis-Rozenblyum-Tanaka's exit-path $\infty$-category construction (Def. \ref{AFRexit})
cannot be used to encode $\Exit\bl(\Ranu(\un{\R}^n)\br)$ because the 
unital Ran space of $\R^n$ does not emit a stratification by cardinality and the Fox-Neuwirth cells;
see \S\ref{FNRanu} and \S\ref{exitFNRanu} for more details. 
\end{remark}

\begin{observation} \label{phii} 
For each $\ds 1 \leq i \leq n$, there is a natural forgetful contravariant functor 
to finite sets $$\ds  \phi_i: \Exit\bl(\Ranu(\un{\R}^n)\br) \lra \Fin^{\op}$$ 
that forgets all but the set data at the $\R^i$ level. 
Its value on an object $\ds \un{S} \xhra{\un{e}} \un{\R}^n$ is $S_i$ and on a morphism 
$$\ds \cylr\left(\un{T} \xra{\un{\sigma}} \un{S}\right) \xhra{\un{E}} \un{\R}^n \times \Delta^1 $$
 from $\un{S} \hra \un{\R}^n$ to $\un{T} \hra \un{\R}^n$ is the map of finite sets 
$$\ds T_i \xra{\sigma_i} S_i.$$ 

The collection of functors $\{\phi_i\}_{i=1}^n$ naturally 
compile to name a canonical functor 
$$\Phi_n: \Exit\bl(\Ranu(\un{\R}^n)\br) \lra \Fun\left(\{1 < \cdots n\}, \Fin^\op\right)$$ 
which just remembers the underlying set data.
Namely, its value on an object $\un{S} \hra \un{\R}^n$ 
is the functor which selects out the composable sequence 
of maps of finite sets $\un{S}$ and its value on a morphism 
$\ds \cylr\left(\un{T} \xra{\un{\sigma}} \un{S}\right) \xhra{\un{E}} \un{\R}^n \times \Delta^1$ 
is the commutative diagram of finite sets $\un{T} \xra{\un{\sigma}} \un{S}$.   
\end{observation}

\begin{observation} \label{forget} 
There is a natural forgetful functor 
$$ \F: \Exit\bl(\Ranu(\un{\R}^n)\br) \lra \Exit\bl(\Ranu(\R^n)\br)$$ 
over $\Fin^\op$ induced by the functor from  
$\Fun\left(\{1<\cdots<n\}, \Fin^\op\right)$ to $\Fin^\op$ that evaluates on $\{n\}$. 
The value of a $[p]$-value $\cylr\left(\un{\sigma}\right) \xhra{\un{E}} \un{\R}^n \times \Delta^p$ 
over (\ref{ppointover}) is the embedding of $\un{E}$ at the $\R^n$ level
 $$\ds \cylr\left(\sigma_n\right) \xhookrightarrow{E_n} \R^n \times \Delta^p$$ 
over $\sigma_n: S_n^p \to \cdots \to S_n^0$. 
\end{observation}

\begin{observation} \label{rho} 
There is a natural forgetful functor 
$$\ds \rho: \Exit\bl(\Ranu(\un{\R}^n)\br) \lra \Exit\left(\Ranu\left(\un{\R}\right)\right)$$ 
that forgets all but the first coordinate data. 
The image of a $[p]$-value 
$$\cylr\left(\un{\sigma}\right) \xhra{\un{E}} \un{\R}^n \times \Delta^p$$
 over (\ref{ppointover}) under $\rho$ is 
$$\cylr\left(\sigma_1\right) \xhra{E_1} \R \times \Delta^p$$
 defined over $\sigma_1: S^p_1 \to \cdots \to S^0_1$. 
 \end{observation}

\begin{observation} \label{pie}  
Recall the wreath product of $\infty$-categories 
(Def. \ref{wreath product}, Def. \ref{wreath2}, \& Note \ref{inftwreath}). 
There is a natural forgetful functor 
$$ \ds \pi: \Exit\bl(\Ranu(\un{\R}^n)\br) \lra 
\Fin^\op \wr \Exit\left(\Ranu\left(\un{\R}^{n-1}\right)\right)$$ 
which remembers the $i$th-coordinate data for each $2 \leq i \leq n$
as fibers over the first coordinate data forgotten to $\Fin^\op$
via the forgetful functor $\phi_1$ (Obs. \ref{phii}). 
The value of $\pi$ on a $[p]$-point 
$\ds \cylr\left(\un{\sigma}\right) \xhra{\un{E}} \un{\R}^n \times \Delta^p$ 
(\ref{ppoint}) over the diagram (\ref{ppointover}) is twofold:
\begin{enumerate}
\item[i.] First, it is the data of the composable string of morphisms 
$\sigma_1: S^p_1 \to \cdots \to S^0_1$ in $\Fin$ from (\ref{ppointover}).
\item[ii.] Second, it is the data of the $[p]$-point in 
$\ds \Exit\left(\Ranu\left(\un{\R}^{n-1}\right)\right)$ obtained as follows. 
For each $\left(p+1\right)$-tuple $\left(s_p \in S^p_1, ... , s_0 \in S^0_1\right)$
such that $\sigma_1\left(s_i\right)=s_{i-1}$ for all $1 \leq i \leq p$,
consider the sequence of embeddings of
the fibers over the cylinder of the restriction of $\sigma_1$
to $s_p \mapsto \cdots \mapsto s_0$, $\cylr\left({\sigma_1}_{|_{s_p}}\right)$:
\begin{equation} \label{pbmor} \ds
\begin{tikzcd}
\cylr\left(\sigma_n\right)_{|_{\sigma_1\left(s_p\right)}} \ar[hookrightarrow]{r} 
\ar[dr, phantom, "\lrcorner", very near start] \ar{d} &  
\cylr\left(\sigma_n\right) \ar[hookrightarrow]{r} \ar{r}\ar{d}  & 
\R^n \times \Delta^p \ar[twoheadrightarrow]{d}{\tx{pr}_{< n} } \\
\cylr\left(\sigma_{n-1}\right)_{|_{\sigma_1\left(s_p\right)}}  
\ar[dr, phantom, "\lrcorner", very near start] \ar[hookrightarrow]{r} \ar{d} & 
\cylr\left(\sigma_{n-1}\right) \ar[hookrightarrow]{r} \ar{r}\ar{d}  & 
\R^{n-1} \times \Delta^p \ar[twoheadrightarrow]{d}{\tx{pr}_{< n-1} } \\
\vdots \ar{d} & \vdots \ar{d} &  \vdots \ar[twoheadrightarrow]{d} \\
\cylr\left(\sigma_2\right)_{|_{\sigma_1\left(s_p\right)}} \ar[hookrightarrow]{r} 
\ar[dr, phantom, "\lrcorner", very near start] \ar{d} &  
\cylr\left(\sigma_2\right) \ar[hookrightarrow]{r} \ar{r}\ar{d}  & 
\R^2 \times \Delta^p \ar[twoheadrightarrow]{d}{\tx{pr}_{< 2} } \\
\cylr\left({\sigma_1}_{|_{s_p}}\right)\simeq \Delta^p \ar[hookrightarrow]{r} & 
\cylr\left(\sigma_1\right) \ar[hookrightarrow]{r}   & \R \times \Delta^p.
\end{tikzcd} 
\end{equation} 
Although this sequence of embeddings defines 
a $[p]$-point in $\Exit\left(\Ranu\left(\un{\R}^{n}\right)\right)$, 
it canonically determines a $[p]$-point in 
$\ds \Exit\left(\Ranu\left(\un{\R}^{n-1}\right)\right)$ because
each embedding 
$$\ds \cylr\left(\sigma_i\right)_{|_{\sigma_1\left(s_p\right)}} \hookrightarrow \R^i \times \Delta^p$$ 
for each $\ds 2 \leq i \leq n$
together with the first embedding
$$\cylr\left({\sigma_1}_{|_{s_p}}\right)\simeq \Delta^p \hookrightarrow \R \times \Delta^p$$
canonically factors through $\ds \R^{i-1}\times \Delta^p$.
\end{enumerate}  
\end{observation}

\subsection{\texorpdfstring{$\Exit\bigl(\Ranu(\protect\un{\R}^n)\bigr)$}{Lg}
is an \texorpdfstring{$\infty$}{Lg}-category} \label{yii}
 This section is devoted to proving the technical result that the simplicial space 
 $\Exit\bl(\Ranu(\un{\R}^n)\br)$ is a complete Segal space (Def. \ref{cSegspc}). 
 
\begin{prop}\label{poo}
The simplicial space $\Exit\bl(\Ranu(\un{\R}^n)\br)$ 
satisfies the Segal and completeness conditions.	
\end{prop}

We build directly on \S\ref{unital}, 
 wherein we show that $\Exit\bl(\Ranu(M)\br)$ is a complete Segal space. 
 Our approach in \S\ref{unital} witnesses the simplicial space 
 $\Exit\bl(\Ranu(M)\br)$ as one derived through formal constructions 
 among complete Segal spaces from the complete Segal space 
 $\Bun$ defined in~\S6 of~\cite{AFR}. 
 Our approach in this section is similar in that we show that 
 $\Exit\bl(\Ranu(\un{\R}^n)\br)$, too, can be derived through formal constructions 
 among complete Segal spaces from $\Bun$. 
 Just as in \S\ref{unital}, we freely use notation and results of \cite{AFR}
 throughout the section.
 
Recall the $\infty$-category $\RRef^0(\R)$ defined as pullback (\ref{Ref}). 
Heuristically, an object is a refinement of $\R$ for which 
the $0$-skeleton of the domain is a finite set, 
and a morphism is a path of such refinements of $\R$ 
witnessing anticollision of strata and disappearences of strata. 
Observe that $\RRef^0(\R) \simeq \Ar^{{\sf ref}}\left(\Bun\right)_{|\R}$ is 
equivalent to the $\infty$-category of refinement arrows in $\Bun$ 
with the target fixed as the trivially stratified space $\R$. 
This is because every refinement of $\R$ has as its $0$-skeleton, 
a finite (possibly empty) set. The equivalence is given by the functor 
from $\Ar^{{\sf ref}}\left(\Bun\right)_{|\R}$ to $\RRef^0(\R)$ which forgets the target. 
Let $\RRef\left(\R^n\right)$ denote the $\infty$-category of refinement arrows in $\Bun$ 
with the target fixed as $\R^n$ equipped with the trivial stratification. 
Explicitly, $\RRef\left(\R^n\right)$ is the simplicial space 
whose value on $[p] \in \DDelta$ is the moduli space of

\begin{itemize}
\item
constructible bundles
\[
Y
\longrightarrow
\Delta^p
\]

\item
together with a refinement 
\[
Y
\xra{~\rm refinement~}
\R^n \times \Delta^p
\]
over $\Delta^p$.
\end{itemize}

We denote such a $[p]$-point of $\RRef\left(\R^n\right)$ simply as 
$\left(Y \xra{\sf ref} \R^n \times \Delta^p\right)$. 
Note that $\R^n \times \Delta^p$ is stratified as a product stratified space, 
where $\R^n$ is trivially stratified over the poset consisting of a singleton, 
and $\Delta^p$ is given the standard stratification (Ex.~\ref{standard}). 

For $n \geq 2$, define the functor 
\[ F_n: \RRef^0(\R) \ra \RRef\left(\R^n\right) 
\]
from refinements of $\R$ to refinements of $\R^n$ as follows: 
The value of a $[p]$-point $\left(X \xra{{\sf ref}} \R \times \Delta^p\right)$ 
under $F_n$ is the refinement of $\R^n \times \Delta^p$ defined as the pullback
\begin{equation} \label{F_n} 
\begin{tikzcd}
F_n\left(X\right) \ar{r} \ar{d}[swap]{\alpha} \pb & 
\R^n \times \Delta^p \ar{d}{\pr_{<2} \times \Id_{\Delta^p}} \\
X \ar{r}{{\sf ref}} & \R \times \Delta^p
\end{tikzcd} 
\end{equation}
of stratified spaces.
It is straightforward to check that the value $F_n\left(X\right)$ 
is a refinement of $\R^n \times \Delta^p$ by virtue of it being a pullback. 
Explicitly, the map $\left(F_n\left(X\right) \to \R^n\times \Delta^p\right)$ is 
$$\left(\left(\pr_{<2}\times \Id_{\Delta^p}\right)^{-1}\left(\sk_0^{{\fib}}\left(X\right)\right) 
\subset \R^ n \times \Delta^p\right) \xra{\sf ref} \R^n \times \Delta^p$$ 
which denotes the coarsest refinement of $\R^n \times \Delta^p$ 
for which the embedding from  
$$\left(\pr_{<2}\times \Id_{\Delta^p}\right)^{-1}\left(\sk_0^{{\fib}}\left(X\right)\right)$$ 
is proper and constructible. As a technicality (that we use later) 
define $F_1$ to be the identity on $\RRef^0(\R)$. 
Figure~\ref{fig1} is a sketch of the values of an object and a morphism under 
$F_2: \RRef^0(\R) \to \RRef\left(\R^2\right).$

\begin{figure}[ht] 
\centering
 \includegraphics[scale=0.6]{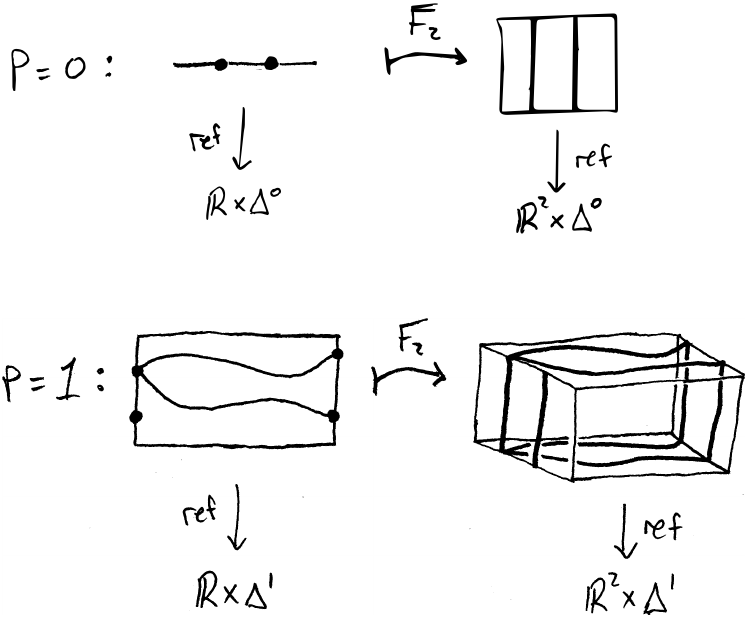}
 \caption{The values of a $[0]$-point and a $[1]$-point in $\RRef^0(\R)$ under $F_2$.}
 \label{fig1}
\end{figure}

Denote the pullback $\infty$-category 
\begin{equation} \label{D3} 
\begin{tikzcd}
\widetilde{\RRef}\left(\R^n\right) \ar{r} \ar{d} \pb & 
\Ar^{{\sf ref}}\left(\RRef\left(\R^n\right)\right) \ar{d}{\tx{target}} \\
\RRef^0(\R) \ar{r}{F_n} & \RRef\left(\R^n\right)
\end{tikzcd} 
\end{equation}
where $\Ar^{{\sf ref}}\left(\RRef\left(\R^n\right)\right)$ is the $\infty$-category 
of refinement arrows of $\RRef\left(\R^n\right)$.
In other words it is the full $\infty$-subcategory of the $\infty$-category of arrows 
of $\RRef\left(\R^n\right)$ consisting of the refinement arrows. 

Recall the open cylinder construction (\ref{e}). 
We aim to determine that $\widetilde{\RRef}\left(\R^n\right)$ is the following simplicial space: 
Its value on $[p] \in \Delta$ is the moduli space of
\begin{itemize}
\item
pairs of constructible bundles
\[
\left( \left(X
\longrightarrow
\Delta^p\right), \shs \left(Y \lra \Delta^p\right) \right)
\]
\item
together with a pair of refinements of stratified spaces
\[
\left( \left(X \xra{\tx{refinement}} \R \times \Delta^p\right), 
\shs \left(Y \xra{\tx{refinement}} F_n\left(X\right)\right) \right) \]
each of which is over $\Delta^p$.
\end{itemize}
To keep in mind that $Y$ is, in particular, a refinement of $\R^n \times \Delta^p$, 
we denote such a $[p]$-point of $\widetilde{\RRef}\left(\R^n\right)$ by 
\begin{equation} \label{tri} \begin{tikzcd} 
Y \ar{r}{\sf ref} \ar{dr}[swap]{\sf ref} & F_n\left(X\right) \ar{d}{\sf ref} \\
& \R^n \times \Delta^p 
\end{tikzcd} 
\end{equation}
Figure~\ref{fig2} is a sketch of an object ($p=0$) 
and a morphism ($p=1$) in $\widetilde{\RRef}\left(\R^2\right)$.

\begin{figure}[ht] 
\centering
 \includegraphics[scale=0.5]{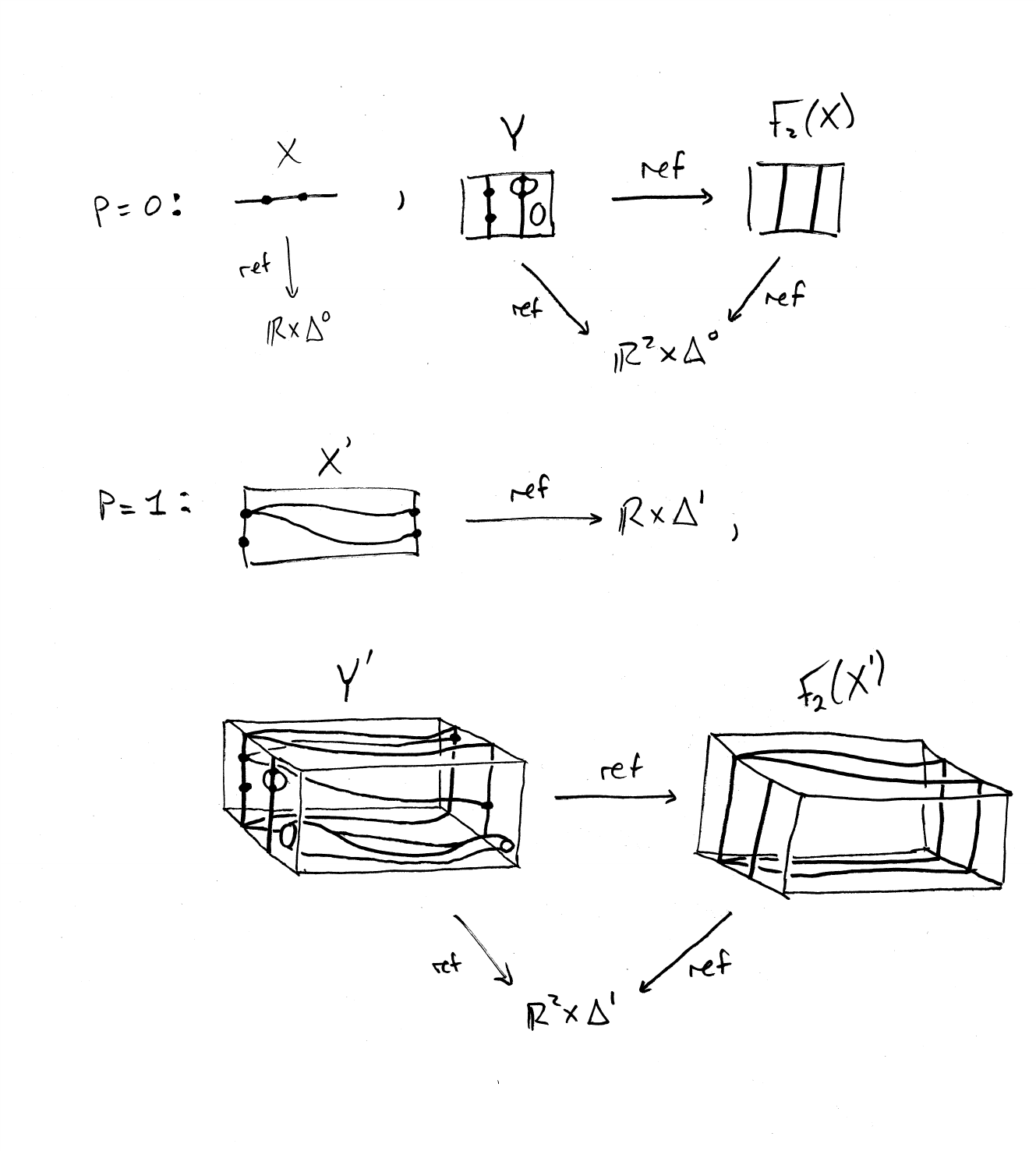}
 \caption{A $[0]$-point and a $[1]$-point in $\widetilde{\RRef}\left(\R^2\right)$.} 
 \label{fig2}
\end{figure}

For each such $[p]$-point of $\widetilde{\RRef}\left(\R^n\right)$ (\ref{tri}) 
there is a canonical stratified map of stratified spaces from $Y$ to $X$ 
defined to be the composite of the stratified maps 
$$Y \xra{{\sf ref}} F_n\left(X\right) \xra{\alpha} X$$ 
where $\alpha$ is the stratified map in (\ref{F_n}) 
which is given by virtue of $F_n\left(X\right)$ being defined as a pullback. 
The map of underlying topological spaces is projection 
onto the first Euclidean coordinate product with the identity on $\Delta^p$. 
We denote this stratified projection map $Y \xra{\pr_{<2}} X$. 

\begin{definition} 
For $n \geq 2$, the $\infty$-category $\widetilde{\RRef}\left(\un{\R}^n\right)$ 
is defined inductively on $n$:

$\widetilde{\RRef}\left(\un{\R}^2\right)$ is the full $\infty$-subcategory 
of $\widetilde{\RRef}\left(\R^2\right)$ consisting of those objects 
\[ \begin{tikzcd} 
Y \ar{r}{\sf ref} \ar{dr}[swap]{\sf ref} & F_2\left(X\right) \ar{d}{\sf ref} \\
& \R^2 
\end{tikzcd} 
\]
for which the $1$-skeleton of the open cylinder of 
$\left(Y \xra{{\sf ref}} F_2\left(X\right)\right)$ is a refinement morphism in $\Bun$. 

$\widetilde{\RRef}\left(\un{\R}^n\right)$ is the full $\infty$-subcategory of 
$\widetilde{\RRef}\left(\R^n\right)$ consisting of those  objects 
\[ \begin{tikzcd} 
Y \ar{r}{\sf ref} \ar{dr}[swap]{\sf ref} & F_n\left(X\right) \ar{d}{\sf ref} \\
& \R^n 
\end{tikzcd} 
\] such that 
\begin{itemize}
\item[i.]  the $\left(n-1\right)$-skeleton of the open cylinder of 
$\left(Y \xra{{\sf ref}} F_n\left(X\right)\right)$ 
is a refinement morphism in $\Bun$. 

\item[ii.]  the fiber of the stratified projection map $Y \xra{\pr_{<2}} X$ 
over each point in the $0$-skeleton of $X$ is an object in 
$\widetilde{\RRef}\left(\un{\R}^{n-1}\right)$. 
\end{itemize}
\end{definition}

Explicitly, $\widetilde{\RRef}\left(\un{\R}^n\right)$ is the simplicial space 
whose value on $[p] \in \DDelta$ is the moduli space of
\begin{itemize}
\item
pairs of constructible bundles
\[
\left(\left(X \lra \Delta^p\right), \shs \left(Y
\longrightarrow
\Delta^p\right)\right)
\]
\item
together with a pair of refinements among stratified spaces
\[
\left( \left(X \xra{\tx{refinement}} \R \times \Delta^p\right), 
\shs \left(Y \xra{\tx{refinement}}  F_n\left(X\right)\right) \right)
\]
each of which is over $\Delta^p$
\end{itemize}
satsifying the conditions
\begin{itemize} 
\item[i.] the fiberwise $\left(n-1\right)$-skeleton of the open cylinder 
of $\left(Y \xra{{\sf ref}}  F_n\left(X\right) \right)$
is a refinement morphism in $\Bun$. 

\item[ii.] the fiber of the stratified projection map $\left(Y \xra{\pr_{<2}} X\right)$  
over each point in the $0$-skeleton of $X$ is 
an object in $\widetilde{\RRef}\left(\un{\R}^{n-1}\right)$.
\end{itemize} 

We will denote such a $[p]$-point in $\widetilde{\RRef}\left(\un{\R}^n\right)$ as 
\[ \begin{tikzcd} 
Y \ar{r}{\sf ref} \ar{dr}[swap]{\sf ref} & F_n\left(X\right) \ar{d}{\sf ref} \\
& \R^n \times \Delta^p 
\end{tikzcd} 
\] 
Figure~\ref{fig3} is a sketch of an object ($p=0$) 
and a morphism ($p=1$) in $\widetilde{\RRef}\left(\un{\R}^2\right)$.
\begin{figure}[h]
\centering
 \includegraphics[scale=0.5]{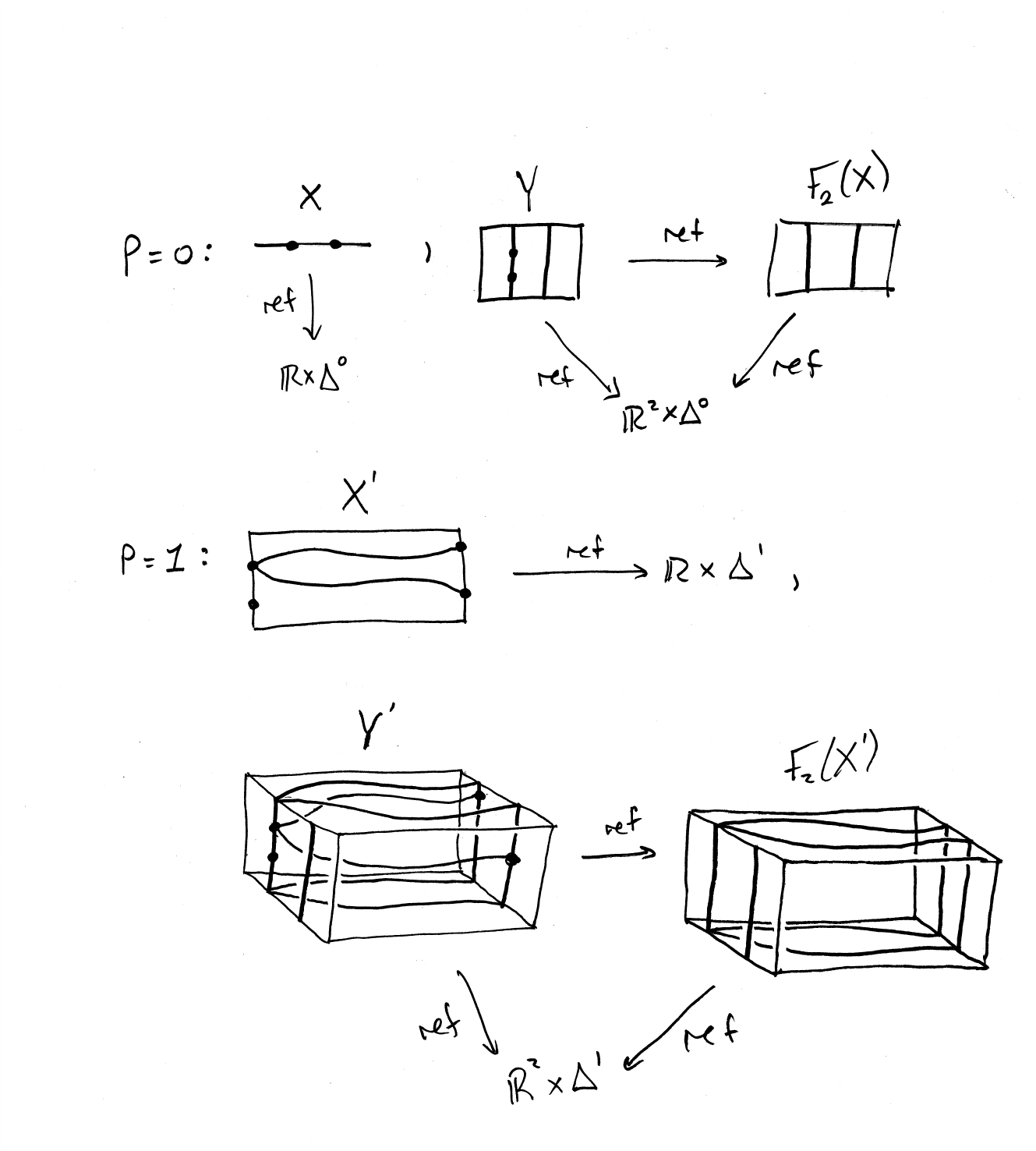}
 \caption{A $[0]$-point and a $[1]$-point in $\widetilde{\RRef}\left(\un{\R}^2\right)$.} 
 \label{fig3}
\end{figure}

\begin{lemma} \label{O1} 
For each integer $n \geq 1$ and $ 1 \leq k \leq n-1$, there is a canonical functor
\[ \pr_{k}: \widetilde{\RRef}\left(\un{\R}^n\right) \lra \widetilde{\RRef}\left(\un{\R}^{k}\right) \]
induced by projection from $\R^n$ onto the first $k$-coordinates. 
The functor is given by, for each $[p] \in \DDelta$, assigning to the $[p]$-point 
\[ \begin{tikzcd} 
Y \ar{r}{\sf ref} \ar{dr}[swap]{\sf ref} & F_n\left(X\right) \ar{d}{\sf ref} \\
& \R^n \times \Delta^p  
\end{tikzcd} 
\] 
the refinement
\[ \begin{tikzcd} 
\left( \pr_{<k+1}\times \Id_{\Delta^p}\left(\sk^{{\fib}}_{n-k}\left(Y\right)\right) 
\subset \cdots \subset \pr_{<k+1}\times \Id_{\Delta^p}\left(\sk^{{\fib}}_{n-1}\left(Y\right)\right) 
\subset \R^{k} \times \Delta^p \right) \ar{r}{\sf ref} \ar{dr}[swap]{\sf ref} & F_k\left(X\right)   \ar{d}{\sf ref} \\
& \R^k \times \Delta^p 
\end{tikzcd}
 \]
the value of which is the coarsest refinement of $\R^k \times \Delta^p$ 
for which the embeddings from \\ $\pr_{<k+1}\left(\sk^{\fib}_i\left(Y\right)\right)$ 
into $\R^k \times \Delta^p$ for each $n-k \leq i \leq n-1$ are proper and constructible. 
We denote such a value by 
\[ \begin{tikzcd} 
Y_k \ar{r}{\sf ref} \ar{dr}[swap]{\sf ref} & F_k\left(X\right) \ar{d}{\sf ref} \\
& \R^k \times \Delta^p 
\end{tikzcd} 
\] 
\end{lemma}

Before we proceed with the proof, we provide a sketch of the values 
of an object in $\widetilde{\RRef}\left(\un{\R}^3\right)$ under $\pr_2$ and $\pr_1$ in Figure~\ref{fig4}.
\begin{figure}[h]
\centering
\includegraphics[scale=0.6]{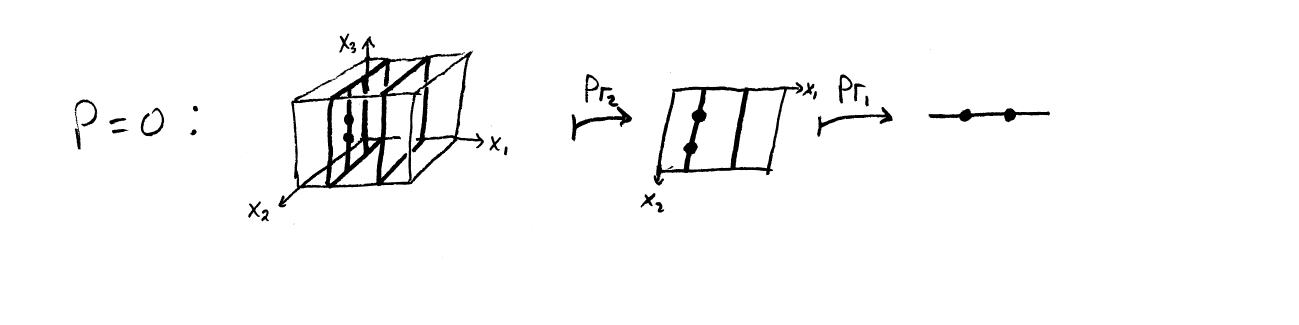}  
 \caption{The values of a $[0]$-point of $\widetilde{\RRef}\left(\un{\R}^3\right)$ under $\pr_2$ and $\pr_1$.}
 \label{fig4}
\end{figure}

 \begin{proof}[Proof of Lemma~\ref{O1}]  
 We proceed by induction on $k$. 
 For the base case, let $k=1$. 
 Recall the definition of $\widetilde{\RRef}\left(\R^n\right)$ as the pullback (\ref{D3}). 
 Observe that the target $\RRef^0(\R)$ of the leftmost vertical functor 
 of (\ref{D3}) is, by definition, $\widetilde{\RRef}\left(\un{\R}\right)$. 
 We claim that $\pr_1$ is precisely this functor upon 
 restricting the domain to $\widetilde{\RRef}\left(\un{\R}^n\right)$. 
 To verify this claim, first note that the functor in (\ref{D3}), 
 by virtue of $\widetilde{\RRef}\left(\R^n\right)$ being a pullback, 
 is given by the assignment, for each $[p] \in \DDelta$,
\[ \begin{tikzcd} 
Y \ar{r}{\sf ref} \ar{dr}[swap]{\sf ref} & 
F_n\left(X\right) \ar{d}{\sf ref}  \ar[mapsto, shorten >= 0.5em, shorten <= 0.5em]{r} & 
X \ar{r}{\sf ref} & \R \times \Delta^p  \\
& \R^n \times \Delta^p  &&
\end{tikzcd} 
\]  
Therefore, we must show that the value of such a $[p]$-point under $\pr_1$, 
\[ \begin{tikzcd} 
Y_1 \ar{r}{\sf ref} \ar{dr}[swap]{\sf ref} & F_1\left(X\right) \ar{d}{\sf ref} \\
& \R \times \Delta^p 
\end{tikzcd} 
\] 
 is $\left(X \xra{\sf ref} \R \times \Delta^p\right)$. 
 Recall that (as a technicality) $F_1$ was previously 
 defined to be the identity on $\RRef^0(\R)$. 
 Thus, we need to verify that $Y_1$ and $X$ 
 are equivalent refinements of $\R \times \Delta^p$. 
 By virtue of being a $[p]$-point of $\widetilde{\RRef}\left(\un{\R}^n\right)$, 
 the fiberwise $\left(n-1\right)$-skeleton of $Y$ refines 
 the fiberwise $\left(n-1\right)$-skeleton of $F_n\left(X\right)$. 
 This means that, in particular, the underlying topological spaces 
 of the $\left(n-1\right)$-skeletons of $Y$ and $F_n\left(X\right)$ are homeomorphic. 
 Thus, their projections onto the first Euclidean coordinate product with the identity on $\Delta^p$ 
\begin{equation} \label{E1} 
\pr_{<2}\times \Id_{\Delta^p}\left(\sk_{n-1}^{\fib}\left(Y\right)\right) 
\cong \pr_{<2}\left(\sk_{n-1}^{\fib}\left(F_n\left(X\right)\right)\right) 
\end{equation}
are homeomorphic subspaces of $\R \times \Delta^p$ over $\Delta^p$. 
This simply means that the fibers of each over the same point in $\Delta^p$ have the same cardinality. 

Previously we observed that explicitly $F_n\left(X\right)$ is the coarsest refinement 
$$\left(\pr_{<2}\times \Id_{\Delta^p}\right)^{-1}\left(\sk_0^{{\fib}}\left(X\right)\right) 
\subset \R^ n \times \Delta^p$$
of $\R^n \times \Delta^p$
for which the embedding from 
$\left(\pr_{<2}\times \Id_{\Delta^p}\right)^{-1}\left(\sk_0^{{\fib}}\left(X\right)\right)$ 
is proper and constructible. 
This means that the fiberwise $\left(n-1\right)$-skeleton of $F_n\left(X\right)$ is 
$\left(\pr_{<2}\times \Id_{\Delta^p}\right)^{-1}\left(\sk_0^{{\fib}}\left(X\right)\right)$. 
Thus, the projection of the $\left(n-1\right)$-skeleton onto the first Euclidean coordinate 
product with the identity on $\Delta^p$ is simply the fiberwise $0$-skeleton of $X$, i.e.,
$$\pr_{<2}\times \Id_{\Delta^p}\left(\sk_{n-1}^{\fib}\left(F_n\left(X\right)\right)\right) 
= \sk_0^{\fib}\left(X\right).$$
 Therefore, through equivalence (\ref{E1}), 
 the fiberwise $0$-skeleton of $X$ is homeomorphic to  
 $$\pr_{<2}\times \Id_{\Delta^p}\left(\sk_{n-1}^{\fib}\left(Y\right)\right).$$
 Upon making the (somewhat trivial) observation that 
 $\left(X \xra{\sf ref} \R \times \Delta^p\right)$ has the explicit description 
 as the coarsest refinement $\left(\sk_0^{\fib}\left(X\right) \subset \R \times \Delta^p\right)$ 
 of $\R \times \Delta^p$ for which the embedding f
 rom the fiberwise $0$-skeleton of $X$ is proper and constructible, 
 we conclude that 
 $$Y_1:= \left(\pr_{<2}\left(\sk_{n-1}^{\fib}\left(Y\right)\right) 
 \subset \R \times \Delta^p\right)$$
 is equivalent to $X$ as a refinement of $\R \times \Delta^p$, 
 which proves the base case.

For the inductive step we need to check that the value
\begin{equation} \label{P1} \begin{tikzcd} 
Y_k \ar{r}{\sf ref} \ar{dr}[swap]{\sf ref} & F_k\left(X\right) \ar{d}{\sf ref} \\
& \R^k \times \Delta^p 
\end{tikzcd} 
\end{equation} 
is in fact a $[p]$-value in $\widetilde{\RRef}\left(\un{\R}^k\right)$. 
Thus, we must verify three things: 
\begin{enumerate}
\item[(a)] $Y_k$ refines $F_k\left(X\right)$,
\item[(b)] the $\left(k-1\right)$-skeleton of the open cylinder of 
$\left(Y_k \xra{{\sf ref}} F_k\left(X\right)\right)$  is a refinement, and
\item[(c)] the fiber of the stratified projection map $Y_k \xra{\pr_{<2}} X$ 
over each point in the $0$-skeleton of $X$ is an object in
 $\widetilde{\RRef}\left(\un{\R}^{k-1}\right)$.
\end{enumerate} 

(a) Recall that $F_n\left(X\right)$ is the coarsest refinement 
$$\left(\pr_{<2}\times \Id_{\Delta^p}\right)^{-1}\left(\sk_0^{{\fib}}\left(X\right)\right) 
\subset \R^ n \times \Delta^p$$ 
of $\R^n \times \Delta^p$ 
for which the embedding from 
$\left(\pr_{<2}\times \Id_{\Delta^p}\right)^{-1}\left(\sk_0^{{\fib}}\left(X\right)\right)$ 
is proper and constructible. 
Projection of $F_n\left(X\right)$ onto its first $k$ Euclidean coordinates 
(product with the identity on $\Delta^p$) yields an explicit description  
\begin{equation} \label{E9} 
F_k\left(X\right) = \left(\pr_{<k+1}\times 
\Id_{\Delta^p}\left(\sk_{n-1}^{\fib}\left(F_n\left(X\right)\right)\right) 
\subset \R^k \times \Delta^p\right) 
\end{equation}
which denotes the coarsest refinement of $\R^k \times \Delta^p$ 
for which the embedding from  
$$\pr_{<k+1}\times \Id_{\Delta^p}\left(\sk_{n-1}^{\fib}\left(F_n\left(X\right)\right)\right)$$ 
is proper and constructible.
 By definition of a $[p]$-point of $\widetilde{\RRef}\left(\un{\R}^n\right)$,
\[ \begin{tikzcd} 
Y \ar{r}{\sf ref} \ar{dr}[swap]{\sf ref} & F_n\left(X\right) \ar{d}{\sf ref} \\
& \R^n \times \Delta^p  
\end{tikzcd} 
\] 
satisfies that the fiberwise $\left(n-1\right)$-skeleton of $Y$ is a refinement 
of the fiberwise $\left(n-1\right)$-skeleton of $F_n\left(X\right)$. 
In particular Then the underlying topological spaces of the 
$\left(n-1\right)$-skeletons are homeomorphic. 
Thus, so are their projections onto the first $k$ Euclidean coordinates 
(product with the identity on $\Delta^p$), i.e., 
\begin{equation} \label{E5} 
\pr_{<k+1}\times \Id_{\Delta^p}\left(\sk_{n-1}^{\fib}\left(Y\right)\right) 
\cong \pr_{<k+1}\times \Id_{\Delta^p}\left(\sk_{n-1}^{\fib}\left(F_n\left(X\right)\right)\right).
\end{equation}
By definition, $Y_k$ refines the coarsest refinement 
$$\pr_{<k+1}\times \Id_{\Delta^p}\left(\sk_{n-1}\left(Y\right)\right) \subset \R^k \times \Delta^p$$
of $\R^k \times \Delta^p$ for which the embedding from 
$\pr_{<k+1}\times \Id_{\Delta^p}\left(\sk_{n-1}\left(Y\right)\right)$ is proper and constructible. 
Thus, through the equivalences (\ref{E9}) and (\ref{E5}) above, 
we conclude that $Y_k$ refines $F_k\left(X\right)$. 

(b) By definition, the fiberwise $\left(k-1\right)$-skeleton of $Y_k$ is a refinement of 
$\pr_{<k+1}\times \Id_{\Delta^p}\left(\sk_{n-1}^{\fib}\left(Y\right)\right)$. 
Through equivalences (\ref{E9}) and (\ref{E5}), 
the fiberwise $\left(k-1\right)$-skeleton of $Y_k$ refines the fiberwise 
$\left(k-1\right)$-skeleton of $F_k\left(X\right)$. 
That is to say, the refinement $\left(Y_k \xra{\sf ref} F_k\left(X\right)\right)$ of stratified spaces 
restricts to a refinement 
$\left(\sk_{k-1}\left(Y_k\right) \xra{\sf ref} \sk_{k-1}\left(F_k\left(X\right)\right)\right)$ 
of stratified spaces between the $\left(k-1\right)$-skeletons. 
Thus, the $\left(k-1\right)$-skeleton of the open cylinder of 
$\left(Y_k \xra{\sf ref} F_k\left(X\right)\right)$ is a refinement morphism in $\Bun$. 

(c) The fiber of $Y_k \xra{\pr_{<2}} X$ over a point in the $0$-skeleton of $X$ can, 
by definition of $Y_k$,  be described in terms of projection of the fiber 
of $Y \xra{\pr_{<2}} X$ over the same point as follows: 
Let $x$ be a point in the $0$-skeleton of $X$ and let $Y_x$ denote the fiber over $x$ in $Y$. 
The fiber over $x$ in $Y_k$ is 
$$\left(Y_k\right)_x := \left(\pr_{<k}\times \Id_{\Delta^p}\left(\sk^{\fib}_{n-k}\left(Y_x\right)\right) 
\subset \cdots \subset \pr_{<k}\times \Id_{\Delta^p}\left(\sk^{\fib}_{n-2}\left(Y_x\right)\right) 
\subset \R^{k-1}\right) \xra{{\sf ref}} \R^{k-1}$$
which denotes the coarsest refinement of $\R^{k-1}$ 
for which the embedding from $\pr_{<k}\times \Id_{\Delta^p}\left(\sk_i\left(Y_x\right)\right)$ into $\R^{k-1}$, 
for each $n-k \leq i \leq n-2$, is proper and constructible. 
By definition, $\left(Y_k\right)_x$ is the value of $Y_x$ under the projection functor  
$\widetilde{\RRef}\left(\un{\R}^{n-1}\right) \xra{\pr_{k-1}} \widetilde{\RRef}\left(\un{\R}^{k-1}\right)$, 
which exists by the inductive step, and therefore identifies 
$\left(Y_k\right)_x$ as an object in $\widetilde{\RRef}\left(\un{\R}^{k-1}\right)$, as desired. 
\end{proof}

\begin{observation} 
There is a natural  functor 
\[ \widetilde{\RRef}\left(\un{\R}^n\right) \lra \Fun\left(\{1 < \cdots < n\}, \Fin^\op\right)
\]
 the value of which on an object 
\[ \begin{tikzcd} 
Y \ar{r}{\sf ref} \ar{dr}[swap]{\sf ref} & F_n\left(X\right) \ar{d}{\sf ref} \\
& \R^n   
\end{tikzcd} 
\]  
is the functor from $\{1 < \cdots < n \}$ to $\Fin^\op$ that selects out the sequence 
\[ \sk_0^{\fib}\left(Y\right) \xra{\pr_{<n}} \sk_0^{\fib}\left(Y_{n-1}\right) 
\ra \cdots \xra{\pr_{<2}} \sk_0^{\fib}\left(X\right) \]
of finite sets. 

The value on a morphism 
\[ \begin{tikzcd} 
Y \ar{r}{\sf ref} \ar{dr}[swap]{\sf ref} & F_n\left(X\right) \ar{d}{\sf ref} \\
& \R^n \times \Delta^1  
\end{tikzcd} 
\] 
 is given by selecting out the diagram
\[ \begin{tikzcd}
\sk_0^{\fib}\left(Y_{|1}\right) \ar{r} \ar{d}[swap]{\pr_{<n}} & 
\sk_0^{\fib}\left(Y_{|0}\right) \ar{d}{\pr_{<n}} \\
\sk_0^{\fib}\left(\left(Y_{n-1}\right)_{|1}\right) \ar{r} \ar{d} & 
\sk_0^{\fib}\left(\left(Y_{n-1}\right)_{|0}\right) \ar{d} \\
\vdots \ar{d}[swap]{\pr_{<2}} & \vdots \ar{d}{\pr_{<2}} \\
\sk_0^{\fib}\left(\left(X\right)_{|1}\right) \ar{r}  & 
\sk_0^{\fib}\left(\left(X\right)_{|0}\right)  
\end{tikzcd} \]
among finite sets, where for each $1 \leq k \leq n$, 
the horizontal arrow is from the $\left(0\right)$-skeleton of the fiber of $Y_k \to \Delta^1$ 
over $\{1\} \in \Delta^1$ to the $\left(0\right)$-skeleton of the fiber of 
$Y_k \to \Delta^p$ over $\{0\} \in \Delta^1$, 
and is a canonical map of sets implemented by taking 
connected components of the fiberwise $\left(0\right)$-skeleton of $X$. 
\end{observation}

\begin{lemma}\label{L1}
There is a canonical equivalence of simplicial spaces 
\[
\widetilde{\RRef}\left(\un{\R}^n\right)
{~\simeq~}
\Exit\bl(\Ranu(\un{\R}^n)\br)
\]
over  $\Fun\left(\{1 < \cdots < n\}, \Fin^\op\right)$.
\end{lemma}

\begin{proof}
A rightward morphism is implemented for each $[p]\in \mathbf{\Delta}$ 
by the assignment
\[
\begin{tikzcd}
  &  & \sk_0\left(Y\right) \ar[hookrightarrow]{r} \ar{d}[swap]{\pr_{<n}\times \Id_{\Delta^p}} & 
  Y \ar{r}  & \R^n \times \Delta^p \ar{d}{\pr_{<n}\times \Id_{\Delta^p} }\\
  Y \ar{r}{{\sf ref}}\ar{dr}[swap]{\sf ref}  & 
  F_n\left(X\right) \ar{d}{\sf ref}   \ar[mapsto, shorten >= 0.5em, shorten <= 0.5em]{r} & 
  \sk_0\left(Y_{n-1}\right) \ar[hookrightarrow]{r} \ar{d} & 
  Y_{n-1} \ar{r} & \R^{n-1} \times \Delta^p \ar{d} \\ 
  &\R^n \times \Delta^p& \vdots \ar{d}[swap]{\pr_{<2}\times \Id_{\Delta^p}} &  
  & \vdots \ar{d}{\pr_{<2}\times \Id_{\Delta^p}} \\
 & & \sk_0\left(X\right) \ar[hookrightarrow]{r} & X \ar{r} & \R \times \Delta^p 
\end{tikzcd} 
\]
whose value is the sequence of embeddings over $\Delta^p$, 
each of which is from the fiberwise $\left(0\right)$-skeleton of the value $Y_k$ of $Y$ 
under the functor $\pr_k$ (Lem. \ref{O1}), which maps to $\Delta^p$ 
as a finite proper constructible bundle. 

A leftward morphism is given by assigning to each $[p]$-point 
\[
\begin{tikzcd}
{\sf cylr}\left(\sigma_n\right) \ar[hookrightarrow]{r} \ar{d}  
& \R^n \times \Delta^p \ar{d}{\tx{pr}_{< n}\times \Id_{\Delta^p} }  \\
{\sf cylr}\left(\sigma_{n-1}\right) \ar[hookrightarrow]{r} \ar{d}  & \R^{n-1} \times \Delta^p \ar{d} \\
\vdots \ar{d} &  \vdots \ar{d}{\pr_{<2}\times \Id_{\Delta^p}} \\
{\sf cylr}\left(\sigma_1\right) \ar{r}   & \R \times \Delta^p 
\end{tikzcd}
\]
the refinement
\[
 \left( \cylr\left(\sigma_n\right) \subset 
 \left(\pr_{<n}\times \Id_{\Delta^p}\right)^{-1}\left(\cylr\left(\sigma_{n-1}\right)\right) 
 \subset \cdots \subset \left(\pr_{<2}\times 
 \Id_{\Delta^p}\right)^{-1}\left(\cylr\left(\sigma_1\right)\right) 
 \subset \R^n \times \Delta^p \right) \to \R^n \times \Delta^p 
\]
whose value is the coarsest refinement of $\R^n \times \Delta^p$ 
for which the embeddings from $\cylr\left(\sigma_n\right)$ 
and each prevalue 
$\left(\pr_{<i}\times \Id_{\Delta^p}\right)^{-1}\left(\cylr\left(\sigma_{i-1}\right)\right)$ for $2 \leq i \leq n$ 
are proper and constructible.  
(Such a refinement exists because the value of each embedding is, 
by definition, a properly embedded stratified subspace.)

It is straightforward to verify that these two assignments are mutually inverse to one another, 
and further, that they are over  $\Fun\left(\{1 < \cdots < n\}, \Fin^\op\right)$.  
Furthermore, it is evident that the structure maps are equivalent. 	
\end{proof}

\begin{proof}[Proof of Proposition~\ref{poo}]
The simplicial space $\widetilde{\RRef}\left(\R^n\right)$ is a pullback of complete Segal spaces 
and is also, therefore, a complete Segal space, 
since the full $\infty$-subcategory of simplicial spaces 
consisting of the complete Segal spaces is closed under formation of pullbacks. 
By virtue of $\widetilde{\RRef}\left(\un{\R}^n\right)$ being a full subsimplicial space of 
$\widetilde{\RRef}\left(\R^n\right)$, it too satisfies the Segal and completeness conditions. 
Through the equivalence of Lemma~\ref{L1}, 
then so does the simplicial space $\Exit\bl(\Ranu(\un{\R}^n)\br)$.
\end{proof}

\section{Part 1 of the proof of the main result} \label{ah}
The goal of this section is to prove that the $\infty$-category 
$\Exit\left(\Ranu\left(\protect\un{\R}^n\right)\right)$
is equivalent to the category $\mathbf{\Theta}_n^\Act$.
This result is the first step in proving Theorem~\ref{loc}.
We start by constructing the functor between them.

\begin{lemma} \label{functor}  
For $n \geq 1$, there is a functor of $\infty$-categories
$$ \ds \G_n: \Exit\bl(\Ranu(\un{\R}^n)\br) \ra \mathbf{\Theta}_n^{\tx{act}}$$ 
over  $\Fun\left(\{1< \cdots < n\}, \Fin^\op\right)$.
 \end{lemma}
 
 \begin{note}
 A functor from an $\infty$-category to the nerve of a (small) category 
is completely determined by its assignment on objects, morphisms 
and the requirement that composition is respected. 
This is due to the fact that the nerve of a small category is completely 
determined by its values on $[i]$ for $0\leq i \leq 2$. 
(See the proof of Lemma 3.5 in \cite{GJ} for more details.) 
Since $\mathbf{\Theta}_n^{\Act}$ is an ordinary category, 
we need only define $\ds \G_n$ on objects and morphisms 
and check that composition is respected. 
\end{note}

\begin{proof} 
We proceed by induction on $n$.
For the base case $n=1$, we seek to define a functor $\G_1$
\[ \begin{tikzcd}
 \Exit\left(\Ran^{\tx{u}}\left(\un{\R}\right)\right) \ar[dashrightarrow]{r}{\G_1} 
 \ar{dr}[swap]{\phi_1} & \mathbf{\Delta}^{\Act} \ar{d}{\gamma_1} \\
& \Fin^\op 
\end{tikzcd} \]
contravariantly over finite sets. 

Let $S \overset{e}\hookrightarrow \R$ be an object in 
$\Exit\left(\Ranu\left(\un{\R}\right)\right)$. 
The value of $\G_1$ on $e$ is the linearly ordered set of connected components 
of the complement of $e\left(S\right)$ in $\R$  
$$\G_1: e \mapsto \pi_o\left(\R-e\left(S\right)\right)$$ 
the linear order of which is inherited from the linear order on $\R$. 

Let $\ds \cylr\left(T \xra{\sigma} S\right) \overset{E}\hookrightarrow \R \times \Delta^1$ 
be a morphism from $\ds S \xhra{e} \R$ to 
$\ds T \xhra{d} \R$ in $\Exit\left(\Ranu\left(\un{\R}\right)\right)$.  
Let $\ds C_{E}$ denote the compliment of the image of the embedding of $E$, 
$$\ds C_E := \left(\R \times \Delta^1\right) - E\left(\cylr\left(\sigma\right)\right).$$                                                                                                                                                                                                                                                                                                                                                                                                                                                                                                                                                                                                                                                                                                                                                                                                                                                                                                                                                                                                                                                                                                                                                                                                                                                                                                                                                                                                                                          Before we name the value of $\G_1$ on $E$, we need three observations:
\begin{itemize}
\item[1.] Consider the inclusion $ \ds \iota_1:\left(\R-d\left(T\right)\right) \hookrightarrow  C_E$ 
given by $\ds x \mapsto \left(x,\{1\}\right)$. 
Taking connected components induces an inclusion of sets 
$$\ds \pi_o\left(\iota_1\right): \pi_o\left(\R-d\left(T\right)\right) \hookrightarrow \pi_o\left(C_E\right).$$ 
It is easy to see that $\pi_o\left(\iota_1\right)$ is, in particular, a bijection. 
We denote its inverse $\ds \pi_o\left(\iota_1\right)^{-1}$.

\item[2.] Taking connected components of the inclusion 
$\ds \iota_0: \left(\R-e\left(S\right)\right) \hookrightarrow C_E$ given by 
$\ds x \mapsto \left(x,\{0\}\right)$ induces a map between sets 
$$\ds \pi_o\left(\iota_0\right): \pi_o\left(\R-e\left(S\right)\right) \hookrightarrow \pi_o\left(C_E\right).$$ 
Note that $\pi_o\left(\iota_0\right)$ is not necessarily injective nor surjective 
because $\sigma$ is not necessarily injective nor surjective. 

\item[3.] $\pi_o\left(\iota_1\right)$ determines a linear order on $\ds \pi_o\left(C_{E}\right)$ 
and thus, $\ds \pi_o\left(C_{E}\right)$ is an object in $\mathbf{\Delta}$.
\end{itemize}
The value of $\ds \G_1$ on $\ds\cylr\left(T \xra{\sigma} S\right) 
\overset{E}\hookrightarrow \R \times \Delta^1$ 
is defined to be the composite 
\begin{equation} \label{mors} \ds
\begin{tikzcd}
\pi_o\left(\R-e\left(S\right)\right) \arrow[swap]{dd}{\G_1\left(E\right)}  
\arrow{r}{\pi_o\left(\iota_0\right)} 
&  \pi_o\left(C_E\right) \arrow{ddl}{\pi_o\left(\iota_1\right)^{-1}}
\\ & 
\\ \pi_o\left(\R-d\left(T\right)\right)  & 
\end{tikzcd} 
\end{equation} 
in $\ds \mathbf{\Delta}^\Act$.

It must be checked that $\ds \G_1\left(E\right)$ is linear and active. 
We do this by verifying that each morphism in the composite is linear and active. 
$\ds \pi_o\left(\iota_1\right)^{-1}$ is a linear map 
because it defines the linear order of $\pi_o\left(C_E\right)$. 
Bijectivity of $\pi_o\left(\iota_1\right)^{-1}$ implies that it is active. 
Similarly, it is easy to see that $\pi_o\left(\iota_0\right)$ is order-preserving 
and sends unbounded components to unbounded components thereby being active.

Next we show that $\G_1$ respects composition by showing 
that the diagram of $\infty$-categories
\begin{equation} \label{finrelation}
\ds \begin{tikzcd} \Exit\left(\Ranu\left(\un{\R}\right)\right) 
\ar[r, dashrightarrow, "\G_1"] \ar[swap]{dr}{\phi_1} 
&  \mathbf{\Delta}^{\Act} \ar{d}{\gamma_1} \\ 
& \tx{Fin}^\op 
\end{tikzcd} 
\end{equation} 
commutes on the level of objects and morphisms.
Indeed, if (\ref{finrelation}) commutes, then faithfulness of $\gamma_1$ 
together with functorality of $\phi_1$ guarantee that $\G_1$ respects composition. 
More precisely, let 
\begin{equation} \label{triangle} 
\begin{tikzcd} 
a \ar{r}{f} \ar{rd}[swap]{h} & b \ar{d}{g} \\
 & c 
\end{tikzcd} 
\end{equation}
denote a commutative triangle in $\Exit\left(\Ranu\left(\un{\R}\right)\right)$. 
We show that if (\ref{finrelation}) commutes, then $\G_1$ carries the composite 
$h$ in (\ref{triangle}) to $\G_1\left(g\right) \circ \G_1\left(f\right)$. 
First, note that functorality of $\phi_1$ 
implies that $\phi_1\left(h\right) = \phi_1\left(g\right) \circ \phi_1\left(f\right)$. 
Commutativity of (\ref{finrelation}) guarantees that the morphisms 
$$\ds \gamma_1\left(\G_1\left(h\right)\right), 
\hspace{.1cm} \gamma_1\left(\G_1\left(f\right)\right) \hspace{.1cm} 
\tx{and} \hspace{.1cm} \gamma_1\left(\G_1\left(g\right)\right)$$ 
are equivalent (up to composition with canonical isomorphisms) to 
$$\phi_1\left(h\right), \hspace{.1cm} \phi_1\left(f\right) \hspace{.1cm} 
\tx{and}\hspace{.1cm} \phi_1\left(g\right)$$ 
respectively, in $\Fin$. 
Thus, 
$$\gamma_1\left(\G_1\left(h\right)\right)=\gamma_1\left(\G_1\left(g\right)\right) \circ 
\gamma_1\left(\G_1\left(f\right)\right) = \gamma_1\left( \G_1\left(g\right) \circ \G_1\left(f\right)\right).$$ 
Then faithfulness of $\gamma_1$ guarantees that 
$\G_1\left(h\right)=\G_1\left(g\right) \circ \G_1\left(f\right)$, as desired.

Now we verify commutativity of (\ref{finrelation}) on objects and morphisms. 
Let $\ds S \overset{e}\hookrightarrow \R$ be an object in 
$\ds \Exit\left(\Ranu\left(\un{\R}\right)\right)$. 
There is a canonical bijection of sets
\begin{equation} \label{gluing1} 
\ds  \gamma_1\left(\G_1\left(e\right)\right) \overset{\cong}\lra \phi_1\left(e\right):= S  
\end{equation} 
 in $\tx{Fin}$ given by 
 $$\ds \left(\pi_o\left(\R-e\left(S\right)\right) \overset{\alpha}\lra [1]\right) 
 \mapsto \tx{inf}\left\{x\in \coprod_{U \in \alpha^{-1}\left(\{1\}\right)} U \right\}$$ 
 verifying commutativity of (\ref{finrelation}) on objects.   
 
Let $\ds \cylr\left(T \overset{\sigma}\to S\right) \overset{E}\hookrightarrow \R \times \Delta^1$ 
be a morphism from $\ds S \xhra{e} \R$ to 
$\ds T \xhra{d} \R$ in $\Exit\left(\Ranu\left(\un{\R}\right)\right)$. 
We consider the canonical bijections of the source and target of 
$\ds \gamma_1\left(\G_1\left(E\right)\right)$, 
and the corresponding composite, $\alpha$, in $\Fin$ from $T$ to $S$:
\begin{equation} 
\begin{tikzcd} 
\gamma_1\left(\G_1\left(d\right)\right)  \ar{rr}{\gamma_1\left(\G_1\left(E\right)\right)} 
& & \gamma_1\left(\G_1\left(e\right)\right) \ar{d}[swap]{\cong}\\
T \ar{u}{\cong} \ar{rr}{\alpha} & &  S.
 \end{tikzcd} 
 \end{equation}
 By definition, the value of $\alpha$ on $\ds r\in T$ is 
 $$\ds \alpha\left(r\right):=\tx{inf}\left\{x\in \coprod_{U\in S^{r}} U\right\} $$ 
 where $ \ds S^{r} := \left\{U \in \pi_0\left(\R-e\left(S\right)\right) \mid
 \tx{inf}\{y \in \G_1\left(E\right)\left(U\right)\right\} \geq r \}$. 
 The composite $\alpha$ agrees with $\phi_1\left(E\right):=\sigma$, as desired. 
 Indeed, if $U\in S^{r}$, then 
 $$\tx{inf}\{x\in U\}=\sigma\left(r\right) \hspace{.1cm} \tx{or} \hspace{.1cm} 
 \tx{inf}\{x\in U\}=\sigma\left(r'\right),$$ 
 for some $r'>r$. But $\ds \sigma\left(r'\right) \geq \sigma\left(r\right)$ whenever $r' > r$, 
 which implies 
 $$\ds \tx{inf}\left\{x\in \coprod_{U\in S^{r}} U\right\} = \sigma\left(r\right).$$ 

In summary, we have shown that (\ref{finrelation}) commutes on objects and morphisms, 
which, as previously argued, implies that $\G_1$ respects composition. 
Therefore, $\G_1$ is a functor, and moreover is defined naturally over $\Fin^\op$.

By induction we assume the existence of $\G_{n-1}$ in the following diagram:
\[ \begin{tikzcd} 
\Exit\left(\Ranu\left(\un{\R}^{n-1}\right)\right) \ar{dr}[swap]{\Phi_{n-1}} \ar{r}{\G_{n-1}} 
& \mathbf{\Theta}_{n-1}^{\Act} \ar{d}{\tau_{n-1}} \\
& \Fun\left(\{1< \cdots < n-1\}, \Fin^\op\right).
\end{tikzcd} \]
In particular, this implies that $\G_{n-1}$ is over $\Fin^\op$ for each $\ds 1 \leq i \leq n-1$. 
In other words the following diagram commutes
\begin{equation} \label{iGn1fin} 
\begin{tikzcd}
\Exit\left(\Ranu\left(\un{\R}^{n-1}\right)\right) \ar{r}{\G_{n-1}} \ar[swap]{ddr}{\phi_i} 
& \mathbf{\Theta}_{n-1}^{\Act} \ar{d}{\tx{tr}_i} \\
 & \mathbf{\Theta}_i^{\Act} \ar{d}{\gamma_i} \\
 & \Fin^\op
 \end{tikzcd} 
 \end{equation} for each $1 \leq i \leq n-1$.
 Recall that $\tx{tr}_i$ denotes the $\left(n-1-i\right)$-fold self-composite 
 of the truncation map $\tr$ (Obs. \ref{tr}).

We take advantage of the definition of $\mathbf{\Theta}_n^\Act$ as a pullback (Def. \ref{act})
and define $\G_n$ by defining $\ds \Psi$ and $\ds \Gamma$ 
such that the following diagram of $\infty$-categories
\[ \ds
\begin{tikzcd} 
\Exit\bl(\Ranu(\un{\R}^n)\br) \arrow[drr, bend left, "\Psi"] 
\arrow[ddr, bend right, swap, "\Gamma"] \ar[dr, dashrightarrow, "\G_n"] && 
\\  & \mathbf{\Theta}_n^{\Act} \ar{r} \ar{d} 
& \Fin^\op \wr \mathbf{\Theta}_{n-1}^{\Act} \arrow{d}{\tx{frgt}} \\ 
& \mathbf{\Theta}_1^{\Act} \arrow{r} & \Fin^\op 
\end{tikzcd}
\]
commutes.
$\Gamma$ is defined to be the composite of the forgetful functor 
$\rho$ followed by $\G_1$, $\G_1 \circ \rho$, where $\rho$ is the functor 
from Observation~\ref{rho} which forgets all but the first coordinate data. 

$\Psi$ is defined to be the composite of the functor 
$$ \Exit\bl(\Ranu(\un{\R}^n)\br) \xra{\pi} 
\Fin^\op \wr \Exit\left(\Ranu\left(\un{\R}^{n-1}\right)\right)$$ 
from Observation~\ref{pie}, followed by the functor 
$$\ds \Fin^\op \wr \Exit\left(\Ranu\left(\un{\R}^{n-1}\right)\right) 
\ra \Fin^\op \wr \mathbf{\Theta}_{n-1}^{\Act}$$ 
determined by the identity on $\Fin^\op$ and the functor $\ds \G_{n-1}$ 
assumed by the inductive hypothesis. 
Thus, $\G_n$ is a well-defined functor.

Unwinding the above definition of $\G_n$, an inductive description of $\G_n$ is apparent. 
We explicate this inductive description on objects and morphisms as follows.
Let $\ds \un{S} \xhra{\un{e}} \un{\R}^n$ be an object in $\Exit\bl(\Ranu(\un{\R}^n)\br)$. 
Its value under $\G_n$ is inductively defined as 
$$ \G_n\left( \un{S} \xhra{\un{e}} \un{\R}^n\right) :=
\G_1\left(S_1 \xhra{e_1} \R\right)\left(\G_{n-1}\left(\left(\un{S}\right)_s 
\xhra{\un{{e}}_{|_{\left(\un{S}\right)_s}}} \un{\R}^{n-1}\right)\right)$$
where $ \ds \left(\un{S}\right)_s \xhra{\un{{e}}_{|_{\left(\un{S}\right)_s}}} \un{\R}^{n-1}$ 
denotes the image of an object (\ref{pbmor}) in 
$\Exit\left(\Ranu\left(\un{\R}^{n-1}\right)\right)$ under $\pi$.

Let $\ds \cylr\left(\un{T}  \xra{\un{\sigma}} \un{S}\right) \xhra{\un{E}} \un{\R}^n \times \Delta^1$ 
be a morphism from $ \ds \un{S} \xhra{\un{e}} \un{\R}^n$ to $\ds \un{T} \xhra{\un{d}} \un{\R}^n$  
in $\Exit\bl(\Ranu(\un{\R}^n)\br)$. 
Its value under $\G_n$ is inductively defined by
\begin{itemize} 
\item  the morphism $\ds \G_1\left(S_1 \xhra{e_1} \R\right) \xra{\G_1\left(\cylr\left(\sigma_1\right) 
\xhra{E_1} \R \times \Delta^1\right)} \G_1\left(T_1 \xhra{d_1} \R\right) $ in $\Delta^{\Act}$  
\item  for each pair $\left(t \in T_1,s \in S_1\right)$ such that $\ds \sigma_1\left(t\right)=s$, 
the morphism given by the image of (\ref{pbmor}) under $\G_{n-1}$ in $\mathbf{\Theta}_{n-1}^{\Act}$. 
\end{itemize} 

Next we will show that for each $1 \leq i \leq n$, the diagram of $\infty$-categories
\begin{equation} \label{iGnfin} 
\begin{tikzcd}
\Exit\bl(\Ranu(\un{\R}^n)\br) \ar{r}{\G_n} \ar[swap]{ddr}{\phi_i} 
& \mathbf{\Theta}_n^{\Act} \ar{d}{\tx{tr}_i} \\
 & \mathbf{\Theta}_i^{\Act} \ar{d}{\gamma_i} \\
 & \Fin^\op
 \end{tikzcd} 
 \end{equation}
 commutes.
For the cases $1 \leq i \leq n-1$, this diagram follows 
by the inductive hypothesis wherein we assume commutativity of (\ref{iGn1fin}). 
For the remaining case, $i=n$, we use the inductive definitions of $\G_n$ and 
$\gamma_n$ in terms of $\G_1$ and $\G_{n-1}$, and $\gamma_1$ and $\gamma_{n-1}$, respectively. 
Then indeed, in employing the commutativity of (\ref{finrelation}) and (\ref{iGn1fin}) for $i=n-1$, 
we see that for the case $i=n$, (\ref{iGnfin}) must commute. 
Through Observation~\ref{phii} wherein the functor $\Phi_n$ 
was defined in terms of $\phi_i$ for $1 \leq i \leq n$, 
commutativity of this diagram for each $1 \leq i \leq n$ 
compiles to prove that $\G_n$ is over $\Fun\left(\{1 < \cdots < n\}, \Fin^\op\right)$, 
completing the proof.
\end{proof}

Next we show that $\G_n$ from Lemma~\ref{functor} is an equivalence of $\infty$-categories. 
\begin{theorem} \label{inteq} 
For each $n \geq 1$, the functor 
$$\G_n: \Exit\bl(\Ranu(\un{\R}^n)\br) \overset{ \simeq}\lra \mathbf{\Theta}_n^{\tx{act}}$$ 
over $\Fun\left(\{1< \cdots < n\}, \Fin^\op\right)$
is an equivalence of $\infty$-categories.
\end{theorem}

\begin{proof} 
We proceed by induction on $n$.
For the base case $n=1$ we show that $\G_1$ is essentially surjective and fully faithful -- 
the former follows easily. 
Let $\ds [p] \in  \mathbf{\Delta}^{\Act}$. 
Define the set $\ds T_p := \{1,2,...,p\}$ together with the object 
$ \ds T_p \overset{d}\hookrightarrow \R $ in $\Exit\left(\Ranu\left(\un{\R}\right)\right)$, 
given by $\ds i \mapsto i$. Then $\ds [p]$ is isomorphic to 
$\ds \G_1\left(T_p\right) :=\pi_o\left(\R-d\left(T_p\right)\right)$ in $\mathbf{\Delta}$, 
with the isomorphism given by $\ds i \mapsto \left[i+\frac{1}{2}\right]$. 

Fix a pair of objects $S \xhra{e} \R$ and $T \xhra{d} \R$ in 
$\Exit\left(\Ranu\left(\un{\R}\right)\right)$. 
Showing fully faithfulness of $\G_1$ amounts to showing that the map 
induced by $\G_1$ between corresponding hom-spaces 
\begin{equation}\label{arrow} 
\ds \tx{Hom}_{\Exit\left(\Ranu\left(\un{\R}\right)\right)}\left(e, d\right) 
\xra{\G_1} \tx{Hom}_{\mathbf{\Delta}^{\Act}}\left(\pi_o\left(\R-e\left(S\right)\right), 
\pi_o\left(\R-d\left(T\right)\right)\right) 
\end{equation} 
 is a surjection on connected components with contractible fibers.  
Fix a morphism 
$$\ds \pi_0\left(\R-e\left(S\right)\right) \xra{\varphi} 
\pi_0\left(\R-d\left(T\right)\right)$$ 
in $\ds \Delta^{\Act}$. Any morphism 
\begin{equation} \label{A} 
\ds \cylr\left(T \xra{\gamma_1\left(\varphi\right)} S\right) 
\xhookrightarrow{E} \R \times \Delta^1 
\end{equation}  
 in $ \ds \tx{Hom}_{\Exit\left(\Ranu\left(\un{\R}\right)\right)}\left(e, d\right)$ 
 is in the fiber of $\G_1$ over $\varphi$. 
 Indeed, in Observation~\ref{injective} we saw that $\ds \gamma_1$ 
 is injective on hom-sets. Thus, commutativity of (\ref{finrelation}) guarantees 
 that $\ds E$ is in the fiber of $\G_1$ over $\varphi$. Hence, (\ref{arrow}) 
 is a surjection on connected components. 
 
The fiber of (\ref{arrow}) over $\varphi$ is the topological space of 
embeddings $\ds \cylr\left(\gamma_1 \circ \varphi\right) \xhra{E} \R \times \Delta^1$ 
over $\Delta^1$ such that $E_{|S}=e$ and $E_{|T\times\{1\}} = d$, 
which we denote by 
$$\G_1^{-1}\left(\varphi\right) \cong 
\tx{Emb}^{e,d}_{/\Delta^1}\left(\cylr\left(\gamma_1\left(\varphi\right)\right), \R\times \Delta^1\right)$$   
under the compact-open topology. 
We want to show that this space is contractible. 
Fix an embedding $\tilde{E}$ in the fiber of $\G_1$ over $\varphi$. 
Let $\ds \mathbb{S}^k \xra{\psi} \G_1^{-1}\left(\varphi\right)$ be continuous 
and based at $\ds \tilde{E}$. 
We construct a null-homotopy of $\psi$. 
For each $\ds z \in \s^n$, denote the image of $z$ under $\psi$ by $\psi_z.$ 
The straight-line homotopy, $H_z$, from $\psi_z$ to $\tilde{E}$ 
defined by 
$$ H_z\left(x,t\right) = \left(1-t\right)\psi_z\left(x\right) + t\tilde{E}\left(x\right)$$ 
names a path from $\psi_z$ to $\tilde{E}$ in $\G_1^{-1}\left(\varphi\right)$. 
For each $z \in \s^k$, we let each path $H_z$ run simultaneously to name 
a null-homotopy of $\psi$ to the constant path at $\{\tilde{E}\}$. 
Explicitly, the null-homotopy $ \s^k \times [0,1] \ra \G_1^{-1}\left(\varphi\right)$ 
is given by $ \left(z,t\right) \mapsto H_z\left(-,t\right)$. 

For the inductive step we want to show that $\ds \G_n$ is essentially surjective and fully faithful
if $\G_{n-1}$ is. 
Let $\ds [k]\left(T_s\right)$ be an object in $\ds  \mathbf{\Theta}_n^{\Act}$. 
Because $\G_1$ is essentially surjective, 
we may choose an object of $\Exit\left(\Ranu\left(\un{\R}\right)\right)$
\begin{equation} \label{1obj} 
\ds \{1,...,k\} \xhookrightarrow{e} \R
\end{equation}
in the fiber of $\G_1$ over $[k]$.
 Likewise, by essential surjectivity of $\G_{n-1}$, 
 for each $ s \in \{1,...,k\}$, we may choose an object of 
 $\Exit\left(\Ranu\left(\un{\R}^{n-1}\right)\right)$ that is in the fiber of $\G_{n-1}$ over $T_s$:
\begin{equation} \label{iobj} \ds
\begin{tikzcd}
\left(S_{n-1}\right)_s \ar[hookrightarrow]{rr}{\left(e_{n-1}\right)_s} 
\ar{d}[swap]{\left(\tau_{n-2}\right)_s}  && \R^{n-1}  \ar[twoheadrightarrow]{d}{\tx{pr}_{< n-1} } \\
\left(S_{n-2}\right)_s \ar[hookrightarrow]{rr}{\left(e_{n-2}\right)_s} 
\ar{d}[swap]{\left(\tau_{n-3}\right)_s}   && \R^{n-2}  \ar[twoheadrightarrow]{d}{\tx{pr}_{< n-2} } \\
\vdots \ar{d} &&  \vdots \ar[twoheadrightarrow]{d} \\
\left(S_1\right)_s \ar[hookrightarrow]{rr}{\left(e_1\right)_s}   && \R .
\end{tikzcd}
 \end{equation} 
The choices (\ref{1obj}) and (\ref{iobj}) for each $s$
uniquely determine an object of $\ds \Exit\bl(\Ranu(\un{\R}^n)\br)$ 
that is in the fiber of $\ds \G_n$ over $\ds [k]\left(T_s\right)$:
\begin{equation} \label{obj} \ds
\begin{tikzcd}
\ds \coprod_{1 \leq s \leq k}\left(S_{n-1}\right)_s 
\ar[hookrightarrow]{rrr}{\coprod\{e\left(s\right)\}\times \left(e_{n-1}\right)_s  } 
\ar{d}[swap]{ \coprod\left(\tau_{n-2}\right)_s}  &&& \R \times \R^{n-1}  
\ar[twoheadrightarrow]{d}{\tx{pr}_{< n} } \\
\ds \coprod_{1 \leq s \leq k} \left(S_{n-2}\right)_s 
\ar[hookrightarrow]{rrr}{ \coprod\{e\left(s\right)\}\times \left(e_{n-2}\right)_s } 
\ar{d}[swap]{ \coprod\left(\tau_{n-3}\right)_s}  &&& \R \times \R^{n-2}  
\ar[twoheadrightarrow]{d}{\tx{pr}_{< n-1} } \\
\vdots \ar{d} &&&  \vdots \ar[twoheadrightarrow]{d} \\
\ds \coprod_{1 \leq s \leq k} \left(S_1\right)_s 
\ar{d}[swap]{\coprod\{s\}} \ar[hookrightarrow]{rrr}{\coprod\{e\left(s\right)\} \times \left(e_1\right)_s }   
&&& \R \times \R \ar[twoheadrightarrow]{d}{\tx{pr}_{<2}}\\
 \{1,...,k\} \ar[hookrightarrow]{rrr}{e} &&& \R
\end{tikzcd} 
\end{equation}
where each map defined in terms of a coproduct is indexed 
over $1 \leq s \leq k$, and $\{e\left(s\right)\}$ and 
$\{s\}$ denote the constant maps at $e\left(s\right)$ and $s$, respectively. 
 
Fix a pair of objects $ \ds \un{T} \xhra{\un{d}} \un{\R}^n$ 
and $\ds \un{S} \xhra{\un{e}} \un{\R}^n$ in $\Exit\bl(\Ranu(\un{\R}^n)\br)$. 
We show fully faithfulness of $\G_n$ by showing that the map induced by $\G_n$ between hom-spaces
\begin{equation}\label{arrown} 
\ds \tx{Hom}_{\Exit\bl(\Ranu(\un{\R}^n)\br)}\left( \un{e}, \un{d}\right) 
\xra{\G_n} \tx{Hom}_{\mathbf{\Theta}_n^{\Act}}\left(\G_n\left( \un{e}\right), \G_n\left(\un{d}\right)\right) 
\end{equation}
is a surjection on connected components with contractible fibers. 
Fix a morphism  
$$ \ds  \G_n\left(\un{e}\right)\xra{\varphi}  \G_n\left(\un{d}\right) $$ 
in $\mathbf{\Theta}_n^{\Act}$. 
Using the inductive description of $\G_n$, $\varphi$ is given by: 
\begin{itemize} 
\item   a morphism $\ds \G_1\left(S_1 \xhra{e_1} \R\right)  
\xra{\varphi_1}  \G_1\left(T_1 \xhra{d_1} \R\right) $ in $\mathbf{\Delta}^{\Act}$  
\item  for each pair $\left(r \in T_1, s \in S_1\right)$ such that 
$\ds \gamma_1\left(\varphi_1\left(r\right)\right)=s$, 
a morphism 
$$ \G_{n-1}\left(\left(\un{S}\right)_s \xhra{\un{{e}}_{|_{\left(\un{S}\right)_s}}} \un{\R}^{n-1}\right) \xra{\varphi_r}     \G_{n-1}\left(\left(\un{T}\right)_r \xhra{\un{{d}}_{|_{\left(\un{T}\right)_r}}} \un{\R}^{n-1}\right) $$
 in $\mathbf{\Theta}_{n-1}^{\Act}$.
\end{itemize}
Using the base case and the inductive hypothesis, 
we define a morphism that is in the fiber of $\G_n$ over $\varphi$.
By fullness of $\G_1$, we may choose a morphism  in the fiber of $ \G_1$ over $\varphi_1$
\begin{equation} \label{first} 
\cylr\left(\gamma_1\left(\varphi_1\right)\right) \xhookrightarrow{E_1} \R \times \Delta^1 
\end{equation}
which is defined over the map of finite sets $T_1 \xra{\gamma_1\circ \varphi_1} S_1$. 

By fullness of $\G_{n-1}$ as assumed by the inductive hypothesis, 
for each pair $\left(r \in T_1, s \in S_1\right)$ such that 
$\ds \gamma_1\left(\varphi_1\left(r\right)\right)=s$, 
we may choose a morphism in the fiber of $\G_{n-1}$ over $\varphi_r$
\begin{equation}\label{leveldownmor}
\begin{tikzcd} 
{\sf cylr}\left(\gamma_{n-1} \circ \varphi_r\right) 
\ar[hookrightarrow]{rr}{\left(E_n\right)_r} \ar{d} & & 
\R^{n-1} \times \Delta^1 \ar[twoheadrightarrow]{d}{\tx{pr}_{< n-1}\times \tx{id}_{\Delta^1}} \\
{\sf cylr}\left(\gamma_{n-2} \circ \tr_{n-2} \circ \varphi_r\right) 
\ar[hookrightarrow]{rr}{\left(E_{n-1}\right)_r} \ar{d}  && 
\R^{n-2} \times \Delta^1 \ar[twoheadrightarrow]{d}{\tx{pr}_{< n-2}\times \tx{id}_{\Delta^1}}  \\
\vdots \ar{d} & & \vdots \ar[twoheadrightarrow]{d} \\
{\sf cylr}\left(\gamma_1 \circ \tr_1 \circ \varphi_r\right) \ar[hookrightarrow]{rr}{\left(E_{2}\right)_r}   
&& \R \times \Delta^1
\end{tikzcd}  
\end{equation} 
Note that (\ref{iGn1fin}) guarantees that (\ref{leveldownmor}) 
must be defined over the diagram of finite sets
\begin{equation}
\begin{tikzcd} \label{sets}
\left(T_{n}\right)_r \ar{rrr}{\gamma_{n-1} \circ \varphi_r} 
\ar{d}[swap]{{\omega_{n-1}}_{|_\bullet}} &&& 
\left(S_{n}\right)_{s} \ar{d}{{\tau_{n-1}}_{|_\bullet}} \\
\left(T_{n-1}\right)_r \ar{rrr}{\gamma_{n-2} \circ \tr_{n-1}\circ \varphi_r} 
\ar{d}[swap]{{\omega_{n-2}}_{|_\bullet}} &&& \left(S_{n-1}\right)_{s} \ar{d}{{\tau_{n-2}}_{|_\bullet}} \\
\vdots \ar{d} & && \vdots \ar{d} \\
\left(T_{2}\right)_r \ar{rrr}{\gamma_{1} \circ \tr_{1}\circ \varphi_r}  &&& \left(S_{2}\right)_{s}. 
\end{tikzcd} 
\end{equation} 

Using (\ref{first}) and (\ref{leveldownmor}), we define a morphism 
in the fiber of $\ds \G_n$ over $\varphi$
\begin{equation} \label{cylmor}
\begin{tikzcd}
\ds {\sf cylr}\left(\coprod_{r \in T_1}\gamma_{n-1} \circ \varphi_r\right)  
\ar[hookrightarrow]{rrr}{\coprod\{E_1\left(r\right)\} \times \left(E_n\right)_r} \ar{d}  
&&& \left(\R \times \R^{n-1}\right) \times \Delta^1 
\ar[twoheadrightarrow]{d}{\tx{pr}_{< n} \times \tx{id}_{\Delta^1}} \\
\ds {\sf cylr}\left(\coprod_{r \in T_1}\gamma_{n-2}\circ \tr_{n-2} \circ \varphi_r\right) 
\ar[hookrightarrow]{rrr}{\coprod\{E_1\left(r\right)\} \times \left(E_{n-1}\right)_r} \ar{d}  
&&& \left(\R \times \R^{n-2}\right) \times \Delta^1 
\ar[twoheadrightarrow]{d}{\tx{pr}_{< n-1} \times \tx{id}_{\Delta^1}} \\
\vdots \ar{d} &&& \vdots \ar[twoheadrightarrow]{d} \\
\ds {\sf cylr}\left(\coprod_{r \in T_1}\gamma_1 \circ \tr_1 \circ \varphi_r\right) 
\ar[hookrightarrow]{rrr}{\coprod\{E_1\left(r\right)\} \times \left(E_{2}\right)_r} \ar{d}  
&&& \left(\R \times \R\right) \times \Delta^1 
\ar[twoheadrightarrow]{d}{\tx{pr}_{< 2} \times \tx{id}_{\Delta^1}}\\
 {\sf cylr}\left(\gamma_1\left(\varphi_1\right)\right) \ar[hookrightarrow]{rrr}{E_1} 
 &&& \R \times \Delta^1
\end{tikzcd}
\end{equation} 
where $\ds \{r\} $ and $\{E_1\left(r\right)\}$ 
denote the constant map at $r$ and $E_1\left(r\right)$, respectively, 
and (\ref{cylmor}) is defined over the diagram of finite sets
\begin{equation}
\begin{tikzcd} 
\ds T_n = \coprod_{r \in T_1} \left(T_n\right)_r 
\ar{rrr}{\coprod \gamma_{n-1} \circ \varphi_r} \ar{d}[swap]{\omega_{n-1}} 
&&& \ds \coprod_{s \in S_1} \left(S_n\right)_s= S_n  \ar{d}{\tau_{n-1}} \\
 \ds T_{n-1} = \coprod_{r \in T_1} \left(T_{n-1}\right)_r 
 \ar{rrr}{\coprod \gamma_{n-2} \circ \tr_{n-2} \circ \varphi_r} \ar{d}[swap]{\omega_{n-2}} 
 &&& \ds \coprod_{s \in S_1} \left(S_{n-1}\right)_s= S_{n-1} \ar{d}{\tau_{n-2}}  \\
\vdots \ar{d} &&& \vdots \ar{d} \\
\ds T_2 = \coprod_{r \in T_1} \left(T_2\right)_r 
\ar{rrr}{\coprod \gamma_{1} \circ \tr_1 \circ \varphi_r} \ar{d}[swap]{\omega_1=\coprod \{r\}} 
&&& \ds  \coprod_{s \in S_1} \left(S_2\right)_s= S_2  \ar{d}{\coprod \{s\}= \tau^1} \\
\ds T_1= \coprod_{r\in T_1}r \ar{rrr}{\gamma_1 \circ \varphi_1} 
&&& \ds \coprod_{s \in S_1} s = S_1.
\end{tikzcd} 
\end{equation}

Lastly, we wish to show that each fiber of (\ref{arrown}) is contractible. 
The fiber of $\G_n$ in (\ref{arrown}) over $\varphi$  is, 
under the compact-open topology, the topological space of compatible embeddings 
\begin{equation} \label{emb}\ds
\begin{tikzcd}
\tx{cyl}\left(\gamma_n \circ \varphi\right) \ar[hookrightarrow]{r}{E_n} \ar{d}  
& \R^n \times \Delta^1 \ar[twoheadrightarrow]{d}{\tx{pr}_{< n} \times \tx{id}} \\
\tx{cyl}\left(\gamma_{n-1}' \circ \tx{tr}_{n-1}\circ \varphi\right) \ar[hookrightarrow]{r}{E_{n-1}} \ar{d}  
& \R^{n-1} \times \Delta^1 \ar[twoheadrightarrow]{d}{\tx{pr}_{< n-1} \times \tx{id}} \\
\vdots \ar{d} &  \vdots \ar[twoheadrightarrow]{d} \\
\tx{cyl}\left(\gamma_1 \circ \tx{tr}_1 \circ \varphi\right) \ar[hookrightarrow]{r}{E_1}   
& \R \times \Delta^1
\end{tikzcd}
\end{equation} 
over $\Delta^1$ such that $E_{|S_n}=e_n$ and $E_{|T_n \times \{1\}}=d_n$. 
Note that (\ref{emb}) guarantees that each morphism 
in $\G_n^{-1}\left(\varphi\right)$ is defined over the diagram of finite sets
\begin{equation} 
\begin{tikzcd} 
\ds T_n  \ar{rr}{ \gamma'_{n} \circ \varphi} 
\ar{d}[swap]{\omega_{n-1}} && \ds  S_n  \ar{d}{\tau_{n-1}} \\
 \ds T_{n-1}  \ar{rr}{ \gamma'_{n-1} \circ \tr_{n-1} \circ \varphi} 
 \ar{d}[swap]{\omega_{n-2}} && \ds  S_{n-1} \ar{d}{\tau_{n-2}}  \\
\vdots \ar{d} && \vdots \ar{d} \\
\ds T_1 \ar{rr}{\gamma_1 \circ \tr_1 \circ \varphi} && \ds  S_1.
\end{tikzcd} 
\end{equation}
Fix an embedding $E$ in the fiber of $\G_n$ over $\varphi$. 
Let $\ds \mathbb{S}^k \xra{\psi} \G_n^{-1}\left(\varphi\right)$ be continuous and based at $\ds E$. 
We construct a null-homotopy of $\psi$.
For each $\ds z \in \s^n$, denote the image of $z$ under $\psi$ by $\psi_z.$ 
The straight-line homotopy, $H_z$, from $\psi_z$ to $E$ defined by 
$$ H_z\left(x,t\right) = \left(1-t\right)\psi_z\left(x\right) + tE\left(x\right)$$ 
names a path from $\psi_z$ to $E$ in $\G_n^{-1}\left(\varphi\right)$. 
For each $z \in \s^k$, we let each path $H_z$ run simultaneously 
to name a null-homotopy of $\psi$ to the constant path at $\{E\}$. 
Explicitly, the null-homotopy $ \s^k \times [0,1] \ra \G_n^{-1}\left(\varphi\right)$ 
is given by $ \left(z,t\right) \mapsto H_z\left(-,t\right)$. 
\end{proof}

\section{Part 2 of the proof of the main result}\label{step2}

This section contains the technical heart of the proof of Theorem~\ref{loc}.
The goal is to prove that the natural forgetful functor from Observation~\ref{forget}
\begin{equation} \label{up}
\Exit\bl(\Ranu(\un{\R}^n)\br) \to \Exit\bl(\Ranu(\R^n)\br)
\end{equation}
which forgets all but the $n$-dimensional data is a localization. 
We begin by defining a localization of $\infty$-categories.
 
\begin{definition} \label{localization} 
Let $\C$ be an $\infty$-category and let $W$ be a $\infty$-subcategory of $\C$ 
which contains the maximal $\infty$-subgroupoid $\C^\sim$ of $\C$. 
The \emph{localization} of $\C$ on $W$ is an $\infty$-category $\C[W^{-1}]$ 
and a functor $\C \xra{L} \C[W^{-1}]$ satisfying the following universal property: 
For any $\infty$-category $\D$, any functor $F$ from $\C$ to $\D$ 
uniquely factors through $L$ if and only if $F$ maps each morphism in $W$ 
to an isomorphism in $\D$; otherwise, there is no filler, as indicated by the following diagram:
\[ \begin{tikzcd}
\C \ar{r}{F} \ar[swap]{d}{L} & \D \\
\C[W^{-1}]. \ar[dashrightarrow]{ur}[swap]{\exists ! \shs \tx{or} \shs \emptyset} & 
\end{tikzcd} \]
We call $W$ the \emph{localizing $\infty$-subcategory} of $\C \to \C[W^{-1}]$.
\end{definition}

Heuristically, a localization $\C[W^{-1}]$ is the $\infty$-category 
that results after formally inverting all of the morphisms in $W \subset \C$. 

The localizing $\infty$-subcategory of the functor (\ref{up}) is the full $\infty$-subcategory of 
$\Exit\bl(\Ranu(\un{\R}^n)\br)$ consisting of all the morphisms that are sent to 
isomorphisms in $\Exit\bl(\Ranu(\R^n)\br)$. 
Recall that the isomorphisms of $\Exit\bl(\Ranu(\R^n)\br)$
are all those that are contained in a configuration space of fixed cardinality (Obs. \ref{neat}).
A morphism in $\Exit\bl(\Ranu(\un{\R}^n)\br)$ is sent to one of these isomorphisms
if and only if it too is contained in a configuration space of fixed cardinality. 
We formally define the localizing $\infty$-subcategory of functor (\ref{up}) as follows. 

\begin{definition} \label{loc subcat}
Let $W_n$ to be the $\infty$-subcategory of 
$\Exit\bigl(\Ranu(\un{\R}^n)\bigr)$ defined to be the pullback
\[ \begin{tikzcd}
W_n  \pb \ar[hookrightarrow]{r} \ar{d} & \Exit\bigl(\Ranu(\un{\R}^n)\bigr) \ar{d}{\Phi_n} \\
\Fun^{\tx{n-bij}}\bigl(\{1< \cdots <n\}, \Fin^\op\bigr) 
\ar[hookrightarrow]{r} & \Fun\bigl(\{1<\cdots<n\}, \Fin^\op\bigr)
\end{tikzcd} \] 
where $\ds \Fun^{\tx{n-bij}}\bigl(\{1< \cdots <n\}, \Fin^\op\bigr)$ 
is the subcategory of $\ds \Fun\bigl(\{1<\cdots<n\}, \Fin^\op\bigr)$ 
consisting of all the same objects and only those morphisms whose value 
under evaluation at $n$ is a bijection; 
$\Phi_n$ is the forgetful functor from Observation~\ref{phii} 
that simply remembers the underlying data of sets at each level $1 \leq i \leq n$. 
\end{definition}

\begin{lemma} \label{preloc} 
The forgetful functor from Observation~\ref{forget} 
$$ \F: \Exit\bigl(\Ranu(\un{\R}^n)\bigr) \lra \Exit\bigl(\Ranu(\R^n)\bigr)$$
is a localization of $\infty$-categories on the $\infty$-subcategory 
$W_n$ defined in Definition~\ref{loc subcat}.
\end{lemma}

\begin{notation} 
Whenever convenient, we will abuse notation 
and refer to $W_n$ as the image of $W_n$ in $\mathbf{\Theta}_n^\Act$ 
under the functor from (Lem. \ref{functor}) which was shown to be an equivalence in Theorem~\ref{inteq}. 
\end{notation}

Heuristically, $W_n$ has the same objects as $\Exit\bigl(\Ranu(\un{\R}^n)\bigr)$ 
and all those morphisms whose values under $\phi_n$ from Observation~\ref{phii} are bijections. 
Intuitively then, $\Exit\bigl(\Ranu(\un{\R}^n)\bigr)$ 
localizing on $W_n$ to $\Exit\bigl(\Ranu(\R^n)\bigr)$ is no surprise. 
Indeed, in formally declaring all those morphisms in $W_n$ to be isomorphisms, 
we forget the Fox-Neuwirth cells, which define morphisms in $\Exit\bigl(\Ranu(\un{\R}^n)\bigr)$  
and only remember cardinality, which is the defining restriction of morphisms 
in $\Exit\bigl(\Ranu(\R^n)\bigr)$. However intuitive, 
our procedure for showing Lemma~\ref{preloc} is in fact quite technical.
To explain our approach, we first need the following definition.

\begin{definition} \label{funW} 
Given an $\infty$-category $\C$ and an $\infty$-subcategory 
$W \hra \C$, $ \ds \Fun^{W}\bigl([p], \C\bigr)$ is defined to be the pullback of $\infty$-categories
\[ \begin{tikzcd} 
\Fun^{W}\bigl([p],\C\bigr) \ar{r} \pb \ar{d} & \Fun\bigl([p],\C\bigr) \ar{d} \\
\Fun\bigl([p]^{\sim}, W \bigr) \ar[hookrightarrow]{r} & \Fun\bigl([p]^{\sim}, \C\bigr) 
 \end{tikzcd} \] 
where $[p]^\sim$ denotes the underlying maximal $\infty$-subgroupoid of $[p]$.
\end{definition}

\begin{observation} \label{one} 
In the case $p=0$, $\ds \Fun^{W}\bigl([0],\C\bigr)$ is  equivalent to $W$. 
Indeed, an object is a functor $[0] \to \C$ selects out an object of $W$, 
which is precisely an object of $\C$; a morphism is a  
natural transformation between any two such functors, 
which is precisely determined by a morphism in $W$.  

A similar examination of the $p=1$ case identifies that 
$\ds \Fun^{W}\bigl([1],\C\bigr)$ is the $\infty$-category whose objects 
are morphisms of $\C$ and whose morphisms are all those natural transformations 
given by morphisms in $W$, i.e., a morphism from $c \to d$ in $\C$ to $c' \to d'$ in $\C$  
is a commutative square in $\C$
\[ \begin{tikzcd}
c \ar{r} \ar{d} & c' \ar{d} \\
d \ar{r} & d' \end{tikzcd} \]
such that both horizontal arrows are morphisms in $W$.  
In general, a $[p]$-point in $\Fun^{W}\left([1], \C \right)$ is a commutative diagram 
in $\C$ of the shape $[p] \times [1]$ such that the two $p$-simplicies $[p] \cong [p]\times \{0\}$ 
and $[p] \cong [p] \times \{1\}$ must be $[p]$-points of the $\infty$-subcategory $W$. 
\end{observation}

In \cite{AMG2} Mazel-Gee proved a theorem that identifies 
localizations of $\infty$-categories in favorable cases. 
It will be our route for proving Lemma~\ref{preloc}. 
We state it next.

\begin{theorem}[Thm. 3.8, \cite{AMG2}] \label{BcSS} 
For an $\infty$-category $\C$ and a $\infty$-subcategory containing 
the maximal $\infty$-subgroupoid of $\C$, $\C^\sim \subset W \subset \C$, 
if the classifying space $\ds\B  \Fun^W\bigl([\bullet], \C\bigr)$ is a complete Segal space, 
then it is equivalent as a simplicial space to the localization of $\C$ on $W$, 
$$ \B \Fun^W\bl([\bullet], \C\br) \simeq \C[W^{-1}].$$ 
\end{theorem}

Thus, to prove Lemma~\ref{preloc} we must show that the simplicial space 
$\ds \B {\sf Fun}^{W_n}\Bigl([\bullet],\Exit\bigl(\Ranu(\un{\R}^n)\bigr)\Bigr)$ 
is equivalent to $\Exit\bigl(\Ranu(\R^n)\bigr)$.
To do this we need the following two lemmas. 

\begin{lemma} \label{class1} 
For $p=0, 1$, there is an equivalence of spaces 
\[ \B\tx{Fun}^{W_n}\Bigl([p],\Exit\bigl(\Ranu(\un{\R}^n)\bigr)\Bigr) 
\simeq \tx{Hom}_{\Cat_{\infty}}\Bigl([p],\Exit\bigl(\Ranu(\R^n)\bigr)\Bigr) \]
 between the classifying space of the $\infty$-category 
 ${\sf Fun}^{W_n}\Bigl([p],\Exit\bigl(\Ranu(\un{\R}^n)\bigr)\Bigr)$ 
 and the hom-space in $\infty$-categories from $[p]$ to $\Exit\bigl(\Ranu(\R^n)\bigr)$.
 \end{lemma}


\begin{lemma} \label{class2} 
The classifying space 
$\ds \B {\sf Fun}^{W_n}\Bigl([\bullet],\Exit\bigl(\Ranu(\un{\R}^n)\bigr)\Bigr)$ 
is a complete Segal space. 
\end{lemma}

Let us explain how it is that Lemma~\ref{preloc} follows from these two lemmas; 
we will give a formal proof at the end of this section. 
First, note that Lemma~\ref{class2} verifies the hypothesis of Theorem~\ref{BcSS}. 
In particular then, the classifying space 
$\B\tx{Fun}^{W_n}\Bigl([\bullet],\Exit\bigl(\Ranu(\un{\R}^n)\bigr)\Bigr)$  
is determined by its values on $[0]$ and $[1]$ because it satisfies 
the Segal condition (Def.~\ref{cSegspc}). 
Through Theorem~\ref{BcSS}, Lemma~\ref{class1} identifies the space of 
objects and morphisms of the localization $\Exit\bigl(\Ranu(\un{\R}^n)\bigr)[W_n]$ 
as the space of objects and morphisms of $\Exit\bigl(\Ranu(\R^n)\bigr)$, respectively.  
The Segal condition, then, implies the desired result, 
that the localization $\Exit\bigl(\Ranu(\un{\R}^n)\bigr)[W_n^{-1}]$ 
is equivalent to the $\infty$-category $\Exit\bigl(\Ranu(\R^n)\bigr)$. 

We organize our proofs of Lemma~\ref{class1} and Lemma~\ref{class2} as follows. 
Lemma~\ref{class1} naturally decomposes into its two cases, $p=0$ and $p=1$; 
we make each case into a lemma, each of which is stated and proven in 
the subsequent subsections \S\ref{0} and \S\ref{1}. 
And lastly, \S\ref{oh} is devoted to the proof of Lemma~\ref{class2}.

\subsection{Identifying the space of objects of the localization} \label{0}
This subsection is devoted to proving the $p=0$ case of Lemma~\ref{class1}.
In light of Observation~\ref{one} wherein we observed that $\ds \Fun^{W}\bl([0],\C\br) \simeq W$, 
we rephrase this case as the following lemma.

\begin{lemma}[Lemma~\ref{class1}; $p=0$] \label{pis0} 
There is an equivalence of spaces 
$$ \B W_n \simeq \Exit\bigl(\Ranu(\R^n)\bigr)^\sim$$ 
from the classifying space of the $\infty$-subcategory 
$W_n$ of $\Exit\bigl(\Ranu(\un{\R}^n)\bigr)$ to the maximal 
$\infty$-subgroupoid of $\Exit\bigl(\Ranu(\R^n)\bigr)$.
 \end{lemma}

To prove Lemma~\ref{pis0}, we need two lemmas. 
The first states that there is an adjunction between $W_n$ 
and the subcategory of $W_n$ consisting of healthy trees (Def. \ref{tree}). 
Explicitly, this subcategory is defined as follows.
 
\begin{definition}  \label{Whealthy} 
$W_n^\h$ is the subcategory of $W_n$ defined to be the pullback 
\[ \begin{tikzcd}
W_n^\h \pb \ar[hookrightarrow]{r} \ar{d} & W_n \ar{d} \\
\Fun^{\tx{n-bij}}\Bigl(\{1< \cdots <n\}, \bl(\Fin^{\surj}\br)^\op\Bigr) 
\ar[hookrightarrow]{r} & \Fun^{\tx{n-bij}}\bl(\{1<\cdots<n\}, \Fin^\op\br)
\end{tikzcd} \] 
of categories. 
\end{definition}

Informally, the category $W_n^\h$ is the full subcategory of $W_n$ 
consisting of all those objects of $\mathbf{\Theta}_n$ that are healthy trees. 

\begin{lemma} \label{adj} 
The inclusion functor $W_n^{\h} \hookrightarrow W_n$ is a right adjoint. 
\end{lemma}

A left adjoint to the inclusion functor is given by forgetting all of the data 
associated to the `unhealthy' parts of the trees in $W_n$. 
We construct such a functor as follows.

\begin{construction}[The Pruning Functor, $P_n$] 
For each $n \geq 1$, we define a canonical functor 
$$P_n:W_n \to W_n^\h.$$

 For $n=1$, $P_1:=\tx{Id}_{W_1}$ since $W_1 = W_1^\h$.

 For $n\geq 2$, we define $P_n$ inductively. 
 First, for each object $T=[p](T_i) \in W_n$ for $n \geq 2$, 
 define the sub-linearly ordered set 
$$ N_T:= \left\{ 0=i_0<i_1< \cdots < i_k \bigg|
\begin{array}{lcr}
 i_j \in \{1,...,p\} \forall 1 \leq j \leq k \\
 T_i = \emptyset \iff \exists 1 \leq j \leq k \hspace{.1cm} \tx{s.t.} \hspace{.1cm} i=i_j 
\end{array} \right\} \subset [p].$$

For the case $n=2$, we define
$$ P_2: T=[p]([q_i]) \mapsto N_T([q_{i_j}])$$
on objects; the value of a morphism under $P_2$ is given by r
estriction of that morphism to $N_T$.
$P_2$ respects composition because restriction respects composition.

For general $n$, we define
$$P_n:  T=[p](T_i) \mapsto N_T(P_{n-1}(T_{i_j}))$$
on objects; the value of a morphism under $P_n$ is again 
determined by restriction of that morphism to $N_T$ together with $P_{n-1}$. 
Composition is preserved by $P_n$ because restriction and $P_{n-1}$ both respect composition. 
\end{construction}

\begin{proof}[Proof of Lemma~\ref{adj}]
We use Lemma 2.17 from \cite{AF2} which states that for a functor 
$\C \xra{F} \D$ of $\infty$-categories, $F$ is a right adjoint if and only if 
for each object $d \in \D$, the $\infty$-undercategory $\C^{d/}$ has an initial object, 
and verify that for each $T \in W_n$, the undercategory ${W_n^\h}^{T/}$  
has an initial object. Recall that the $\infty$-undercategory $\C^{d/}$ 
is defined as the pullback of $\infty$-categories
\[
\begin{tikzcd}
\C^{d/} \pb \ar{d} \ar{r} & \D^{d/} \ar{d}{\sf frgt} \\
\C \ar{r}{F} & \D 
\end{tikzcd}
\]

Fix an object $T \in W_n$. 
To define the initial object of ${W_n^\h}^{T/}$, 
we define a morphism $T \xra{\alpha_T} P_n(T)$ in $W_n$ 
such that for any morphism $T \xra{f} S$ in $W_n$ to a healthy tree $S$, 
there is a unique (up to isomorphism) factorization through $\alpha_T$
\[ \begin{tikzcd}
T \ar{r}{f} \ar{d}[swap]{\alpha_T} & S \\
P_n(T) \ar[dashrightarrow]{ur}[swap]{\exists!} & 
\end{tikzcd} \]
in $W_n$. 

We define $\alpha_T$ as follows. 
For $n=1$, $T=P_1(T)$ for each $T \in W_1$ 
and thus, we define $\alpha_T:=\tx{Id}_T$. 

To define $\alpha_T$ for $ T \in W_n$ for $n \geq 2$, 
we proceed by induction. 
Fix an object $T \in W_2$. 
Define $\alpha_T:T=[p]([q_i]) \to N_T([q_{i_j}])$ by 
\begin{enumerate}
\item[i)] $ [p] \to N_T$ is given by the assignment 
$i \mapsto 
\begin{cases} i_j, \hspace{.1cm} \tx{if} \hspace{.1cm} 
\exists 0 \leq j \leq k-1 \hspace{.1cm} \tx{s.t.} \hspace{.1cm}  
i_j \leq i < i_{j+1} \\ i_k, \hspace{.1cm} \tx{if} \hspace{.1cm}  i \geq i_k. 
\end{cases} $
\item[ii)] For each pair $(i, i_j)$ such that $i=i_j$, $[q_i] \xra{\tx{Id}} [q_{i_j}]$.
\end{enumerate}

For the general case, fix an object $T \in W_n$.  
Define $\alpha_T:T=[p](T_i) \to N_T(P_{n-1}(T_{i_j})$ by 
\begin{enumerate}
\item[i)] $ [p] \to N_T$ is the same as i) for $n=2$.
\item[ii)] For each pair $(i, i_j)$ such that $i=i_j$, 
$T_i \xra{\alpha_{T_i}} P_{n-1}(T_{i_j})$, where $\alpha_{T_i}$ 
is guaranteed by the inductive hypothesis. 
\end{enumerate}

Observe that by design each morphism $T \xra{f} S$ in $W_n$ 
to a healthy tree $S$ factors through $\alpha_T$ via $P_n(f)$:

\begin{equation} \label{trim} 
\begin{tikzcd} 
T \ar{r}{f} \ar{d}[swap]{\alpha_T} & S \\
P_n(T) \ar[dashrightarrow]{ur}[swap]{P_n(f)}. & 
\end{tikzcd} 
\end{equation}
Further, up to isomorphism $P_n(f)$ is the unique such filler.
This verifies that for each fixed object $T \in W_n$, 
the initial object of ${W_n^\h}^{T/}$ is 
$$(P_n(T), T \xra{\alpha_T} P_n(T)).$$ 
\end{proof}

\subsubsection{The space of configurations of $r$ unordered points in $\R^n$} \label{unordered} 
In this section, we prove the second lemma needed for the proof of Lemma~\ref{pis0}. 
This lemma identifies the classifying space of the following $\infty$-category.

 \begin{definition} \label{exit conf} 
 For each $r \geq 1$, $W_n^\h(r)$ is the $\infty$-subcategory of $W_n^\h$
defined to be the pullback of $\infty$-categories
\[ \begin{tikzcd}
W_n^\h(r) \pb\ar[hookrightarrow]{r} \ar{d} 
& W_n^\h\ar{d} \\
 \Fun^{n-bij}_r\Bigl(\{1< \cdots <n\}, \bl(\Fin^{\surj} \bigr)^\op\Bigr) \ar[hookrightarrow]{r} 
& \Fun^{\tx{n-bij}}\Bigl(\{1<\cdots<n\}, \bl(\Fin^{\surj}\br) ^\op\Bigr)
\end{tikzcd} \] 
where the functor category $\ds  \Fun^{n-bij}_r\Bigl(\{1< \cdots <n\}, \bl(\Fin^{\surj}\bigr)^\op \Bigr)$ 
is the full subcategory of $\ds  \Fun^{\tx{n-bij}}\Bigl(\{1<\cdots<n\}, \bl(\Fin^{\surj}\br) ^\op\Bigr)$ 
in which the value of an object upon evaluation at $n$ has cardinality $r$. 
We define $W_n^\h(0)$ to be the terminal $\infty$-category.
\end{definition}

Heuristically, an object of $W_n^\h(r)$ is a configuration of $r$ distinct points in $\R^n$ 
and a morphism is a path in the unordered configuration space of $r$ points in $\R^n$ 
such that for each  $0 < k < n$, the projection of the path onto its first $k$-coordinates 
may witness anticollision of points, but not collision. 
In other words, the morphisms are precisely exit-paths 
in the Fox-Neuwirth cell stratification $\Conf_r(\un{\R}^n)_{\Sigma_r}$, as we will soon see.

The underlying data of sets of a $[p]$-point in $W_n^\h(r)$ is restricted to be 
bijective between sets of cardinality $r$. 
This allows us to significantly simplify the description of a $[p]$-point. 
Recall the definition (\ref{ppoint}) of a $[p]$-point in $\Exit\bl(\Ranu(\un{\R}^n)\br)$. 
We may simplify the description for those in $W_n^\h(r)$ in two ways. 
First, the embedding at dimension $n$ will determine all of the embeddings 
of the lower dimensions by taking its projection. 
This is due to the fact that all of our maps between sets are, in particular, surjective. 
Thus, it suffices to only describe the $n$-dimensional embedding data. 
Second, the reverse cylinder construction on a sequence of $p$ bijections 
is trivial in that it is simply the product of the set with the topological $p$-simplex. 
Thus, we make the following observation.

\begin{observation} \label{exit-conf} 
The data of a $[p]$-point in $W_n^\h(r)$ may be described as an embedding 
\begin{equation} \label{simple-ppoint}
R \times \Delta^p \hra \R^n \times \Delta^p
\end{equation}
over $\Delta^p$ for a set $R$ of cardinality $r$ whose projection onto the first 
$k$-coordinates for each $1 \leq k < n$ may witness anticollision but not collision. 
\end{observation}

\begin{observation} \label{two} 
By definition the $\infty$-category $W^\h_n$ (Def.~\ref{loc subcat}) is 
canonically equivalent to the coproduct of $\infty$-categories $\ds \coprod_{r \geq 0}W_n^\h(r).$ 
 \end{observation}

Recall Ayala-Francis-Rozenblyum-Tanaka's exit-path $\infty$-category construction
applied to the Fox-Neuwirth cell stratification $\Conf_r(\un{\R}^n)_{\Sigma_r}$ (\S\ref{exitFN}). 
We wish to show that it is canonically equivalent to the $\infty$-category $W_n^\h(r)$.

\begin{notation} 
For the remainder of the paper, we let $\underline{r}$ denote the set $\{1,...,r\}$. 
\end{notation}

\begin{lemma} \label{confW}
There is a canonical equivalence between $\infty$-categories 
$$W_n^\h(r) \simeq \Exit\left(\Conf_r\left(\un{\R}^n\right)_{\Sigma_r}\right)$$
over $(\Fin^\surj)^\op$.
\end{lemma}

\begin{proof} 
Both $W_n^\h(r)$ and $\Exit\left(\Conf_r\left(\un{\R}^n\right)_{\Sigma_r}\right)$ 
are modeled as complete Segal spaces. 
Thus, they are completely determined by their $[0]$ and $[1]$ values. 
Therefore, it suffices to show an equivalence between the simplicial data at $[0]$ and $[1]$,
which is nearly immediate from their definitions. 
	
On the level of objects this is in fact immediate -- 
the objects of both are points in the underlying stratified space, 
which are configurations of $r$ unordered points in $\R^n$. 
	
On the level of morphisms, we recall the definitions of each. 
In Observation~\ref{exit-conf}, we saw that a morphism in $W_n^\h(r)$ is an embedding 
$$ R \times \Delta^1 \hra \R^n \times \Delta^1$$ 
over $\Delta^1$ and for a set $R$ of cardinality $r$, 
whose projections may witness anticollision, but not collision. 
This restriction (which we note is actually part of the data of the definition)
about anticolliding projections is precisely encoded by the poset 
$nOrd(\un{r})_{\Sigma_r}$ in the Fox-Neuwirth stratification $\Conf_r(\un{\R}^n)_{\Sigma_r}$.
Now recall Ayala-Francis-Rozenblyum-Tanaka's exit-path $\infty$-category functor (Def. \ref{AFRexit}).
The value $\Exit(X)$ of a conically smooth stratified space $X$ 
is a simplicial space assigning $[1]$ to the hom-space 
$\Strat(\Delta^1, X)$ of the $\infty$-category of conically smooth stratified spaces. 
Here, $\Delta^1$ has the standard stratification defined in Example~\ref{standard} over $[1]$. 
Thus, a $[1]$-point of $\Exit\left(\Conf_r\left(\un{\R}^n\right)_{\Sigma_r}\right)$ 
is a stratified map from $\Delta^1$ to $\Conf_r(\un{\R}^n)_{\Sigma_r}$.
The image of such a map is a path in $\Conf_r(\R^n)_{\Sigma_r}$
that is allowed to immediately exit from its present Fox-Neuwirth cell 
to a higher-dimensional one, but not vice-versa.
Such a path witnesses a configuration of $r$ points moving from more coordinate coincidence to less.
Each subsequent projection of such a path would witness anticollision of points, but not collision.
This identifies the $[1]$-points. The equivalence between the structure maps are evident. 
\end{proof}

We are finally ready to prove Lemma~\ref{pis0}. 

\begin{proof}[Proof of Lemma~\ref{pis0}]
Corollary 2.1.28 in \cite{AMG} states that an adjunction between 
$\infty$-categories yields an equivalence between their classifying spaces. 
We apply this result to the adjunction from Lemma~\ref{adj} 
to obtain an equivalence of the classifying spaces, 
$$\B W_n \simeq \B W_n^{\h}.$$

In Observation~\ref{two}, we saw that by definition $W^{\h}_n$ 
is equivalent to the coproduct $\ds \coprod_{r \geq 0} W_n^\h(r)$. 
Then, in Lemma~\ref{confW}, we proved an equivalence of $\infty$-categories 
between $\Exit(\Conf_r(\R^n)_{\Sigma_r})$ and $W_n^\h(r)$. 
Taking the classifying space of these two equivalences yields the equivalence of spaces
\begin{equation} \label{peas}
\ds \B W_n^{\h} \simeq \coprod_{r \geq 0}\B\Exit\bl(\Conf_r(\un{\R}^n)_{\Sigma_r}\br).
\end{equation}

In Corollary~\ref{conically smooth} we proved that the Fox-Neuwirth stratification 
$\Conf_r(\un{\R}^n)_{\Sigma_r}$ is conically smooth. 
In \cite{AFT}  Corollary~1.2.7 states that the classifying space of 
the exit-path $\infty$-category of a conically smooth stratified space 
is homotopy equivalent to the topological space underlying the stratified space. 
Thus, we have the following equivalence of spaces
\begin{equation} \label{carrots} 
\ds \coprod_{r \geq 0}\B\Exit\bl(\Conf_r(\un{\R}^n)_{\Sigma_r}\br) 
\simeq \coprod_{r \geq 0} \Conf_r(\R^n)_{\Sigma_r}.
\end{equation}
Lastly, $ \ds \coprod_{r \geq 0} \Conf_r(\R^n)_{\Sigma_r}$ is
by definition equivalent to $\Exit\bigl(\Ranu(\R^n)\bigr)^\sim$ 
the maximal $\infty$-subgroupoid of $\ds \Exit\bigl(\Ranu(\R^n)\bigr)$.
\end{proof}

\subsection{Identifying the space of morphisms of the localization} \label{1}
By proving Lemma~\ref{pis0}, we have identified that the 
space of objects of the localization $\Exit\bigl(\Ranu(\un{\R}^n)\bigr)[W_n^{-1}]$ 
is equivalent to the space of objects of $\Exit(\Ranu(\R^n))$.
This subsection is devoted to identifying the space of morphisms 
by proving the $p=1$ case of Lemma~\ref{class1}; we restate it as follows. 

\begin{lemma}[Lemma~\ref{class1}; $p=1$] \label{pis1} 
There is an equivalence of spaces 
\[ \B{\sf Fun}^{W_n}\Bigl([1],\Exit\bigl(\Ranu(\un{\R}^n)\bigr)\Bigr) 
\simeq \mor\Bigl(\Exit\bigl(\Ranu(\R^n)\bigr)\Bigr) \]
induced by the forgetful functor 
$\Exit\bigl(\Ranu(\un{\R}^n)\bigr) \to \Exit\bigl(\Ranu(\R^n)\bigr)$. 
 \end{lemma}

The core idea of the proof is the following: 
Both spaces of the equivalence in Lemma~\ref{pis1} naturally assemble 
as the fibrations depicted in (\ref{hey}) below. 
We use the induced long exact sequence of homotopy groups 
to show a weak homotopy equivalence of total spaces by showing 
a homotopy equivalence between the base spaces and between the fibers. 
Then by the Whitehead Theorem we obtain a homotopy equivalence 
of the total spaces since they are CW complexes.
\begin{equation} \label{hey}  
\begin{tikzcd} 
\B\Fun^{W_n}\Bigl([1], \Exit\bigl(\Ranu(\un{\R}^n)\bigr)\Bigr) 
\ar{r}{\tx{frgt}} \ar{d}[swap]{\B\tx{ev}_0} 
& \mor\Bigl(\Exit\bigl(\Ranu(\R^n)\bigr)\Bigr) \ar{d}{\tx{ev}_0} \\
\B\Fun^{W_n}\Bigl([0], \Exit\bigl(\Ranu(\un{\R}^n)\bigr)\Bigr) \ar{r} 
& \Exit\bigl(\Ranu(\R^n)\bigr)^{\sim}
\end{tikzcd} 
\end{equation}

Since we already showed that the base spaces are equivalent in Lemma~\ref{pis0}, 
our work lies in showing an equivalence on the level of fibers. 
We identify the fibers throughout the course of three lemmas: 
Lemma~\ref{fiber2} identifies the fibers of $\ev_0$ 
and lemmas~\ref{fiber1pt1} and \ref{fiber1pt2} identify the fibers of $\B \ev_0$. 
Each lemma is technical, relying on the concept of a \emph{Cartesian fibration}. 
In the next subsection, we follow \cite{AF2} to recall this technical machinery 
and further tailor it to the situation at hand, towards the goal of proving
Lemma~\ref{fiber2}, Lemma~\ref{fiber1pt1}, and Lemma~\ref{fiber1pt2}.

\subsubsection{Cartesian fibrations}
\begin{definition}[Def. 2.1, \cite{AF2}] 
Let $\E \xra{\pi} \B$ be a functor between $\infty$-categories. 
A morphism $c_1 \xra{\la e \xra{\varphi} e' \ran} \E$ is $\pi$-\emph{Cartesian} 
if the diagram of $\infty$-overcategories
\[ \begin{tikzcd}
\E_{/e} \ar{r}{\varphi \circ -} \ar{d}[swap]{\pi} & \E_{/e'} \ar{d}{\pi} \\
\B_{/\pi(e)} \ar{r}{\pi(\varphi) \circ -} & \B_{/\pi(e')}
\end{tikzcd} \]
is a pullback. 

$\pi$ is a \emph{Cartesian fibration} if for every solid square
\[ \begin{tikzcd}
\ast \ar{r} \ar{d}[swap]{\la t \ran} & \E \ar{d}{\pi} \\
c_1 \ar[dashrightarrow]{ur} \ar{r} & \B 
\end{tikzcd} \]
there is a $\pi$-Cartesian filler. 
\end{definition}

\begin{observation} \label{Cart} 
The map 
$$\ds \Fun^{W_n}\Bigl([1], \Exit\bigl(\Ranu(\un{\R}^n)\bigr)\Bigr) 
\xra{\ev_0} \Fun^{W_n}\Bigl([0], \Exit\bigl(\Ranu(\un{\R}^n)\bigr)\Bigr)$$ 
is a Cartesian fibraton. 
The proof is straightforward, using Example 2.5 in \cite{AF2}, 
wherein it is shown that for an $\infty$-category $\C$, 
the functor given by evaluation at $0$, $\Fun\bl([1], \C\br) \xra{\ev_0} \Fun\bl([0], \C\br)$
 is a Cartesian fibration.
 \end{observation}

\begin{observation} 
Let $\mathcal{E} \xra{\pi} \mathcal{B}$ be a Cartesian fibration. 
For each object $ b \in \B$ there is a canonical inclusion 
$\pi^{-1}(b) \hookrightarrow \E^{b/} $ from the fiber of $\pi$ over $b$ 
to the undercategory of $\E$ under $b$. Its value on an object $e \in \pi^{-1}(b)$ 
such that $b \cong \pi(e)$ is the equivalence $b \xra{\cong} \pi(e)$. 
Its value on a morphism $e \xra{f} e'$ is $\pi(f)$. 
\end{observation}


\begin{definition}[Cartesian Monodromy Functor] 
Let $\mathcal{E} \xra{\pi} \mathcal{B}$ be a Cartesian fibration. 
For each morphism $b \xra{f} b'$ in $\mathcal{B}$, 
the induced \emph{Cartesian monodromy functor} $f^\ast : \pi^{-1}(b')\ra \pi^{-1}(b)$  
from the fiber over $b'$ to the fiber over $b$   is defined to be the threefold composite
\[ \begin{tikzcd}
\pi^{-1}(b')\ar[dashrightarrow]{r}{f^\ast} \ar[hookrightarrow]{d} & \pi^{-1}(b)  \\
\E^{b'/} \ar{r}{- \circ f} & \E^{b/} \ar{u}[swap]{\mu}
\end{tikzcd} \]
where $\mu$ is right adjoint to the inclusion functor 
$ \ds \pi^{-1}(b) \hookrightarrow \E^{b/}$, 
the existence of which is guaranteed by Lemma~2.20 in \cite{AF2}.  
\end{definition}

\begin{observation} \label{mono} 
Given a diagram of $\infty$-categories 
\begin{equation} \label{EB} 
\begin{tikzcd}
\E \ar{r}{G} \ar{d}[swap]{\pi} & \E' \ar{d}{\pi'} \\
\B \ar{r}{F} & \B' 
\end{tikzcd} \end{equation}
in which $\pi$ and $\pi'$ are Cartesian fibrations, 
for each morphism $b \xra{\alpha} b'$ in $\B$, 
$G$ carries the induced monodromy functor $\alpha^\ast$ 
to the monodromy functor induced by $F(\alpha)$, 
\begin{equation} \label{EB1} \begin{tikzcd}
\pi^{-1}(b') \ar{d} \ar{r}{\alpha^\ast} & \pi^{-1}(b) \ar{d} \\
\pi'^{-1}(F(b')) \ar{r}{F(\alpha)^\ast} & \pi'^{-1}(F(b)).
\end{tikzcd} \end{equation}
Moreover, if (\ref{EB}) is a pullback of $\infty$-categories, 
then the downward vertical arrows of (\ref{EB1}) are equivalences.
\end{observation}

\subsubsection{The fibers of $\tx{ev}_0$}
We are now equipped to identify the fibers of the evaluation functor 
on the righthand side of (\ref{hey}) with the following lemma.
First take note that the following notational changes are applied for the remainder of the paper.
\begin{itemize}
\item In $\Exit\bigl(\Ranu(\R^n)\bigr)$ we denote an object $S \xhra{e} \R^n$ by $S$, 
by which we mean the image of $S$ in $\R^n$ under the embedding $e$.
\item  In $\Exit\bigl(\Ranu(\un{\R}^n)\bigr)$ we denote an object 
$\un{S} \xhra{\un{e}} \un{\R}^n$ by $\un{S} = S_n \ra \cdots \ra S_1$, 
by which we mean the images of $S_i$ under $e_i$ for each $1 \leq i \leq n$ 
together with the coordinate projection data given by the sequence of 
maps of finite sets $S_n \ra \cdots \ra S_1$.  

We denote a morphism $\cylr(\un{S}' \xra{\un{\sigma}} \un{S} ) 
\xhra{\un{E}} \un{\R}^n \times \Delta^1$ from $\un{S}$ to $\un{S}'$ 
by simply an arrow $\un{S} \ra \un{S}'$. 
\end{itemize} 

\begin{lemma} \label{fiber2} 
For a smooth connected manifold $M$, the fiber of the map of spaces
$$\ds \mor\Bigl(\Exit\bigl(\Ranu(M)\bigr)\Bigr) \xra{\tx{ev}_0} \Exit\bigl(\Ranu(M)\bigr)^{\sim}$$ 
from the space of morphisms of $\Exit\bigl(\Ranu(M)\bigr)$ 
to the maximal $\infty$-subgroupoid of $\Exit\bigl(\Ranu(M)\bigr)$ 
over an object $S \subset M$ of $\Exit\bigl(\Ranu(M)\bigr)$ is the product 
$$\ds \prod_{s\in S} \Exit\bl(\Ranu(T_sM)\br)^{\sim}$$
indexed over $s \in S$ of the maximal $\infty$-subgroupoid 
of the exit-path $\infty$-category of the unital Ran space of 
the tangent space $T_sM$ of $M$ at $s \in S$. 
\end{lemma}

\begin{proof}  Fix $S \subset M$. 
We want to identify the fiber over $S$ by constructing an explicit homotopy equivalence
$$ \prod_{s \in S} \coprod_{r \geq 0} \Conf_r(T_sM)_{\Sigma_r} 
\overset{\simeq}\to \ev_0^{-1}(S).  $$
We remind the reader that the underlying maximal $\infty$-subgroupoid of 
$\Exit\bl(\Ranu(T_sM)\br)$ is equivalent to the coproduct of configuration spaces
$$\ds \Exit\bl(\Ranu(T_sM)\br)^\sim \simeq \coprod_{r \geq 0} \Conf_r(T_sM)_{\Sigma_r}.$$
 For each $s \in S$, choose a smooth open embedding $T_sM \hra M$ 
 which carries the origin to $s$ such that the images $\{U_s\}_{s\in S}$ are disjoint. 
 These open embeddings induce homeomorphisms 
 $\Conf_r(T_sM)_{\Sigma_r} \cong \Conf_r(U_s)_{\Sigma_r}$, 
 which in turn identify homeomorphisms
\begin{equation} \label{first1} 
\ds \prod_{s \in S} \coprod_{r \geq 0} 
\Conf_r(T_sM)_{\Sigma_r} \cong \prod_{s \in S}\coprod_{r \geq 0} \Conf_r(U_s)_{\Sigma_r} 
\cong \coprod_{r\geq 0}\Conf_r \Big(\coprod_{s\in S}U_s \Big)_{\Sigma_r}
\end{equation}
where the second homeomorphism is canonical because the images 
$\{U_s\}$ are disjoint and the empty configuration is allowed. 
The inclusions $\ds \coprod_{s \in S} U_s \hra M$ determine an inclusion
\begin{equation} \label{second} 
\ds \coprod_{r \geq 0} \Conf_r\Big(\coprod_{s \in S} U_s\Big)_{\Sigma_r} 
\hookrightarrow \coprod_{r \geq 0} \Conf_r(M)_{\Sigma_r}. 
\end{equation}
Composing (\ref{first1}) and (\ref{second}) results in the top horizontal map in the following diagram
\begin{equation} 
\begin{tikzcd} \label{third} \ds
\prod_{s \in S} \coprod_{r \geq 0} \Conf_r(T_sM)_{\Sigma_r} 
\ar[hookrightarrow]{r} \ar[dashrightarrow]{dr}[swap]{{\sf straight}} 
& \ds \coprod_{r \geq 0} \Conf_r(M)_{\Sigma_r} \\
& \mor \Exit\bl(\Ranu(M)\br). \ar[dashrightarrow]{u}[swap]{\sf target} 
\end{tikzcd} 
\end{equation}
This injection factors through the space of morphisms of $\Exit\bl(\Ranu(M)\br)$ 
by taking the straight-line paths from the origin of each $T_sM$ 
to the given configuration in the image $U_s$.
Let us define the factorization ${\sf straight}$ explicitly. 
Consider a configuration $(R_s \subset T_sM)_{s \in S}$. 
For each $s \in S$ define 
$$\gamma_s: \bar{C}(R_s) \to T_sM \times \Delta^1$$
 over $\Delta^1$ from the closed cone on the underlying set $R_s$ 
 to be the map which carries the cone point to the origin of $T_sM$ at $t=0 \in \Delta^1$ 
 and which carries each $\{r\} \times \Delta^1$ to the straight-line path from $0$ to $r \in R_s$. 
Define the map ${\sf straight}$ in (\ref{third}) to assign to each 
configuration $(R_s \subset T_sM)_{s \in S}$ the morphism in $\Exit\bl(\Ranu(M)\br)$ 
determined by the composites 
$$ \bar{C}(R_s) \xra{\gamma_s} T_sM \times \Delta^1 \simeq U_s \times \Delta^1 
\hookrightarrow M \times \Delta^1$$
for each $s \in S$. 

By construction, the functor ${\sf straight}$ in (\ref{third}) factors through the fiber $\ev_0^{-1}(S)$
\[\ds 
\begin{tikzcd}
\ds \prod_{s \in S} \coprod_{r \geq 0} \Conf_r(T_sM)_{\Sigma_r}  \ar{r}{{\sf straight}} 
\ar[dashrightarrow]{dr}[swap]{{\sf Str}}& \mor \Exit\bl(\Ranu(M)\br) \\
& \ar[hookrightarrow]{u} \ev_0^{-1}(S).
\end{tikzcd} \]
We will show that the factorization ${\sf Str}$
 is a homotopy equivalence by showing that for any commutative square
\begin{equation} \label{fourth}
\begin{tikzcd}
\mathbb{S}^n \ar{r} \ar[hookrightarrow]{d} & \ds \prod_{s \in S} \coprod_{r \geq 0} 
\Conf_r(T_sM)_{\Sigma_r} \ar{d}{\sf Str} \\
\mathbb{D}^{n+1} \ar[dashrightarrow]{ur}{\exists} \ar{r} & \ev_0^{-1}(S) 
\end{tikzcd}
 \end{equation}
there exists a filler making each triangle commute up to homotopy. 
Recall that a point in $\ev_0^{-1}(S) \subset \mor \Exit\bl(\Ranu(M)\br)$ 
is the data of a finite proper constructible bundle 
$$\cylr(T \to S) \lra \Delta^1$$
 together with an embedding 
 $$\cylr(T \to S) \hra M \times \Delta^1$$
  over $\Delta^1$ for some finite set $T$ and some map $T \to S$. 
  (Explicitly, such a description is due to the equivalence $\Exit\bl(\Ranu(M)\br) 
  \simeq \RRef^0(M)$ of Lemma~\ref{l1}.) 
  Heuristically a map $\mathbb{D}^{n+1} \to \ev_0^{-1}(S)$ 
  is a movie of embeddings $\cylr(T \to S) \hra M \times \Delta^1$ 
  parameterized by $\mathbb{D}^{n+1}$. 
  Explicitly, this means that a map $\mathbb{D}^{n+1}  \to \ev_0^{-1}(S)$ 
  is the data of a finite proper constructible bundle 
$$ \tilde{K} \lra \mathbb{D}^{n+1} \times \Delta^1$$
together with an embedding
\[ \begin{tikzcd}
\tilde{K} \ar[hookrightarrow]{r} \ar{dr} & \mathbb{D}^{n+1} \times \Delta^1 \times M \ar{d} \\
& \mathbb{D}^{n+1} \times \Delta^1
\end{tikzcd} \]
that is constant at $S$ over $\mathbb{D}^{n+1} \times \Delta^{\{0\}}$.
For ease of notation, identify $\Delta^1 \cong [0,1]$ such that $\Delta^{\{0\}} = 0$. 
Compactness of $\mathbb{D}^{n+1}$ ensures that there exists $t \in (0,1]$ 
such that $\tilde{K}|_{[0,t]}$ factors through $\ds \coprod_{s \in S} U_s \subset M$. 
Further, taking the restriction of $\tilde{K}|_{[0,t]}$ along 
$\mathbb{S}^n \hookrightarrow \mathbb{D}^{n+1}$ 
and scaling by this $t$ extends to a commutative square (\ref{fourth}) 
for the case that $M$ is a finite disjoint union of Euclidean spaces. 
By virtue of the coproduct, we may reduce yet further to the case of a single 
Euclidean space; i.e., the case that $S = \{s\}$ is a singleton. 
Without loss of generality, choose $s$ to be the origin. 
So, showing that the map given by taking straight-line paths from the origin
$$ \ds \prod_{s \in S} \coprod_{r \geq 0} \Conf_r(T_sM)_{\Sigma_r} \xra{\sf Str} \ev_0^{-1}(S)$$
is a homotopy equivalence has been reduced to showing that the same map
\begin{equation} \label{fifth} 
\ds \coprod_{r \geq 0} \Conf_r(\R^n)_{\Sigma_r} \xra{\sf Str}  \ev_0^{-1}(\{0\}) \end{equation}
in the case of $M =\R^n$ and $S =\{0\}$ is a homotopy equivalence. 
In fact, we will explicitly describe such a homotopy equivalence.
A map in the other direction
\begin{equation} \label{sixth}
\ev_0^{-1}(\{0\}) \lra \coprod_{r \geq 0} \Conf_r(\R^n)_{\Sigma_r} \end{equation}
is defined as follows.
A point in $\ev_0^{-1}(\{0\})$ is a conically smooth embedding over $[0,1]$
\begin{equation} \label{seventh}  \bar{C}(I) \hra \R^n \times [0,1] \end{equation}
from the closed cone of some finite set $I \in \Fin$ that carries the cone point to the origin. 
By conical smoothness, the blowup of such an embedding at the cone point is a smooth map 
$$ {\sf Bl}(\bar{C}(I)) = I \times [0,1] \lra \R^n$$
which is an embedding over $(0,1]$ and which carries each $\{i\} \times [0,1]$ 
to a smooth path $\gamma_i$ in $\R^n$ from the origin, with non-zero derivative at the origin. 
The collection of these derivatives determine the map (\ref{sixth}). 
Explicitly, the point (\ref{seventh}) is sent to the collection of 
non-zero initial velocity vectors $(d\gamma_i(0) \subset \R^n)_{i \in I}$ 
associated to it; note that the image names a point in $\Conf_{|I|}(\R^n)_{\Sigma_r}$.

To see that (\ref{fifth}) and (\ref{sixth}) yield a homotopy equivalence, 
first note that the composite of (\ref{fifth}) followed by (\ref{sixth}) 
is the identity on $\ds \coprod_{r \geq 0} \Conf_r(\R^n)_{\Sigma_r}$. 
We describe a homotopy from the identity on $\ev_0^{-1}(S)$ 
to the other composite as follows. Let $\bar{C}(I) \xra{\gamma} \R^n \times \Delta^1$ 
be a point in $\ev_0^{-1}(\{0\})$. By virtue of approximating smooth functions 
by piecewise linear functions, for each $i\in I$, there is a $t \in (0,1]$ 
such that the $i^{\tx{th}}$ path in $\gamma$ projects orthogonally 
onto the straight ray from the origin in the direction of the initial velocity vector of that path. 
Choose the minimum such $t$ over all $i \in I$. 
Straight-line homotopy from each path in $\gamma$ restricted to $[0,t]$ 
to the straight ray given by path's initial velocity vector yields the desired homotopy 
from the identity on $\ev_0^{-1}(\{0\})$ to the composite of (\ref{sixth}) followed by (\ref{fifth}). 
\end{proof}

Applying Lemma~\ref{fiber2} to the case of $M=\R^n$, 
we have identified the fiber of $\ev_0$ in (\ref{hey}) as desired. 
The next step is to identify the fiber of the lefthand functor 
$\B\ev_0$ in (\ref{hey}) as equivalent to that of the righthand functor. 
The next lemma is the first of two steps towards this.
It identifies the classifying space of the fiber of the map from 
${\sf Fun}^{W_n}\Bigl([1],\Exit\bigl(\Ranu(\un{\R}^n)\bigr)\Bigr)$ 
to $ {\sf Fun}^{W_n}\Bigl([0],\Exit\bigl(\Ranu(\un{\R}^n)\bigr)\Bigr)$ given by evaluation at $0$.

\begin{lemma} \label{fiber1pt1} 
The classifying space of the fiber of the functor 
$$\ds {\sf Fun}^{W_n}\Bigl([1],\Exit\bigl(\Ranu(\un{\R}^n)\bigr)\Bigr) 
\xra{ \tx{ev}_0} {\sf Fun}^{W_n}\Bigl([0],\Exit\bigl(\Ranu(\un{\R}^n)\bigr)\Bigr)$$ 
over an object $\un{S} := S_n \to S_{n-1} \to \cdots \to S_1$ is the product  space
$$\ds \prod_{s\in S_n} \Exit\bl(\Ranu(T_s\R^n)\br)^{\sim}$$ 
 indexed by $ S_n$ of the maximal $\infty$-subgroupoid of 
 the exit-path $\infty$-category of the unital Ran space 
 of the tangent space of $\R^n$ at $s \in S_n$.
 \end{lemma}
 
\begin{proof} 
Fix an object $\un{S} := (S_n \to S_{n-1} \to \cdots \to S_1) 
\in {\sf Fun}\Bigl([0],\Exit\bigl(\Ranu(\un{\R}^n)\bigr)\Bigr)$. 
Our procedure for identifying the fiber over $\un{S}$ has two steps. 
First recall the Fox-Neuwirth stratification of the unordered configuration space 
$\nu_{\Sigma_r}: \Conf_r(\R^n)_{\Sigma_r} \to nOrd(\un{r})_{\Sigma_r}$ (Obs~\ref{toe}). 
The first step is to show that there exists a refinement (Def. \ref{refinement}) of the stratified space
\begin{equation} \label{heyo} 
\prod \coprod \nu_{\Sigma_r}: \prod_{s \in S_n} \coprod_{r \geq 0}\Conf_r(T_s\R^n)_{\Sigma_r} 
\lra  \prod_{s \in S_n} \coprod_{r \geq 0} nOrd(\underline{r})_{\Sigma_r} 
\end{equation}
where the stratification is given by taking the product of the coproduct 
of the Fox-Neuwirth stratification of the configuration space of the tangent space of $\R^n$ at $s \in S_n$. 
(Note that products of stratified spaces emit a natural stratification (Ex.~2.1.7, \cite{AFT}).) 
Just like in the proof of Lemma~\ref{fiber2}, for each $s \in S_n$, 
choose a smooth open embedding 
\begin{equation} \label{kk} 
T_s\R^n \hra \R^n 
\end{equation}
 which carries the origin to $s$ such that the images $\{U_s\}_{s\in S_n}$ are disjoint.
The inclusions $\ds \coprod_{s \in S_n} U_s \hra \R^n$ 
determine an inclusion of the configuration spaces 
naming an inclusion of the substratified space 
$$
	 \Conf_r\Big(\coprod_{s \in S_n} U_s\Big)_{\Sigma_r} \hra
	 \Conf_r(\R^n)_{\Sigma_r} \overset{\nu_{\Sigma_r}}\lra
	  nOrd(\underline{r})_{\Sigma_r}
$$
consisting of configurations of $r$ unordered points in the images 
$\coprod_{s \in S_n} U_s \subset \R^n$ stratified by the Fox-Neuwirth cells. 
Taking the coproduct over cardinality determines the substratified space
\begin{equation} \label{peace}  
 \coprod_{r\geq 0} \Conf_r\left(\coprod_{s \in S_n} U_s\right)_{\Sigma_r} \hra
 \coprod_{r \geq 0} \Conf_r(\R^n)_{\Sigma_r} \overset{\coprod \nu_{\Sigma_r}}\lra
 \coprod_{r \geq 0} nOrd(\underline{r})_{\Sigma_r} 
\end{equation}
consisting of all finite subsets in the images $\ds \coprod_{s \in S_n}U_s \hra \R^n$ 
stratified by the Fox-Neuwirth cells. 

The open embeddings (\ref{kk}) induce homeomorphisms 
$\Conf_r(T_s\R^n)_{\Sigma_r} \xra{\cong} \Conf_r(U_s)_{\Sigma_r}$, 
which in turn identify the topological spaces
\begin{equation} \label{nuts} 
\ds \prod_{s \in S_n} \coprod_{r \geq 0} \Conf_r(T_s\R^n)_{\Sigma_r} 
\xra{\cong} \prod_{s \in S_n}\coprod_{r \geq 0} \Conf_r(U_s)_{\Sigma_r} 
\cong \coprod_{r\geq 0}\Conf_r \Big(\coprod_{s\in S_n}U_s \Big)_{\Sigma_r}
\end{equation}
where the last equivalence is canonical because the subsets $U_s$ 
are disjoint and we have the empty configuration. 
The composite of these homeomorphisms (\ref{nuts}) is a refinement  
of stratified spaces, as indicated in the following diagram:
\begin{equation} \label{zzp}  
\begin{tikzcd}
\coprod_{r\geq 0} \Conf_r\left(\coprod_{s \in S_n} U_s\right)_{\Sigma_r} \ar{r}{\sf refine} \ar{d} &
 \prod_{s \in S_n} \coprod_{r \geq 0} \Conf_r(T_s\R^n)_{\Sigma_r} \ar{d}{ \prod \coprod \nu_{\Sigma_r}} \\
\coprod_{r \geq 0} nOrd(\underline{r})_{\Sigma_r} \ar{r} &
 \prod_{s \in S_n} \coprod_{r \geq 0} nOrd(\underline{r})_{\Sigma_r} 
\end{tikzcd}
\end{equation}
Let us explain further. The lefthand downward arrow is the Fox-Neuwirth stratification 
$\nu_{\Sigma_r}$ restricted to configurations in the images $\coprod_{s \in S_n} U_s \subset \R^n$. 
This stratification remembers coordinate coincidence between any points, 
even those in distinct $U_s$ subsets. 
The righthand downward arrow is also determined by the Fox-Neuwirth stratification. 
However, it only remembers the coordinate coincidence between points in the same $U_s$ subset. 
Any coordinate coincidence between points in distinct $U_s$ subsets is ignored. 
Thus, this map of stratified spaces forgets information, 
meaning the map between posets is a many-to-one map. 
Therefore, since the top horizontal arrow is a homeomorphism, 
each stratum of the domain in (\ref{zzp}) is embedded into a stratum of the codomain, 
which verifies that this map is indeed a refinement, completing the first step of the proof.

Our second step is to show that there is an adjunction between 
the exit-path $\infty$-category of the stratified space 
$\ds \coprod_{r\geq 0} \Conf_r\left(\coprod_{s \in S_n} U_s\right)_{\Sigma_r}$ 
and the fiber $\ds \ev_0^{-1}(\un{S}) \subset \Fun^{W_n}\Bigl([1], \Exit\bl(\Ranu(\un{\R}^n)\br)\Bigr)$.
First we apply Ayala-Francis-Rozenblyum-Tanaka's exit-path $\infty$-category 
functor (Def. \ref{AFRexit}) to the inclusion of stratified spaces (\ref{peace}) 
to obtain the top horizontal functor in the following diagram
\begin{equation} 
\begin{tikzcd} \label{diligence} \ds
\coprod_{r\geq 0} \Exit\left(\Conf_r\left(\un{\coprod_{s \in S_n} U_s}\right)_{\Sigma_r}\right) 
\ar{r}{\Exit((\ref{peace}))} \ar[dashrightarrow]{dr}[swap]{\sf straight} 
& \ds \coprod_{r \geq 0} \Exit\bl(\Conf_r(\un{\R}^n)_{\Sigma_r} \br)\\
&\Fun^{W_n}\Bigl([1], \Exit\bl(\Ranu(\un{\R}^n)\br)\Bigr). \ar[dashrightarrow]{u}[swap]{\sf target} 
\end{tikzcd} 
\end{equation}
In particular, this functor factors through $ \ds \Fun^{W_n}\Bigl([1], \Exit\bl(\Ranu(\un{\R}^n)\br)\Bigr)$  
by taking the straight-line paths from each $s \in S_n$ 
(the same idea as that of (\ref{third}) in the proof of Lemma~\ref{fiber2}). 
We wish to define the value of ${\sf straight}$ on a $[p]$-point. 
We use the model of $W_n^\h(r)$ as a simplicial space 
(inherited from $\Exit\bigl(\Ran(\un{\R}^n)\bigr)$) to describe $[p]$-points 
in $\Exit\left(\Conf_r\left(\un{\coprod_{s \in S_n} U_s}\right)_{\Sigma_r}\right)$, 
taking advantage of the equivalence from Lemma~\ref{confW}. 
Also note that the $\infty$-category $ \Fun^{W_n}\Bigl([1], \Exit\bl(\Ranu(\un{\R}^n)\br)\Bigr)$ 
inherits a model as a simplicial space from that of $\Exit\bl(\Ranu(\un{\R}^n)\br)$. 
Recall from Observation~\ref{exit-conf} that a $[p]$-point of 
$\Exit\left(\Conf_r\left(\un{\coprod_{s \in S_n} U_s}\right)_{\Sigma_r}\right)$ is an embedding 
\begin{equation} \label{simp} 
R \times \Delta^p  \hra \coprod_{s \in S_n} U_s \times \Delta^p 
\end{equation}
over $\Delta^p$ for a set $R$ of cardinality $r$. 
Consider the natural decomposition $R = \coprod_{s \in S_n} R_s$ 
into the subsets of its image contained in each $U_s$.
The $[p]$-point (\ref{simp}) can be described as a collection of disjoint embeddings
\begin{equation} \label{simp2}
R_s \times \Delta^p \overset{E}\hra U_s \times \Delta^p
\end{equation}
one for each $s \in S_n$. 
The value of this $[p]$-point when $p=0$ is described by the factorization 
(\ref{third}) in the proof of Lemma~\ref{fiber2} using the cone construction, 
but the same concept applies for any $p$; let us make this explicit. 
For each $s \in S_n$ define 
\begin{equation} \label{simp3}
\bar{C}(R_s \times \Delta^p) \overset{E_s} \hra U_s \times \Delta^{p+1}
\end{equation}
over $\Delta^{p+1}$ from the closed cone of the product $R_s \times \Delta^p$ 
to be the map which carries the cone point to $s \in U_s$ at $t=0 \in \Delta^{p+1}$
 and which carries each $\{r\} \times \{t\} \times \Delta^1$, for $r \in R_s$ and $t \in \Delta^p$, 
 to the straight-line path from $s$ to the point $E_s(r,t)$. 
 The collection of these embeddings (\ref{simp3}) for $s \in S_n$ 
 determines a $[p+1]$-point in $\Exit\bl(\Ranu(\un{\R}^n)\br)$ 
 whose data at $t=0 \in \Delta^{p+1}$ is the object $\un{S}$. 
 Moreover, this $[p+1]$-point determines the value of (\ref{simp}) under the functor ${\sf straight}$; 
 let us explain further. 
 Recall from Observation~\ref{one} that a $[p]$-point in 
 $\Fun^{W_n}\Bigl([1], \Exit\bl(\Ranu(\un{\R}^n)\br)\Bigr)$ is a commutative diagram 
 in $\Exit\bl(\Ranu(\un{\R}^n)\br)$ of the shape $[p] \times [1]$ 
 such that the two $p$-simplicies $[p] \cong [p]\times \{0\}$ and $[p] \cong [p] \times \{1\}$ 
 must be $[p]$-points of the $\infty$-subcategory $W_n$. 
By taking the identity on $\un{S}$ the appropriate number of times, 
we can build the $[p+1]$-point (\ref{simp3}) of $\Exit\left(\un{\R}^n)\right)$ 
into a commutative diagram of the shape $[p] \times [1]$. All of the morphisms $\un{S} \to R$ 
given by straight-line paths in (\ref{simp}) comprise the ``$[p]$-thickened walking arrow'' of the diagram. 
Namely, all those arrows $\{i\} \times [1]$ for each $0 \leq i \leq p$. 
These are not necessarily in $W_n$. 
The arrows of $[p] \times \{0\}$ are all the identity on $\un{S}$, 
creating a $[p]$-point in $W_n$. 
Lastly, the $p$-simplex $[p] \times \{0\}$ is precisely the $[p]$-point (\ref{simp}) 
of $\Exit\left(\Conf_r\left(\un{\coprod_{s \in S_n} U_s}\right)_{\Sigma_r}\right)$. 
It, too, is in $W_n$ because $\Exit\bl(\Conf_r(\un{\R}^n)_{\Sigma_r} \br)$ 
embeds fully faithfully into $W_n$ by the equivalences Observation~\ref{two} and Lemma~\ref{confW}
$\coprod_{r \geq 0} \Exit\bl(\Conf_r(\un{\R}^n)_{\Sigma_r} \br) 
\simeq \coprod_{r \geq 0} W_n^\h(r) \simeq W_n^\h \hra W_n.$

By construction, the functor ${\sf straight}$ in (\ref{diligence}) 
factors through the fiber $\ev_0^{-1}(\un{S})$
\[\ds 
\begin{tikzcd}
\ds \coprod_{r\geq 0} \Exit\left(\Conf_r\left(\un{\coprod_{s \in S_n} U_s}\right)_{\Sigma_r}\right) 
 \ar{r}{\sf straight} \ar[dashrightarrow]{dr}[swap]{\sf Str} 
 & \Fun^{W_n}\Bigl([1], \Exit\bl(\Ranu(\un{\R}^n)\br)\Bigr) \\
& \ar[hookrightarrow]{u} \ev_0^{-1}(\un{S})
\end{tikzcd} 
\]
We claim that ${\sf Str}$ is a right adjoint functor. 
This follows from Lemma~\ref{adj}. 
Indeed, observe that in the following commutative diagram of $\infty$-categories
\begin{equation} \label{beach}
\begin{tikzcd}
\ds  \coprod_{r\geq 0} \Exit\left(\Conf_r\left(\un{\coprod_{s \in S_n} U_s}\right)_{\Sigma_r}\right) 
\ar{r}{\sf Str} \ar[hookrightarrow]{d}[swap]{\sf inclusion} 
& {\sf ev}_0^{-1}(\un{S}) \ar[hookrightarrow]{d}{\sf target} \\
W_n^\h \ar[hookrightarrow]{r}{\sf inclusion} & W_n
\end{tikzcd} 
\end{equation}
both vertical arrows are fully faithful. 
In Lemma~\ref{adj}, we showed that the bottom horizontal arrow in (\ref{beach}) is a right adjoint. 
Thus, it is that ${\sf Str}$ is a right adjoint functor, as desired.

To finish up the proof, note that Corollary~2.1.28 of \cite{AMG} 
states that the classifying spaces of adjoint $\infty$-categories are equivalent. 
Applying this result yields the equivalence 
\begin{equation}
\label{puff}
\B\ev_0^{-1}(\un{S}) \simeq \B \coprod_{r\geq 0} 
\Exit\left(\Conf_r\left(\un{\coprod_{s \in S_n} U_s}\right)_{\Sigma_r}\right). 
\end{equation}
Moreover, Corollary 1.2.7 of \cite{AFT} identifies the classifying space of a value 
under the $\Exit$ functor as its underlying topological space. 
Recall the refinement 
$$\coprod_{r\geq 0} \Conf_r\left(\coprod_{s \in S_n} U_s\right)_{\Sigma_r}  
\to	\prod_{s \in S_n} \coprod_{r \geq 0} \Conf_r(T_s\R^n)_{\Sigma_r}$$ 
from (\ref{zzp}). 
Refinements always have the same underlying topological space as that which it refines. 
Thus, the topological space underlying $\coprod_{r\geq 0} 
\Conf_r\left(\coprod_{s \in S_n} U_s\right)_{\Sigma_r}$ 
is $ \prod_{s \in S_n} \coprod_{r \geq 0} \Conf_r(T_s\R^n)_{\Sigma_r}$. 
Applying Corollary~1.2.7 then, we have the equivalence 
$$\B \coprod_{r\geq 0} \Exit\left(\Conf_r\left(\un{\coprod_{s \in S_n} U_s}\right)_{\Sigma_r}\right) 
\simeq  \prod_{s \in S_n} \coprod_{r \geq 0}\Conf_r(T_s\R^n)_{\Sigma_r} $$
which together with (\ref{puff}) yields the desired equivalence, 
since for each $s \in S_n$, the maximal $\infty$-subgroupoid is identified 
$\Exit\bl(\Ranu(T_s\R^n)\br)^{\sim} \simeq \coprod_{r \geq 0}\Conf_r(T_s\R^n)_{\Sigma_r}$.
\end{proof}

Let us briefly review. We seek to show an equivalence on the level of fibers in (\ref{hey}).
We have just shown that the classifying space of the fiber of 
$$\ev_0: {\sf Fun}^{W_n}\Bigl([1],\Exit\bigl(\Ranu(\un{\R}^n)\bigr)\bigr) 
\ra  {\sf Fun}^{W_n}\Bigl([0],\Exit\bigl(\Ranu(\un{\R}^n)\bigr)\Bigr)$$
over $\un{S}$ is equivalent to the fiber of
$$ \ev_0: \mor\Bigl(\Exit\bigl(\Ranu(\R^n)\bigr)\Bigr) \ra \Exit\bigl(\Ranu(\R^n)\bigr)^{\sim}$$
over $S$, as identified in Lemma~\ref{fiber2}.
Thus, we have yet to show that this fiber over $S$ is equivalent to the fiber of 
$$\B\ev_0: \B{\sf Fun}^{W_n}\Bigl([1],\Exit\bigl(\Ranu(\un{\R}^n)\bigr)\Bigr) \ra  \B{\sf Fun}^{W_n}\Bigl([0],\Exit\bigl(\Ranu(\un{\R}^n)\bigr)\Bigr)$$
over $\un{S}$, since, in general, the classifying space of a fiber is not 
inherently equivalent to the fiber of the map induced between classifying spaces. 
That is to say, for a functor of $\infty$-categories $\C \xra{F} \D$, 
the classifying space $\B(F^{-1}d)$ of the fiber of $F$ over $d \in \D$ 
is not necessarily equivalent to the homotopy fiber $ (\B F)^{-1}(d)$ 
of the map induced between the classifying spaces $\B\C \xra{\B F} \B\D$ 
over $d \in \B \D$. Quillen's Theorem~B \cite{Quil, AF2} identifies the homotopy fibers 
of the map induced between classifying spaces in favorable cases. 
 
\begin{theorem}[Quillen's Theorem~B] 
Given a functor between $\infty$-categories $\C \xra{F} \D$, 
if each morphism $d \xra{f} d'$ in $\D$ induces a weak equivalence 
$\B(\C^{d'/}) \xra{\simeq} \B(\C^{d/})$ between the classifying spaces 
of the induced $\infty$-undercategories, then $\B(\C^{d/})$ is 
the homotopy fiber of $\B F$ over $d$ and thus, $\B(\C^{d/}) \hra \B\C \xra{\B F} \B\D$ is a fiber sequence. 
\end{theorem}

\begin{note} 
Quillen originally proved Theorem~B in \cite{Quil} for categories. 
Theorem 5.16 in \cite{AF2} generalizes the result for $\infty$-categories, 
which is the statement of Quillen's Theorem~B given above. 
\end{note}

We apply Quillen's Theorem~B for the situation at hand 
and identify the homotopy fibers of $\B\ev_0$ as follows. 

\begin{lemma} \label{fiber1pt2} 
The fiber of the map of spaces
$$\ds \B{\sf Fun}^{W_n}\Bigl([1],\Exit\bigl(\Ranu(\un{\R}^n)\bigr)\Bigr) 
\xra{\B \tx{ev}_0} \B{\sf Fun}^{W_n}\Bigl([0],\Exit\bigl(\Ranu(\un{\R}^n)\bigr)\Bigr)$$
 over an object $\un{S}=S_n \ra \cdots \ra S_1$ is equivalent to the classifying space 
 of the fiber of $\ev_0$ over $\un{S}$, 
$$\ds (\B\ev_0)^{-1}(\un{S}) \simeq \B\bl(\ev_0^{-1}(\un{S})\br).$$ 
\end{lemma}

\begin{proof}
Fix a morphism $\ds \un{S} \xra{\alpha} \un{S}'$ in 
$\ds {\sf Fun}^{W_n}\Bigl([0],\Exit\bigl(\Ranu(\un{\R}^n)\bigr)\Bigr)$. 
Recall that the induced monodromy functor $\alpha^\ast$ is defined as the composite
\[ \begin{tikzcd}
\ev_0^{-1}(\un{S}') \ar[hookrightarrow]{d} \ar{r}{\alpha^\ast} & \ev_0^{-1}(\un{S}) \\
 {\sf Fun}^{W_n}\Bigl([1],\Exit\bigl(\Ranu(\un{\R}^n)\bigr)\Bigr)^{\un{S}'_/} \ar{r}{- \circ \alpha} 
 &  {\sf Fun}^{W_n}\Bigl([1],\Exit\bigl(\Ranu(\un{\R}^n)\bigr)\Bigr)^{\un{S}_/} \ar{u}[swap]{\mu} 
\end{tikzcd} \]
where $\mu$ is right adjoint to inclusion (which exists by Lemma~2.20 of \cite{AF2}). 
The diagram induced upon taking classifying spaces 
\begin{equation} \label{Bmono} 
\begin{tikzcd}
\B\ev_0^{-1}(\un{S}') \ar{d}[swap]{\simeq} \ar{r}{\B\alpha^\ast} & \B\ev_0^{-1}(\un{S}) \\
 \B\tx{Fun}^{W_n}\Bigl([1],\Exit\bigl(\Ranu(\un{\R}^n)\bigr)\Bigr)^{\un{S}'_/} \ar{r}{\B(- \circ \alpha)} 
 &  \B\tx{Fun}^{W_n}\Bigl([1],\Exit\bigl(\Ranu(\un{\R}^n)\bigr)\Bigr)^{\un{S}_/} \ar{u}[swap]{\simeq} 
\end{tikzcd} 
\end{equation}
yields the vertical arrows to be equivalences by Corollary~2.1.28 of \cite{AMG}.

We seek to show that $\B\alpha^\ast$ is an equivalence. Indeed, our desired equivalence 
$$\ds (\B\ev_0)^{-1}(\un{S}) \simeq \B\bl(\ev_0^{-1}(\un{S})\br)$$ 
will follow by Quillen's Theorem~B; let us explain further. If $\B\alpha^\ast$ 
is an equivalence, then, according to (\ref{Bmono}), $\B(- \circ \alpha)$ is an equivalence. 
By Quillen's Theorem~B, $\B(- \circ \alpha)$ being an equivalence yields an equivalence
$$ (\B\ev_0)^{-1}(\un{S}) \simeq \B{\sf Fun}^{W_n}([1],\Exit\bigl(\Ranu(\un{\R}^n)\bigr))^{\un{S}_/}.$$
The upward vertical arrow in (\ref{Bmono}) is an equivalence between 
$\B\tx{Fun}^{W_n}\Bigl([1],\Exit\bigl(\Ranu(\un{\R}^n)\bigr)\Bigr)^{\un{S}_/}$ and  $\B\ev_0^{-1}(\un{S})$.
Thus, $ (\B\ev_0)^{-1}(\un{S})$ is equivalent to $\B\ev_0^{-1}(\un{S})$.

Now consider the diagram 
\begin{equation} \begin{tikzcd} 
\Fun^{W_n}\Bigl([1], \Exit\bigl(\Ranu(\un{\R}^n)\bigr)\Bigr) \ar{r}{\tx{frgt}} \ar{d}[swap]{\tx{ev}_0} 
& \mor\Bigl(\Exit\bigl(\Ranu(\R^n)\bigr)\Bigr) \ar{d}{\tx{ev}_0} \\
\Fun^{W_n}\Bigl([0], \Exit\bigl(\Ranu(\un{\R}^n)\bigr)\Bigr) \ar{r}{\simeq} 
& \Exit\bigl(\Ranu(\R^n)\bigr)^{\sim}
\end{tikzcd} \end{equation}
in which each vertical arrow is a Cartesian fibration. 
By Observation~\ref{mono}, the induced monodromy functor $\alpha^\ast$ 
is carried by the forgetful functor to the induced monodromy functor 
of the image of $\alpha$ under the forgetful functor, $\tx{frgt}(\alpha)^\ast$:
\begin{equation} \label{diag1} \begin{tikzcd}
\ev_0^{-1}(\un{S}') \ar{r}{\alpha^\ast} \ar{d}[swap]{\text{frgt}} 
& \ev_0^{-1}(\un{S}) \ar{d}{\tx{frgt}} \\
\ev_0^{-1}(S_n') \ar{r}{\tx{frgt}(\alpha)^\ast} & \ev_0^{-1}(S_n).
\end{tikzcd} \end{equation}
Observe that $\ds \tx{frgt}(\alpha)^\ast$ is an equivalence precisely 
because the image of the morphism $\alpha$ under the forgetful functor, 
$ S_n \xra{\alpha_n} S_n'$,
is a bijection.
We apply the universal property of localization to the canonical localization 
$\ds  \ev_0^{-1}(\un{S}) \ra \B \ev_0^{-1}(\un{S})$ to obtain
\begin{equation} \label{diag2} 
\begin{tikzcd}
 \ev_0^{-1}(\un{S}) \ar{r}{\tx{frgt}} \ar{d} 
& \ds \ev_0^{-1}(S_n) \simeq \prod_{s \in S_n} \Exit\Bigl(\Ranu(T_s\R^n)_{\Sigma_r}\Bigr)^{\sim} \\
\B\ev_0^{-1}(\un{S}) \simeq \ds \prod_{s \in S_n} \Exit\Bigl(\Ranu(T_s\R^n)_{\Sigma_r}\Bigr)^{\sim} 
\ar[dashrightarrow]{ur}[swap]{\exists !} &
\end{tikzcd} 
\end{equation}
and observe that such a filler must be an equivalence. 
We paste (\ref{diag1}) and (\ref{diag2}) for $\un{S}$ and $\un{S}'$
 together to see that $\B\alpha^\ast$ is an equivalence:
\[ \begin{tikzcd}
\ev_0^{-1}(\un{S}') \ar{dr}{\tx{loc}}\ar{rrr}{\alpha^\ast} \ar{dd}[swap]{\text{frgt}} 
&&& \ev_0^{-1}(\un{S}) \ar{dd}{\tx{frgt}} \ar{dl}{\tx{loc}} \\
&\B\ev_0^{-1}(\un{S}') \ar{dl}{\simeq} \ar{r}{\B\alpha^\ast}  
& \B\ev_0^{-1}(\un{S}) \ar{dr}{\simeq} & \\
\ev_0^{-1}(S_n') \ar{rrr}{\simeq} &&& \ev_0^{-1}(S_n).
\end{tikzcd} \]

According to (\ref{Bmono}), the map $\B\alpha^\ast$ being an equivalence 
implies that the map $\B(- \circ \alpha)$ is an equivalence. Thus, by Quillen's Theorem~B, 
the fiber of $\B\ev_0$ over $\un{S}$ is identified as 
$$\B{\sf Fun}^{W_n}([1],\Exit\bigl(\Ranu(\un{\R}^n)\bigr))^{\un{S}_/}$$ 
which according to (\ref{Bmono}) is equivalent to $\B\ev_0^{-1}(\un{S})$.
\end{proof}

Lemma~\ref{pis1} which states that the map induced by the forgetful functor
$$ \B\Fun^{W_n}\Bigl([1], \Exit\bigl(\Ranu(\un{\R}^n)\bigr)\Bigr) 
\ra \mor\Bigl(\Exit\bigl(\Ranu(\R^n)\bigr)\Bigr)$$
is an equivalence, follows almost immediately. 

\begin{proof}[Proof of Lemma~\ref{pis1}]
We first recall (\ref{hey})
\begin{equation} \label{fibrations}  
\begin{tikzcd} 
\B\Fun^{W_n}\Bigl([1], \Exit\bigl(\Ranu(\un{\R}^n)\bigr)\Bigr) 
\ar{rr}{\tx{frgt}} \ar{d}[swap]{\B\tx{ev}_0} 
&& \mor\Bigl(\Exit\bigl(\Ranu(\R^n)\bigr)\Bigr) \ar{d}{\tx{ev}_0} \\
\B\Fun^{W_n}\Bigl([0], \Exit\bigl(\Ranu(\un{\R}^n)\bigr)\Bigr) \ar{rr}{\simeq} 
\ar{rr}[swap]{\tx{Lemma}~\ref{pis0}} && \Exit\bigl(\Ranu(\R^n)\bigr)^{\sim}.
\end{tikzcd} 
\end{equation}
First note that the bottom horizontal equivalence in (\ref{fibrations})
is induced by the forgetful functor from 
$\Exit\bigl(\Ranu(\un{\R}^n)\bigr)$ to $\Exit\bigl(\Ranu(\R^n)\bigr)$
since the map between total spaces is induced by the forgetful functor.

Lemmas \ref{fiber2} and \ref{fiber1pt2} identify the fibers over the 
righthand and lefthand vertical arrows of (\ref{fibrations}), respectively, as equivalent. 
Thus, the induced long exact sequence in homotopy of each fibration
induces a weak homotopy equivalence between the total spaces, 
which is, in fact, a homotopy equivalence since the total spaces are CW complexes. 
\end{proof}

\subsection{The final lemma} \label{oh} 
The goal of this subsection is prove Lemma~\ref{class2}
which states that the simplicial space 
$\B \Fun^{W_n}\Bigl([\bullet], \Exit\bigl(\Ranu(\un{\R}^n)\bigr)\Bigr)$ 
is a complete Segal space. 
In so doing we verify the hypothesis of Theorem~\ref{BcSS}
and complete our proof of our main result Theorem~\ref{loc}.
The proof of Lemma \ref{class2} is technical using 
Cartesian fibrations and Quillen's Theorem~B, 
and builds off of arguments developed in \S\ref{1}. 
We begin with a key lemma. 

\begin{lemma} \label{Qlemma} Given a pullback of $\infty$-categories 
\begin{equation} \label{okie} \begin{tikzcd}
\E \pb \ar{d}[swap]{\pi} \ar{r}{G} & \E' \ar{d}{\pi'} \\
\B \ar{r}{F} & \B' 
\end{tikzcd} \end{equation}
in which $\pi'$ is a Cartesian fibration, if $\pi'$ satsifies Quillen's Theorem~B, then so does $\pi$. 
\end{lemma}

\begin{proof} 
Let $b \xra{f} b'$ be a morphism in $\B$. Lemma~2.8 of \cite{AF2} 
states that Cartesian fibrations are closed under base change. 
Thus, $\pi'$ being a Cartesian fibration and (\ref{okie}) 
being a pullback implies that $\pi$ is a Cartesian fibration. 
By Observation~\ref{mono} then, $G$ carries the induced monodromy functor 
$f^\ast$ to the induced monodromy functor $F(f)^\ast$ and yields equivalences between the fibers
\[ \begin{tikzcd}
\pi^{-1}(b') \ar{r}{f^\ast}  \ar{d}[swap]{\simeq}& \pi^{-1}(b)   \ar{d}{\simeq}  \\
\pi'^{-1}(F(b')) \ar{r}{F(f)^\ast}  & \pi'^{-1}(F(b)).
\end{tikzcd} \]
Combining this diagram with the definition of the monodromy functor yields
\begin{equation}  
\begin{tikzcd}
\E^{b'_/} \ar{r}{- \circ f} & \E^{b_/}  \\ 
\pi^{-1}(b') \ar{r}{f^\ast} \ar[hookrightarrow]{u} \ar{d}[swap]{\simeq}
& \pi^{-1}(b)  \ar[hookrightarrow]{u}  \ar{d}{\simeq}  \\
\pi'^{-1}(F(b')) \ar{r}{F(f)^\ast} \ar[hookrightarrow]{d} & \pi'^{-1}(F(b)) \ar[hookrightarrow]{d} \\
\E'^{F(b')_/} \ar{r}{- \circ F(f)} & \E'^{F(b)_/}.
\end{tikzcd} 
\end{equation}
The diagram induced upon taking classifying spaces yields the desired result. 
Indeed, $\pi'$ satisfying Quillen's Theorem~B implies $\B(- \circ F(f))$ 
is a equivalence and thus, each horizontal arrow resulting 
between classifying spaces is an equivalence, 
which in particular means $\B(- \circ f)$ in the following diagram is an equivalence
\[  \begin{tikzcd}
\B\E^{b'_/} \ar{r}{\B(- \circ f)} & \B\E^{b_/}  \\ 
\B\pi^{-1}(b') \ar{r}{\B f^\ast} \ar[hookrightarrow]{u}{\simeq} \ar{d}[swap]{\simeq}
& \B\pi^{-1}(b)  \ar[hookrightarrow]{u}{\simeq}  \ar{d}{\simeq}  \\
\B\pi'^{-1}(F(b')) \ar{r}{\B F(f)^\ast} \ar[hookrightarrow]{d}{\simeq} 
& \B\pi'^{-1}(F(b)) \ar[hookrightarrow]{d}{\simeq} \\
\B\E'^{F(b')_/} \ar{r}{\B(- \circ F(f))} & \B\E'^{F(b)_/}.
\end{tikzcd} \]
\end{proof}

\begin{observation} \label{pb} 
The following diagram of $\infty$-categories is pullback
\begin{equation} \label{pullback} 
\begin{tikzcd}[row sep=15, column sep=0]
\Fun^{W_n}\Bigl([p], \Exit\bigl(\Ranu(\un{\R}^n)\bigr)\Bigr) \pb \ar{r}{ \tau}\ar{d}[swap]{ \sigma} 
& \Fun^{W_n}\Bigl(\{1-p<p\}, \Exit\bigl(\Ranu(\un{\R}^n)\bigr)\Bigr) \ar{d}{ s} \\
\Fun^{W_n}\Bigl(\{0<\cdots < p-1\}, \Exit\bigl(\Ranu(\un{\R}^n)\bigr)\Bigr) \ar{r}{ t} 
& \Fun^{W_n}\Bigl(\{p-1\}, \Exit\bigl(\Ranu(\un{\R}^n)\bigr)\Bigr).
\end{tikzcd} 
\end{equation}
Indeed, for an $\infty$-category $\C$, $\Fun([\bullet], \C)$ 
satisfies the Segal condition (Def. \ref{cSegspc}),
i.e., for each $p \geq 2$, the diagram obtained by replacing 
$ \Fun^{W_n}\Bigl([p], \Exit\bigl(\Ranu(\un{\R}^n)\bigr)\Bigr) $ 
with $\Fun\bl([p], \C\br)$ in (\ref{pullback}) is pullback. 
Using this, it is straightforward to show (\ref{pullback}) is pullback.
\end{observation}

We are finally ready to prove Lemma~\ref{class2} 
which states that 
$\B \Fun^{W_n}\Bigl([\bullet], \Exit\bigl(\Ranu(\un{\R}^n)\bigr)\Bigr)$ 
is a complete Segal space. 

\begin{proof}[Proof of Lemma~\ref{class2}]
First we show that  
$\ds \B\tx{Fun}^{W_n}\Bigl([\bullet],\Exit\bigl(\Ranu(\un{\R}^n)\bigr)\Bigr)$ 
satisfies the Segal condition. 
Consider the diagram of spaces obtained by taking the classifying spaces of (\ref{pullback})
\begin{equation} \label{Bpullback} 
\begin{tikzcd}[row sep=25, column sep=2]
\B\Fun^{W_n}\Bigl([p], \Exit\bigl(\Ranu(\un{\R}^n)\bigr)\Bigr) 
\ar{r}{\B \tau}\ar{d}[swap]{\B \sigma} 
& \B\Fun^{W_n}\Bigl(\{p-1 < p\}, \Exit\bigl(\Ranu(\un{\R}^n)\bigr)\Bigr) \ar{d}{\B s} \\
\ds \B\Fun^{W_n}\Bigl(\{0<\cdots<p-1\}, \Exit\bigl(\Ranu(\un{\R}^n)\bigr)\Bigr) \ar{r}{\B t} 
& \B\Fun^{W_n}\Bigl(\{p-1\}, \Exit\bigl(\Ranu(\un{\R}^n)\bigr)\Bigr).
\end{tikzcd} 
\end{equation}
To show that this diagram is a pullback, 
we will show that the map induced between fibers of (\ref{Bpullback}) is an equivalence. 
By Observation~\ref{Cart}, the functor $s$ in (\ref{pullback}) is a Cartesian fibration. 
Further, in the proof of Lemma~\ref{fiber1pt2}, 
we showed that $s$ satisfies Quillen's Theorem~B. 
Thus, (\ref{pullback}) satsifies the hypothesis' of Lemma~\ref{Qlemma} 
and we identify the fibers of $\B\sigma$ and $\B s$ over the objects 
$$\ds \un{S}_0 \to \cdots \to \un{S}_{p-1} \shs \shs \tx{in} \shs \shs 
\B\Fun^{W_n}\Bigl(\{0<\cdots<p-1\}, \Exit\bigl(\Ranu(\un{\R}^n)\bigr)\Bigr)$$
 and $$\ds \un{S}_{p-1} \shs \shs \tx{in} \shs \shs \B\Fun^{W_n}\Bigl(\{p-1\}, 
 \Exit\bigl(\Ranu(\un{\R}^n)\bigr)\Bigr)$$ 
respectively, as the classifying spaces of the fibers of $\sigma$ and $s$ over  
$\ds \un{S}_0 \to \cdots \to \un{S}_{p-1}$ and $\ds \un{S}_{p-1}$, respectively
$$(\B\sigma)^{-1}(\un{S}_0 \to \cdots \to \un{S}_{p-1}) 
\simeq \B\sigma^{-1}(\un{S}_0 \to \cdots \to \un{S}_{p-1})$$
\centerline{and} 
$$(\B s)^{-1}(\un{S}_{p-1}) \simeq \B s^{-1}(\un{S}_{p-1}).$$
Therefore, because (\ref{pullback}) being a pullback implies 
an equivalence between fibers induced by $\tau$
$$\ds \tau_{|}: \sigma^{-1}(\un{S}_0 \to \cdots \to \un{S}_{p-1}) \xra{\simeq} s^{-1}(\un{S}_{p-1})$$ 
there results an equivalence between fibers of (\ref{Bpullback}) given by $\B \tau_{|}$
$$(\B\sigma)^{-1}(\un{S}_0 \to \cdots \to \un{S}_{p-1}) 
\simeq \B\sigma^{-1}(\un{S}_0 \to \cdots \to \un{S}_{p-1}) 
\xra{\simeq}  \B s^{-1}(\un{S}_{p-1}) \simeq (\B s)^{-1}(\un{S}_{p-1})$$
which verifies that (\ref{Bpullback}) is a pullback.

Now that we have show that 
$\ds \B{\sf Fun}^{W_n}\Bigl([\bullet],\Exit\bigl(\Ranu(\un{\R}^n)\bigr)\Bigr)$ 
satisfies the Segal condition, we know that its values on $[0]$ and $[1]$ 
determine all of its higher $[p]$ values. 
In particular this means that
Lemma~\ref{class1} extends to an equivalence of spaces 
\begin{equation} \label{cat}
\ds \B{\sf Fun}^{W_n}\Bigl([p],\Exit\bigl(\Ranu(\un{\R}^n)\bigr)\Bigr) 
\simeq \Hom_{\Cat_{\infty}}\Bigl([p], \Exit\bigl(\Ranu(\R^n)\bigr)\Bigr)
\end{equation} 
for each $p \geq 0$.
Because $\Exit\bigl(\Ranu(\R^n)\bigr)$ is a complete Segal space,
$\ds \Hom_{\Cat_{\infty}}\Bigl([\bullet], \Exit\bigl(\Ranu(\R^n)\bigr)\Bigr)$ 
is a complete Segal space as well.
Thus, $\ds \B{\sf Fun}^{W_n}\Bigl([\bullet],\Exit\bigl(\Ranu(\un{\R}^n)\bigr)\Bigr) $ 
by equivalence (\ref{cat}).
\end{proof} 

We have now verified the hypothesis of Theorem~\ref{BcSS}, 
which means that the localization of $\Exit\bl(\Ranu(\un{\R}^n)\br)$ on $W_n$ 
is equivalent to the simplicial space $\Fun^{W_n}\Bigr([\bullet], \Exit\bl(\Ranu(\un{\R}^n)\br)\Bigr)$. 
We are finally ready to prove Lemma~\ref{preloc}, 
which states that the forgetful functor from
$\Exit\bigl(\Ranu(\un{\R}^n)\bigr)$ to $\Exit\bigl(\Ranu(\R^n)\bigr)$
is a localization on $W_n$.

\begin{proof}[Proof of Lemma~\ref{preloc}]
In the proof of the Lemma~\ref{class2}, we showed an equivalence
\begin{equation} \label{ahhh} 
\B\Fun^{W_n}\Bigl([\bullet], \Exit\bigl(\Ranu(\un{\R}^n)\bigr)\Bigr) 
\simeq \tx{Hom}_{\Cat_{\infty}}\Bigl([\bullet], \Exit\bigl(\Ranu(\R^n)\bigr)\Bigr) 
\end{equation}
 which, in particular means that the hypothesis of Theorem~\ref{BcSS} is satisfied. 
 Thus, by Theorem~\ref{BcSS},
$$\Exit\bigl(\Ranu(\un{\R}^n)\bigr)[W_n^{-1}] 
\simeq \B\Fun^{W_n}\Bigl([\bullet], \Exit\bigl(\Ranu(\un{\R}^n)\bigr)\Bigr).$$
Then, by the equivalence (\ref{ahhh}), we have an equivalence of simplicial spaces
$$\Exit\bigl(\Ranu(\un{\R}^n)\bigr)[W_n^{-1}]  
\simeq \tx{Hom}_{\Cat_{\infty}}\Bigl([\bullet], \Exit\bigl(\Ranu(\R^n)\bigr)\Bigr)$$ 
which  establishes that $\Exit\bigl(\Ranu(\un{\R}^n)\bigr)$ 
localizes on $W_n$ to $\Exit\bigl(\Ranu(\R^n)\bigr)$. 

Note that this localization is given by the forgetful functor 
from $\Exit\bigl(\Ranu(\un{\R}^n)\bigr) $ to $\Exit\bigl(\Ranu(\R^n)\bigr)$ 
because our identification of  $\B \Fun^{W_n}\Bigl([\bullet], \Exit\bigl(\Ranu(\un{\R}^n)\bigr)\Bigr)$ 
with  $\Hom_{\Cat_\infty}\Bigl([\bullet], \Exit\bigl(\Ranu(\R^n)\bigr)\Bigr)$ 
was induced by the forgetful functor. (This was proven in Lemma~\ref{pis1}.)

Lastly, we need to show that this localization is over $\Fin^\op$. 
In Observation~\ref{forget}, we observed that the forgetful functor from 
$ \Exit(\Ran^u(\un{\R}^n))$ to $\Exit(\Ran^u(\R^n))$ 
is naturally over $\Fin^\op$ by just remembering the data of underlying sets at the $\R^n$ level. 
Then, by the universal property of localization, we have
\[ \begin{tikzcd} 
\Exit\bigl(\Ranu(\un{\R}^n)\bigr) \ar{r}{\tx{frgt}} \ar{dr}[swap]{\phi_n} 
& \Exit\bigl(\Ranu(\R^n)\bigr) \ar[phantom]{r}{\simeq} \ar{d}{\phi} 
& \Exit\bigl(\Ranu(\un{\R}^n)\bigr)[W_n^{-1}] \ar[dashrightarrow]{dl}{\exists !}  \\
& \Fin^\op. & \end{tikzcd} \]
The unique existence of such a filler is guaranteed because each morphism 
in $W_n$ gets carried to isomorphisms in $\Fin^\op$ under $\phi_n$.
Thus, we see that the forgetful functor from 
$\Exit\bigl(\Ranu(\un{\R}^n)\bigr)$ to $\Exit\bigl(\Ranu(\R^n)\bigr)$ yields a localization over $\Fin^\op$. 
\end{proof}

Theorem~\ref{inteq} and Lemma~\ref{preloc} together immediately imply our main result, 
Theorem~\ref{loc}, that the category $\mathbf{\Theta}_n^\Act$ 
localizes to the $\infty$-category $\Exit(\Ranu(\R^n))$. 

An immediate corollary to Theorem~\ref{loc} is that the wreath product 
decomposition of $\mathbf{\Theta}_2$ induces a likewise decomposition of $\Exit(\Ranu(\R^2))$. 

\begin{cor} \label{cor 2}
There is a localization of $\infty$-categories
$$\Exit\bigl(\Ranu(\R)\bigr) \wr \Exit\bigl(\Ranu(\R)\bigr) \to \Exit\bigl(\Ranu(\R^2)\bigr)$$
from the two fold wreath product of the exit-path $\infty$-category 
of the unital Ran space of $\R$ with itself 
and the exit-path $\infty$-category of the unital Ran space of $\R^2$. 
\end{cor}

\section{A nonunital version of the main result} \label{bigcor}

In this section, we prove a corollary to Theorem~\ref{loc}
which identifies the exit-path $\infty$-category 
of the Ran space of $\R^n$ $\Exit(\Ran(\R^n))$ as a localization of 
a certain subcategory of $\mathbf{\Theta}_n^\Act$. 
This result specializes the localization of Theorem~\ref{loc} 
by restricting $\Exit\bigl(\Ranu(\R^n)\bigr)$ to its $\infty$-subcategory $\Exit(\Ran(\R^n))$.

Recall the definition of $\Exit\bigl(\Ran(M)\bigr)$ (Def. \ref{little exit}) 
as the pullback of simplicial spaces
\[ \begin{tikzcd}
\Exit\bigl(\Ran(M)\bigr) \ar[hookrightarrow]{r} \pb \ar{d} 
& \Exit\bigl(\Ranu(M)\bigr)\ar{d}{\phi} \\
\bl(\Fin_{\neq \emptyset}^\surj \br)^\op \ar[hookrightarrow]{r} &  \Fin^\op
\end{tikzcd} \]
where $\Fin^\surj_{\neq \emptyset}$ is the subcategory of finite sets 
consisting of nonempty sets and only surjective morphisms.
As a pullback of $\infty$-categories, note that $\Exit\bigl(\Ran(M)\bigr)$ 
too is an $\infty$-category (Cor. \ref{wow}).

\begin{remark}
The $\infty$-category $\Exit\bigl(\Ran(M)\bigr)$ is equivalent to
Lurie's exit-path $\infty$-category 
construction (Def. \ref{exitdef}) applied to the Ran space of $M$ (Def. \ref{Ran}).
\end{remark}

The main result of this section identifies $\Exit(\Ran(\R^n))$ as a localization 
of a certain subcategory of $\mathbf{\Theta}_n^\Act$ defined as follows.

\begin{definition} \label{exit} 
The category $\mathbf{\Theta}_n^\exit$ is the 
subcategory of $\mathbf{\Theta}_n^\Act$ defined as the pullback
\[ \begin{tikzcd}
\mathbf{\Theta}_n^\exit \ar[hookrightarrow]{r} \pb \ar{d} 
& \mathbf{\Theta}_n^\Act \ar{d}{\tau} \\
\Fun\Bigl(\{1< \cdots < n\}, \bl(\Fin_{\neq \emptyset}^\surj \br)^\op\Bigr) \ar[hookrightarrow]{r} 
& \Fun\bigl(\{1< \cdots < n\}, \Fin^\op\bigr)
\end{tikzcd} \]
where we recall the functor $\tau$ from Observation~\ref{layers} 
(defined by the truncation functor $\tr_i$ and $\gamma_i$). 
 \end{definition}

Heuristically $\mathbf{\Theta}_n^\exit$ consists of healthy trees 
as its objects and all those morphisms that induce surjections between the sets of leaves. 
We are now able to state the main result of this section 
which articulates the sense in which $\Exit(\Ran(\R^n))$ is identified 
combinatorially in terms of $\mathbf{\Theta}_n$. 

\begin{cor} \label{locexit} 
For $n \geq 1$ there is a localization of $\infty$-categories
$$\mathbf{\Theta}_n^\exit \ra \Exit\bigl(\Ran(\R^n)\bigr) $$
which is contravariant over the category of nonempty finite sets and surjective morphisms.
\end{cor}

Recall the category $W_n^\h$ (Def. \ref{Whealthy}) which consists of healthy trees as 
its objects and all those morphisms in $\mathbf{\Theta}_n$ that 
induce bijections on the sets of leaves under $\gamma_n$ to $\Fin$.
(Note that we are considering $W_n^\h$ as a subcategory of 
$\mathbf{\Theta}_n^\Act$ in light of the equivalence 
$\Exit\bigl(\Ranu(\un{\R}^n)\bigr) \simeq \mathbf{\Theta}_n^\Act$ (Thm.~\ref{inteq}).)
The localizing subcategory of Corollary~\ref{locexit} is nearly 
equivalent to the subcategory $W_n^\h$.
The only difference is that it does not contain the empty tree as an object 
nor the identity morphism on the empty tree.
The definition of the localizing subcategory of Corollary~\ref{locexit} is as follows.

\begin{definition} 
The subcategory $(W_n^\h)_{\neq \emptyset}$ of $\mathbf{\Theta}_n^\exit$ is defined to be the pullback
\[ \begin{tikzcd}
(W_n^\h)_{\neq \emptyset} \pb \ar[hookrightarrow]{r} \ar[hookrightarrow]{d} & W_n^\h \ar[hookrightarrow]{d} \\
\mathbf{\Theta}_n^\exit \ar[hookrightarrow]{r} & \mathbf{\Theta}_n^\Act
\end{tikzcd} \] 
of categories. 
\end{definition}

Heuristically then, $(W_n^\h)_{\neq \emptyset} \hra \mathbf{\Theta}_n^\exit$ 
has all those nonempty, healthy trees of $\mathbf{\Theta}_n$ 
as its objects and only those morphisms that induce bijections between the sets of leaves. 

The proof of Corollary~\ref{locexit} is similar to that of Lemma~\ref{preloc} 
in that we identify the localization of $\mathbf{\Theta}_n^\exit$ on $(W_n^\h)_{\neq \emptyset}$ 
using Theorem~\ref{BcSS}, which states that if the simplicial space 
$\B\Fun^{(W_n^\h)_{\neq \emptyset}}\bl([\bullet], \mathbf{\Theta}_n^\exit \br)$ is a complete Segal space, 
then it is equivalent to the localization of $\mathbf{\Theta}_n^\exit$ on $(W_n^\h)_{\neq \emptyset}$. 
We do this by extrapolating the argument of Lemma~\ref{preloc}, 
using the fact that the domain and codomain of Corollary~\ref{locexit} 
are $\infty$-subcategories of the domain and codomain of the localization of Theorem~\ref{loc}. 

The first lemma that we need follows from Lemma~\ref{adj} 
wherein we showed that there is an adjunction between $W_n^\h$ and $W_n$. 
Before stating the lemma we introduce a subcategory of $\mathbf{\Theta}_n^\Act$
that will be useful for the rest of this section.

\begin{definition}
The category $\mathbf{\Theta}_n^{\Act, \h}$ is the full subcategory of 
$\mathbf{\Theta}_n^\Act$ consisting of all those objects that are healthy trees. 
\end{definition}

\begin{lemma} \label{extadj} 
For each $p \geq 0$, the inclusion functor between $\infty$-categories 
$$\Fun^{W_n^\h} \bl([p], \mathbf{\Theta}_n^{\Act, \h}\br) 
\hra \Fun^{W_n}\bl([p], \mathbf{\Theta}_n^\Act \br)$$
 induces an equivalence between their classifying spaces
$$ \B \Fun^{W_n^\h} \bl([p], \mathbf{\Theta}_n^{\Act, \h}\br) \xra{\simeq} 
\B\Fun^{W_n}\bl([p], \mathbf{\Theta}_n^\Act \br).$$
\end{lemma}

\begin{proof} 
First observe that we can describe the subcategory $W_n^\h$ of $W_n$ 
as the following pullback of categories over $\mathbf{\mathbf{\Theta}}_n^{\Act, \h}$
\[ \begin{tikzcd}
W_n^\h \pb \ar[hookrightarrow]{d} \ar[hookrightarrow]{r} & W_n \ar[hookrightarrow]{d} \\
\mathbf{\mathbf{\Theta}}_n^{\Act, \h} \ar[hookrightarrow]{r} & \mathbf{\mathbf{\Theta}}_n^\Act 
\end{tikzcd} \]
where we recall that $W_n$ consists of all the same objects as 
$\mathbf{\Theta}_n^\Act$ and all those morphisms that induce bijections on
the sets of leaves, and $W_n^\h$ is the full subcategory of $W_n$ consisting of only 
those trees that are healthy. 

In Lemma~\ref{adj} we showed that the inclusion functor $W^\h_n \hra W_n$ is a right adjoint. 
The reader may observe that no where in the proof did we use that the 
morphisms of $W^\h_n$ and $W_n$ induce bijections between their sets of leaves. 
Thus, Lemma~\ref{adj} immediately extends to an adjunction between 
$\mathbf{\Theta}_n^{\Act, \h}$ and $\mathbf{\Theta}_n^\Act$ whose right adjoint is given by inclusion. 
Further, observe that the unit transformation of this adjunction is given by morphisms in $W_n$. 
Indeed, for each tree $T \in \mathbf{\Theta}_n^\Act$, the morphism assigned to $T$ 
by the unit is $T \xra{\epsilon_T} P_n(T)$, which, in particular, induces a bijection on the leaves, 
and is thus, in $W_n$. In identifying that the unit of the right adjoint 
$\mathbf{\Theta}_n^{\Act, \h} \hra \mathbf{\Theta}_n^\Act$  is given by morphisms in $W_n$, 
we may extend this adjunction to an adjunction between 
$\Fun^{W_n^\h}\bl([p], \mathbf{\Theta}_n^{\Act, \h}\br)$ and 
$\Fun^{W_n}\bl([p], \mathbf{\Theta}_n^\Act \br)$ whose right adjoint is inclusion. 

Corollary 2.1.28 in \cite{AMG} states that the classifying space of 
an adjunction is an equivalence of spaces. 
Thus, upon taking the classifying space of the right adjoint  
$$\Fun^{W_n^\h} \bl([p], \mathbf{\Theta}_n^{\Act, \h}\br) 
\hra \Fun^{W_n}\bl([p], \mathbf{\Theta}_n^\Act \br),$$ 
there results the desired equivalence of spaces. 
\end{proof}

We need another lemma before commencing with the proof of Corollary~\ref{locexit}.
First note that the inclusion of the subcategory 
$\mathbf{\Theta}_n^\exit \hra \mathbf{\Theta}_n^{\Act, \h}$ 
together with the induced inclusion of the respective subcategories $(W_n^\h)_{\neq \emptyset} \hra W_n^\h$ 
guarantees that for each $p \geq 0$, the induced map 
$ \Fun^{(W_n^\h)_{\neq \emptyset}} \bl([p], \mathbf{\Theta}_n^\exit \br) 
\hra \Fun^{W_n^\h}\bl([p], \mathbf{\Theta}_n^{\Act, \h}\br)$ 
is also an inclusion of a $\infty$-subcategory. 
Towards the proof of Corollary~\ref{locexit} we need that the map induced 
between the classifying spaces of this inclusion is, in particular, a monomorphism. 
This is articulated by the following lemma. 

\begin{lemma} \label{mono3} For each $p \geq 0$ the inclusion functor 
$$\Fun^{(W_n^\h)_{\neq \emptyset}} \bl([p], \mathbf{\Theta}_n^\exit \br) 
\hra \Fun^{W_n^\h}\bl([p], \mathbf{\Theta}_n^{\Act, \h}\br)$$
 induces a monomorphism between classifying spaces
$$ \B \Fun^{(W_n^\h)_{\neq \emptyset}} \bl([p], \mathbf{\Theta}_n^\exit \br) 
\hra \B\Fun^{W_n^\h}\bl([p], \mathbf{\Theta}_n^{\Act, \h}\br).$$
\end{lemma}

To prove this lemma we need a technical result involving the following notion.

\begin{definition}\label{d1}
A functor $\C \to \D$ between $\infty$-categories is an 
\emph{inclusion of a cofactor} if there is an $\infty$-category $\E$ 
and an equivalence between $\infty$-categories under $\C$
\[ 
\C \coprod \E
~\cong~
\D
~.
\]
\end{definition}

The technical result we need to prove Lemma~\ref{mono3} is as follows.

\begin{lemma}\label{T1}
A functor $\C \xra{F} \D$ is an inclusion of a cofactor if and only if 
$F$ is a monomorphism and for each solid commutative square
\begin{equation} \label{cofact} 
\begin{tikzcd}  \left[ 0 \right] \ar{r} \ar{d}[swap]{\nu} & \C \ar{d}{F} \\
\left[ 1 \right] \ar[dashrightarrow]{ur}{\exists} \ar{r} & \D  \end{tikzcd} 
\end{equation}
for either $\nu := \langle 0 \rangle$ or $\nu := \langle 1 \rangle$, 
there exists a filler.
\end{lemma}

\begin{proof}
First, notice that if $F$ is a monomorphism and (\ref{cofact}) is satisfied 
(with the two possible lifts), then $\C\xra{F} \D$ is fully faithful.  
Consider the full $\infty$-subcategory $\E\subset \D$ consisting of those 
objects that are not isomorphic to objects in the image of $\C\to \D$.
Consider the canonical functor 
\[
\C\coprod \E 
\longrightarrow
\D
~,
\] 
which is canonically under $\C$. 
By design, this functor is essentially surjective, and fully faithful.
This established the implication that $F$ being a monomorphism 
and satisfying (\ref{cofact}) implies $\C\xra{F} \D$ is an inclusion of a cofactor.

We now show the converse. 
Suppose there is an $\infty$-category $\E$ together with 
an equivalence $\C\coprod \E \simeq \D$ under $\C$.  
Consider a solid diagram
\[  \begin{tikzcd}  \left[ 0 \right] \ar{r} \ar{d}[swap]{\nu} & \C \ar{d}{} \\
\left[ 1 \right] \ar[dashrightarrow]{ur}{\exists} \ar{r} &  \C \coprod \E  \end{tikzcd} \]
The functor $[0]\xra{\nu} [1]$ has the feature that every object in $[1]$ 
admits a morphism to or from an object in the image of $\nu$.  
It follows that there is a unique filler, as desired.
\end{proof}

We are now ready to prove Lemma~\ref{mono3}.

\begin{proof}[Proof of Lemma~\ref{mono3}] 
First we need to verify that the functor  
$$ \Fun^{(W_n^\h)_{\neq \emptyset}} \bl([p], \mathbf{\Theta}_n^\exit\br) 
\hra \Fun^{W_n^\h}\bl([p], \mathbf{\Theta}_n^{\Act, \h}\br)$$ 
is an inclusion of a cofactor. 
In other words, we will show that this functor is a monomorphism and satisfies (\ref{cofact}). 

Note that because  $ \Fun^{(W_n^\h)_{\neq \emptyset}} \bl([p], \mathbf{\Theta}_n^\exit\br) 
\hra \Fun^{W_n^\h}\bl([p], \mathbf{\Theta}_n^{\Act, \h}\br)$ 
is an inclusion of $\infty$-categories, it is in particular a monomorphism.

Similar to Observation~\ref{pb}, it is straightforward to verify that for each 
$p \geq 0$ the following diagram of $\infty$-categories
\begin{equation} \label{pullback1} 
\begin{tikzcd}
\Fun^{(W_n^\h)_{\neq \emptyset}}\bl([p], \mathbf{\Theta}_n^\exit\br) \pb \ar{r}{}\ar{d}[swap]{ } 
& \Fun^{(W_n^\h)_{\neq \emptyset}}\bl(\{1-p<p\}, \mathbf{\Theta}_n^\exit\br) \ar{d}{ } \\
\Fun^{(W_n^\h)_{\neq \emptyset}}\bl(\{0<\cdots < p-1\}, \mathbf{\Theta}_n^\exit\br) \ar{r}{ } 
& \Fun^{(W_n^\h)_{\neq \emptyset}}\bl(\{p-1\}, \mathbf{\Theta}_n^\exit\br).
\end{tikzcd} 
\end{equation}
is pullback, as it is the diagram obtained by just replacing 
$\mathbf{\Theta}_n^\exit$ with $\mathbf{\Theta}_n^{\Act, \h}$. 
Thus, to verify the diagram (\ref{cofact}) for our situation, namely that
 \begin{equation} 
 \begin{tikzcd} \label{now} 
 \left[ 0 \right] \ar{r} \ar{d}[swap]{\nu} 
 &  \Fun^{(W_n^\h)_{\neq \emptyset}} \bl([p], \mathbf{\Theta}_n^\exit\br) \ar[hookrightarrow]{d}{} \\
\left[ 1 \right] \ar[dashrightarrow]{ur}{\exists} \ar{r} 
&  \Fun^{W_n^\h} \bl([p], \mathbf{\Theta}_n^{\Act, \h}\br) 
 \end{tikzcd} 
 \end{equation}
 is satisfied for all $p \geq 0$, it suffices to show for the cases for $p=0,1$. 
 Indeed, it is straightforward to verify that the cases $p=0,1$ 
 imply each $p \geq 0$ case upon applying the universal property of pullback 
 from (\ref{pullback1}) and the diagram obtained by replacing $\mathbf{\Theta}_n^\exit$ 
 with $\mathbf{\Theta}_n^{\Act, \h}$  in (\ref{pullback}). 

Both cases, $p=0,1$, come down to the following observation. 
Each morphism in $W^\h_n$ between objects in $\mathbf{\Theta}_n^{\Act, \h}$ 
is a morphism in $\mathbf{\Theta}_n^\exit$. 
The root reason for this is that surjections enjoy the `2 out of 3' property. 
Namely, for any commutative triangle of morphisms among sets in which 
two of the morphisms are surjections, the third map is necessarily a surjection as well. 
For our situation, any morphism $T \xra{f} T'$ in $W_n^\h$ yields 
a bijection between the sets of leaves, $ \gamma_n(f): \gamma_n(T') \xra{\cong} \gamma_n(T)$. 
For any $1 \leq i \leq n-1$, in applying the natural transformation 
$\epsilon$ from Observation~\ref{tr-factor}, whose value on $T$ 
is the natural map between sets of leaves $\gamma_n(T) \xra{\epsilon_T} \gamma_i(\tr_i(T))$ 
induced by the structure of $T$, we obtain the following diagram among sets: 
\[ \begin{tikzcd} 
\gamma_n(T') \ar{r}{\cong} \ar{d}[swap]{\epsilon_{T'}} & \gamma_n(T) \ar{d}{\epsilon_T} \\
\gamma_i(\tr_i(T')) \ar{r}{\gamma_i(\tr_i(f))} & \gamma_i(\tr_i(T)). 
\end{tikzcd} \]
Observe that because $T$ and $T'$ are healthy trees, 
both $\epsilon_T$ and $\epsilon_{T'}$ are surjections. 
Then, by the `2 out of 3' property, $\gamma_i(\tr_i(T')) \xra{\gamma_i(\tr_i(f))} \gamma_i(\tr_i(T))$ 
is a surjection. 
Such a surjection at each $i$ guarantees that the image of $f$ under the functor 
$\mathbf{\Theta}_n^\Act \ra \Fun\bigl(\{1< \cdots < n\}, \Fin^\op\bigr)$  lands in 
$ \Fun\Bigl(\{1< \cdots < n\}, \bl(\Fin_{\neq \emptyset}^\surj \br)^\op\Bigr)$ 
and is thus, a morphism in $\mathbf{\Theta}_n^\exit$. 
Using this observation, we now verify (\ref{now}) for the cases $p=0,1$.

For $p=0$, the desired lift in
 \[ \begin{tikzcd}  
 \left[ 0 \right] \ar{r}{\langle T \rangle} \ar{d}[swap]{\langle 0 \rangle } 
 &  (W_n^\h)_{\neq \emptyset} \ar[hookrightarrow]{d}{} \\
\left[ 1 \right] \ar[dashrightarrow]{ur}{\exists} \ar{r}[swap]{\langle T \xra{f} T' \rangle} 
&  W_n^\h  \end{tikzcd} \]
is given  by selecting out the morphism $T \xra{f} T'$, 
which is in $\mathbf{\Theta}_n^\exit$ because each morphism in 
$W^\h_n$ between objects in $\mathbf{\Theta}_n^{\Act, \h}$ 
is a morphism in $\mathbf{\Theta}_n^\exit$, as previously discussed. 
A similar argument yields a lift for the square whose downward arrow on the left is $\langle 1 \rangle$. 

For $p=1$, the desired lift in
 \[ \begin{tikzcd}  
 \left[ 0 \right] \ar{r} \ar{d}[swap]{\langle 0 \rangle } 
 &  \Fun^{(W_n^\h)_{\neq \emptyset}} \bl([1], \mathbf{\Theta}_n^\exit\br) \ar[hookrightarrow]{d}{}  \\
\left[ 1 \right] \ar[dashrightarrow]{ur}{\exists} \ar{r}{\alpha} 
&  \Fun^{W_n^\h} \bl([1], \mathbf{\Theta}_n^{\Act, \h}\br)  \end{tikzcd} \]
is again given by $\alpha$, which is straightforward to check upon applying 
the fact discussed above, that each morphism in $W^\h_n$ between objects 
in $\mathbf{\Theta}_n^{\Act, \h}$ is a morphism in $\mathbf{\Theta}_n^\exit$. 
A similar argument applies for  the square whose downward arrow on the left is $\langle 1 \rangle$. 

Thus, $\Fun^{(W_n^\h)_{\neq \emptyset}} \bl([p], \mathbf{\Theta}_n^\exit\br) 
\hra \Fun^{W_n^\h}\bl([p], \mathbf{\Theta}_n^{\Act, \h}\br)$ 
is an inclusion of a cofactor, meaning the target is equivalent to a coproduct, 
one term of which is the source. 
Because the classifying space respects colimits, the induced map between classifying spaces 
$\B \Fun^{(W_n^\h)_{\neq \emptyset}} \bl([p], \mathbf{\Theta}_n^\exit\br) 
\hra \B \Fun^{W_n^\h}\bl([p], \mathbf{\Theta}_n^{\Act, \h}\br)$ is, in particular, still a monomorphism. 
\end{proof}

Finally we are equipped to prove the main result of this section, Corollary~\ref{locexit}.

\begin{proof}[Proof of Corollary~\ref{locexit}]
Just like in the proof of Lemma~\ref{preloc}, we employ Theorem~\ref{BcSS} 
which states that if the classifying space of 
$\Fun^{(W_n^\h)_{\neq \emptyset}}\bl([\bullet], \mathbf{\Theta}_n^\exit\br)$ 
is a complete Segal space, then it is equivalent to the localization of 
$\mathbf{\Theta}_n^\exit$ on $(W_n^\h)_{\neq \emptyset}$. 
First we need to show that there is an equivalence of simplicial spaces 
from the classifying space of 
$\Fun^{(W_n^\h)_{\neq \emptyset}}\bl([\bullet], \mathbf{\Theta}_n^\exit\br)$ 
to $\Hom_{\Cat_\infty}\Bigl([\bullet], \Exit\bigl(\Ran(\R^n)\bigr)\Bigr)$. 
Observe the following diagram of simplicial spaces
\begin{equation} \label{diagram} 
\begin{tikzcd}[row sep=25, column sep=10]
\B \Fun^{(W_n^\h)_{\neq \emptyset}}\bl([\bullet], \mathbf{\Theta}_n^\exit\br) 
\ar[hookrightarrow]{r}{Lem.~\ref{mono3}} \ar[dashrightarrow]{d} 
& \B \Fun^{W^\h_n}\bl([\bullet], \mathbf{\Theta}_n^{\Act, \h}\br) \ar{r}{\simeq} 
& \B \Fun^{W_n}\bl([\bullet], \mathbf{\Theta}_n^\Act\br) \ar{ld}{\simeq} \\
 \Hom_{\Cat_\infty}\Bigl([\bullet], \Exit\bigl(\Ran(\R^n)\bigr)\Bigr) 
 \ar[hookrightarrow]{r} & \Hom_{\Cat_\infty}\Bigl([\bullet], \Exit\bigl(\Ranu(\R^n)\bigr)\Bigr) & 
 \end{tikzcd} 
 \end{equation}
which we explain as follows. 
The top horizontal arrow on the left is a monomorphism by Lemma~\ref{mono3},
the top horizontal arrow on the right is an equivalence by Lemma~\ref{extadj}, and
the downward arrow on the right is an equivalence by Lemmas \ref{class1} and \ref{class2}. 
Because $\Exit\bigl(\Ran(\R^n)\bigr)$ is a $\infty$-subcategory 
of $\Exit\bigl(\Ranu(\un{\R}^n)\bigr)$, the bottom horizontal arrow is a monomorphism. 
Lastly, to define the induced downward functor on the left of (\ref{diagram}), 
first recall the definition of $\Exit(\Ran(M))$ for $M=\R^n$ as an $\infty$-subcategory of $\Exit(\Ranu(\R^n))$ 
given as a pullback over surjective finite sets (Def.~\ref{little exit}). 
The downward arrow on the righthand side of (\ref{diagram}) 
is induced by the unique (up to a contractible space of choices) 
functor given by the universal property of pullback in the following diagram of $\infty$-categories
\begin{equation} \label{cube} 
\begin{tikzcd}[row sep=15, column sep=-19]
 & \mathbf{\Theta}_n^\Act \ar{rr} \ar{dd} && \Exit\bigl(\Ranu(\R^n)\bigr) \ar{dd} \\
\mathbf{\mathbf{\Theta}}_n^\exit \ar[hookrightarrow]{ur}{} 
\ar[dashrightarrow,crossing over]{rr}{\exists !} \ar{dd} 
&& \Exit\bigl(\Ran(\R^n)\bigr) \ar[hookrightarrow]{ur}{} & \\
&  \Fun\bigl(\{1< \cdots < n\}, \Fin^\op\bigr) \ar{rr}[xshift=2ex]{\ev_n}   &&  \Fin^\op    \\
 \Fun\Bigl(\{1< \cdots < n\}, \bl(\Fin_{\neq \emptyset}^\surj\br)^\op\Bigr) 
 \ar{rr}{\ev_n} \ar[hookrightarrow]{ur} &&  \bl(\Fin_{\neq \emptyset}^\surj\br)^\op. 
 \ar[hookrightarrow]{ur} \ar[from=uu,crossing over]  & 
\end{tikzcd} 
\end{equation}
Note that the top, back horizontal functor is the localization from Theorem~\ref{loc} 
and the square on the right wall is the definition of $\Exit\bigl(\Ranu(\R^n)\bigr)$ 
as pullback (Def. \ref{little exit}). 
Also note that we apply the universal property of the classifying space to 
ensure that the unique functor in (\ref{cube}) from $\mathbf{\Theta}_n^\exit$ 
to $\Exit\bigl(\Ran(\R^n)\bigr)$ induces the fuctor
\begin{equation} \label{kappa}
\B \Fun^{(W_n^\h)_{\neq \emptyset}}\bl([\bullet], \mathbf{\Theta}_n^\exit\br) 
\xra{\kappa} \Hom_{\Cat_\infty}\Bigl([\bullet], \Exit\bigl(\Ran(\R^n)\bigr)\Bigr)
\end{equation}
on the lefthand side of (\ref{diagram}).
We wish to show that $\kappa$ is an equivalence. 
First, observe that monomorphisms enjoy the `2 out of 3' property 
(Obs.~5.4, \cite{AFR2}) and thus, $\kappa$ is a monomorphism. 

All that remains to be shown then is that $\kappa$ induces a surjection 
on path components between each space given by the value on $[p]$ 
$$\B \Fun^{(W_n^\h)_{\neq \emptyset}}\bl([p], \mathbf{\Theta}_n^\exit\br) 
\xra{\kappa} \Hom_{\Cat_\infty}\Bigl([p], \Exit\bigl(\Ran(\R^n)\bigr)\Bigr).$$ 
Recall in the beginning of the proof of Lemma~\ref{class2} 
that we show $\B \Fun^{W_n}\Bigl([\bullet], \Exit\bigl(\Ranu(\R^n)\bigr)\Bigr)$ 
satisfies the Segal condition. 
Observe that the same argument applies to  
$\B \Fun^{(W_n^\h)_{\neq \emptyset}}\bl([\bullet], \mathbf{\Theta}_n^\exit\br)$ 
to show that it, too, satisfies the Segal condition. 
Thus, to show $\kappa$ is a surjection on path components, 
it suffices to show it for the cases $p=0,1$. 

For the case $p=0$, we wish to show that the map of spaces induced by $\kappa$
$$\B (W_n^\h)_{\neq \emptyset} \xra{\kappa} \ds \coprod_{r \geq 1} \Conf_r(\R^n)_{\Sigma_r}$$ 
is a surjection on path components. 
This follows from Lemma~\ref{pis0}. Indeed, in the proof we showed 
a homotopy equivalence between the classifying space of 
$W^\h_n$ and the coproduct 
$\ds \coprod_{r \geq 0}\Conf_r(\R^n)_{\Sigma_r}$ (\ref{peas}, \ref{carrots}).
The only difference in this case is that $r=0$ is excluded.

For the case $p=1$, we wish to show that the map of spaces induced by $\kappa$
$$\B \Fun^{(W_n^\h)_{\neq \emptyset}}\bl([1], \mathbf{\Theta}_n^\exit\br) 
\xra{\kappa} \mor\Bigl(\Exit\bigl(\Ran(\R^n)\bigr)\Bigr)$$ 
is a surjection on path components. 
Let $\cylr(S \xra{f} T) \xhra{E} \R^n \times \Delta^1$ be a point in the target. 
Recall from (\ref{cube}) that $\kappa$ is determined by 
$\mathbf{\Theta}_n^\exit \hra \mathbf{\Theta}_n^\Act 
\simeq \Exit\bigl(\Ranu(\un{\R}^n)\bigr)$ over  
$ \Fun\Bigl(\{1< \cdots < n\}, \bl(\Fin_{\neq \emptyset}^\surj\br)^\op\Bigr)$. 
Thus, we wish to identify a point in the fiber over $E$ under $\kappa$ 
by identifying a morphism in $\Exit\bigl(\Ranu(\un{\R}^n)\bigr)$ over $\bl(\Fin^\surj\br)^\op$. 
Such a morphism is precisely obtained by naming the projection data of $E$, namely,

\[ \begin{tikzcd} 
\cylr(S \xra{f} T) \ar[twoheadrightarrow]{d}[swap]{\pr_{<n}} 
\ar[hookrightarrow]{r}{E} & \R^n \times \Delta^1 \ar[twoheadrightarrow]{d} \\
\cylr(\pr_{<n}(S) \ra \pr_{<n}(T)) \ar[twoheadrightarrow]{d} \ar[hookrightarrow]{r} 
& \R^{n-1} \times \Delta^1 \ar[twoheadrightarrow]{d} \\
\vdots \ar[twoheadrightarrow]{d} & \vdots \ar[twoheadrightarrow]{d} \\
\cylr(\pr_1(S) \ra \pr_1(T)) \ar[hookrightarrow]{r} & \R \times \Delta^1 
\end{tikzcd} \]
the value of which under the functor 
$\Exit\bigl(\Ranu(\un{\R}^n)\bigr) \ra  \Fun\bl(\{1 < \cdots < n\}, \Fin^\op\br)$ 
factors through  $\Fun\Bigl(\{1 < \cdots < n\}, \bl(\Fin^\surj\br)^\op\Bigr)$ 
precisely because $\Exit\bigl(\Ran(\R^n)\bigr)$ is naturally over 
$\Fun\Bigl(\{1 < \cdots < n\}, \bl(\Fin^\surj\br)^\op\Bigr)$ 
as we can see by its definition as pullback (Def. \ref{little exit}). 
Thus, this morphism in $\Exit\bigl(\Ranu(\un{\R}^n)\bigr)$ 
defines a morphism in $\mathbf{\Theta}_n^\exit$ whose value under $\kappa$ is $E$. 
Thus, $\kappa$ for the case $p=1$ is a surjection on path components which, 
as previously argued, implies that $\kappa$ (\ref{kappa}) is an equivalence of simplicial spaces. 

The target of $\kappa$, $\Hom_{\Cat_\infty}\Bigl([\bullet], \Exit\bigl(\Ran(\R^n)\bigr)\Bigr)$
is, in particular, a complete Segal space, and hence, 
$\B \Fun^{(W_n^\h)_{\neq \emptyset}}\bl([\bullet], \mathbf{\Theta}_n^\exit \br)$ 
is a complete Segal space as well.
As such, the hypothesis of Theorem~\ref{BcSS} is satisfied 
and we establish that $\mathbf{\Theta}_n^\exit$ localizes on 
$(W_n^\h)_{\neq \emptyset}$ to $\Exit\bigl(\Ran(\R^n)\bigr)$

Lastly, to verify that the localization of Corollary~\ref{locexit} 
is over $(\Fin^\surj)^\op$, recall that the equivalence $\kappa$
$$\B \Fun^{(W_n^\h)_{\neq \emptyset}}\bl([\bullet], \mathbf{\Theta}_n^\exit\br) 
\simeq \Hom_{\Cat_\infty}\Bigl([\bullet], \Exit\bigl(\Ran(\R^n)\bigr)\Bigr)$$ 
is induced by the functor $\mathbf{\Theta}^\exit_n \ra \Exit\bigl(\Ran(\R^n)\bigr)$ 
from (\ref{cube}), which is in particular over $\bl(\Fin^\surj \br)^\op$.
This implies that the localization, too, is over $\bl(\Fin^\surj\br)^\op$. 
\end{proof}

\appendix
\section{Stratified spaces} \label{stratapp}
Lurking in the background of this work with only a few important appearances 
deep within the proof of our main result is the notion of a stratified space. 
This section is devoted to reviewing stratified spaces in order to develop 
two important examples,
namely the stratification of the configuration space of $r$ 
unordered points in $\R^n$ by its Fox-Neuwirth cells
and the stratification of the Ran space of a nonempty smooth connected manifold $M$ by cardinality.
Along the way we review two notions of stratified spaces, namely
topologically and Whitney stratified spaces. 
(Note that we touch on conical and conically smooth stratified spaces in the next section (\S\ref{exit sec}).)
A key result of this section is that the stratification of the configuration space 
$\R^n$ by its Fox-Neuwirth cells is Whitney stratified.
We begin with the simplest version of a stratified space as follows.

\begin{definition}[Def. A.5.1, \cite{Lu1} \& Def. 2.1.10, \cite{AFT}] \label{strat} 
\begin{itemize} \item[] {}
\item (Topological structure on a poset.) 
Let $P$ be a partially ordered set. We equip $P$ with the topology 
that defines $U \subset P$ to be open if and only if it is \emph{closed upwards}: 
that is, if $ \ds a \in U$, then every $b \geq a$ is also in $U$. 

\item  A \emph{topologically stratified space} $X \xrightarrow{\sigma} P$ is a paracompact, 
Hausdorff topological space $X$ together with a poset $\ds P$ and a continuous map 
$\ds \sigma$ such that for each $p \in P$, the fiber over $p$ is nonempty and connected. \\
The fiber over $p \in P$ is called the \emph{$p$-stratum}, which we denote by $X_p$.

\item A \emph{stratified map} from $\ds \left(X \ra P\right)$ to $\left(Y \ra Q\right)$ 
is a continuous map $X \xra{f} Y$ such that 

\[ \ds 
\begin{tikzcd}
X \arrow{r}{f} \arrow{d} &  Y \arrow{d} \\ 
 P \arrow{r}{ } & Q 
\end{tikzcd}
\]
is a commutative diagram of topological spaces.
\end{itemize} \end{definition}

We will denote a stratified space $X \xra{S} P$ by its underlying topological space $X$, 
if we expect $S$ and $P$ to be understood. 

\begin{example}[Ex. 1.2.9, \cite{AFT}] \label{standard} 
The standard stratification of the topological $k$-simplex 
$$ \ds \Delta^k := \big\{\left(t_0,...,t_k\right) \in [0,1]^{k+1} 
\mid \ds \Sigma_{i=0}^k t_i =1 \big\} \rightarrow [k]:=\{0< \cdots < k\}$$
 is given by $$\ds \left(t_0,...,t_k\right) \mapsto \text{max} \{i \mid t_i \neq 0\}.$$
This stratification for $\Delta^0$, $\Delta^1$, and $\Delta^2$ 
is depicted below, from left to right, respectively:

\centering 
\begin{tikzpicture}[scale=1]

\path[fill=green] (4,0) -- (5,1) -- (6,0) -- (4,0); 
\draw[orange, ultra thick] (4,0) -- (5,1);

\draw[orange, ultra thick] (1,0) -- (2,0);
\filldraw[black] (-1,0) circle (2pt) node {};
\filldraw[black] (1,0) circle (2pt) node {};
\filldraw[black] (4,0) circle (2pt) node {};

\path 
(-1,.1) node[anchor=south] () {{\color{black} $0$}}
(1.73,.1) node[anchor=south] () {{\color{orange} $1$}}
(1,.1) node[anchor=south] () {{\color{black} $0$}}
(3.9,0) node[anchor=east] () {{\color{black} $0$}}
(4.4,.5) node[anchor=south] () {{\color{orange} $1$}}
(5.8,.3) node[anchor=south] () {{\color{green} $2$}};

\end{tikzpicture}
\end{example}

\subsection{The Fox-Neuwirth cell stratification} \label{FNsec}
We develop an important example of a stratified space,
namely the stratification of the configuration space of $\R^n$ by its Fox-Neuwirth cells.

Let $A$ be a finite, nonempty set, 
and $\Conf_A\left(\R^n\right)$ denote the configuration space of $A$-labeled points in $\R^n$.
In \cite{FN} Fox-Neuwirth decompose $\Conf_A\left(\R^n\right)$ into open cells
according to coordinate-coincidence.
Not only does this decomposition stratify $\Conf_A\left(\R^n\right)$, 
but the stratification is particularly nice as we prove in Theorem~\ref{strat thm}.
We begin by defining the stratifying poset of $\Conf_A\left(\R^n\right)$ 
following Ayala-Hepworth in \cite{AH}

\begin{definition}
	\begin{itemize}
		\item[]
		\item An \emph{$n$-ordering} on $A$ is a pair $\left(S,\sigma\right)$ consisting of 
		a healthy planar level tree (Def. \ref{tree}) $S$ 
		of height $n$ and a bijection $\sigma$ between $A$ and level-$n$ leaves of $S$. 
		We will usually leave the $\sigma$ implicit and denote an $n$-ordering by simply $S$.
		\item The leaves of $S$ inherit a canonical linear order, 
		which induces a linear ordering on $A$ that we denote by $\leq_S$. 
		\item Given $a, b \in A$, the \emph{branching level} of $a$ and $b$, 
		denoted $b_S\left(a,b\right)$, is defined to be the level of the vertex at which the directed paths 
		from $a$ and $b$ to the root first meet. 
		\item The \emph{poset of $n$-orderings on $A$}, 
		denoted $nOrd\left(A\right)$, is the poset of $n$-orderings on $A$. 
		A morphism $S \to T$ exists if and only if the following condition, 
		called the \emph{branching condition}, holds:
		for all $a, b \in A$, $b_T\left(a,b\right) \leq b_S\left(a,b\right)$, 
		with equality only if the orderings of $a$ and $b$ under $<_T$ and $<_S$ agree. 
	\end{itemize}
\end{definition}

\begin{definition}[Lem. 13, \cite{AH}] \label{conf image}
	For an injection $\phi: A \hookrightarrow \R^n$, 
	define $S_\phi$ to be a tree in $nOrd\left(A\right)$ with the following properties:
	\begin{itemize}
		\item The linear order on the set of leaves labeled by $A$ 
		is given by the lexiographic order on $A$ in $\R^n$. 
		\item For all $a,b \in A$, the branching number $b\left(a,b\right)$ 
		is the largest integer $i$ such that $\phi\left(a\right)_i=\phi\left(b\right)_i$. 
	\end{itemize}
\end{definition}

\begin{example}  Figure~\ref{trees1} depicts a tree $S_\phi$ 
and its associated injection $\phi: A=\{1,2,3,4\} \hookrightarrow \R^2$.

	 \begin{figure}[ht]   
	\begin{tikzpicture}[scale=.5]
		\draw[black] (-1,0) -- (4,0);
		\draw[black] (0,-1) -- (0,4);
		
		\filldraw[black] (1.5,1.75) circle (4pt);
		\filldraw[black] (3,1) circle (4pt);
		\filldraw[black] (3,2.8) circle (4pt);
		\filldraw[black] (3,2.3) circle (4pt);
		
		\path
		(1.5,1.75) node[label=left:2] () {}
		(3,1) node[label=right:3] () {}
		(3,2.8) node[label=right:4] () {}
		(3,2.3) node[label=right:1] () {};
		
		\draw[<->] (6.5,2) -- (7.5, 2);
	
		\filldraw[black] (10,1.5) circle (4pt) {};
		\filldraw[black] (11,0) circle (4pt){};
		\filldraw[black] (12,1.5) circle (4pt) {};
		\filldraw[black] (12,3) circle (4pt) {};
		\filldraw[black] (11,3) circle (4pt){};
		\filldraw[black] (10,3) circle (4pt){};
		\filldraw[black] (13,3) circle (4pt){};
		
		\path
		(10,3) node[label=above:2] () {}
		(11,3) node[label=above:3]() {}
		(12,3) node[label=above: 1](){}
		(13,3) node[label=above:4]() {};

		\draw
		(10,1.5) -- (11,0)
		(10,1.5) -- (10,3)
		(11,0) -- (12,1.5)
		(12,1.5) -- (12,3)
		(12,1.5) -- (11,3)
		(12,1.5) -- (13,3);
	\end{tikzpicture}
	\caption{The image of $\phi$ on the left; $S_\phi$ on the right}
	\label{trees1}
	\end{figure}
\end{example}

We stratify the space of configurations of $A$-labeled points 
in Euclidean space over the poset $nOrd\left(A\right)$. 

\begin{lemma} \label{strat lem}
	The map of sets $\nu: \Conf_A\left(\R^n\right) \lra nOrd\left(A\right)$ given by 
	$$ \left(A \overset{\phi}\hookrightarrow\R^n\right) \mapsto S_\phi$$
	where $S_\phi$ is from Definition \ref{conf image}, is a topologically stratified space.
\end{lemma}

\begin{proof}
	In Lemma 13 of \cite{AH} they show that $S_\phi$ is unique. Thus, $\nu$ is well defined. \\
	To show continuity, we will show that the preimage of a closed subset is closed. 
	Take the downward closure of $S_\phi$ in $nOrd\left(A\right)$
	and consider its preimage $\Conf_A\left(\R^n\right)_{\leq S_\phi} \subset \Conf_A\left(\R^n\right)$. 
	Following Definition 11 of \cite{AH}, 
	the preimage consists of all those injections $\phi: A \hra \R^n$ with the following properties: 
	For each pair $a,b \in A$ with $a <_{S_\phi} b$, 
	\begin{itemize}
		\item $\phi\left(a\right)_i = \phi\left(b\right)_i$ for all $i=1,...,b_{S_\phi}\left(a,b\right),$
		\item $\phi\left(a\right)_i \leq \phi\left(b\right)_i$ for $i=b_{S_\phi}\left(a,b\right)+1.$
	\end{itemize}
	From this description, it is evident that the preimage is closed in $\Conf_A\left(\R^n\right)$. 
\end{proof}

It is evident that the strata of $\nu: \Conf_A\left(\R^n\right) \to nOrd\left(A\right)$ 
are precisely the Fox-Neuwirth cells. 
Thus, we make the following definition.

\begin{definition} \label{defFN}
The \textit{Fox-Neuwirth cell stratification} of $\Conf_A\left(\R^n\right)$ 
 is the topologically stratified space 
$\nu: \Conf_A\left(\R^n\right) \lra nOrd\left(A\right)$ defined in Lemma \ref{strat lem}.
\end{definition}

\begin{terminology}
We denote the Fox-Neuwirth cell stratification $\Conf_A\left(\R^n\right) \lra nOrd\left(A\right)$
by $\Conf_A\left(\un{\R}^n\right)$ when we omit the stratifying poset. 
\end{terminology}

\subsubsection{The Fox-Neuwirth cell stratification is Whitney stratified}
There are many stronger notions of stratified spaces used in topology and geometry.
Among these are Whitney stratified spaces due to Mather and Thom \cite{Mat1, Mat2, Th}.
Our immediate goal is to prove that $\Conf_A\left(\un{\R}^n\right)$
is a Whitney stratified space. 
Since $\Conf_A\left(\R^n\right)$ is a subset of Euclidean space, 
we only need the definition of a Whitney stratification of a subset of Euclidean space. 

\begin{definition}[Defs. 2.3 \& 2.5, \cite{NV}] \label{whit}
Let $X, Y$ be smooth submanifolds of $\R^n$, and let $y \in Y$ be a point. 
	\begin{itemize}
		\item (Whitney's condition in $\R^n$.) 
		The pair $\left(X,Y\right)$ is said to satisfy \emph{Whitney's Condition B} at $y$ 
		if for any two sequence of points 
		$\left(x_i\right) \subset X$ and $\left(y_i\right) \subset Y$ both converging to $y$ 
		such that the sequence of tangent planes $T_{x_i}X$ converges to some vector space $\tau$ 
		in the $r$-Grassmannian of $\R^n$ and the sequence of secant lines $x_iy_i$ 
		converges to some line $l$ in the $1$-Grassmannian of $\R^n$, then $l \subset \tau$. 
		\item (Whitney stratification in $\R^n$.)
		A topological stratification $X \to P$ of a subset 
		of Euclidean space is a \textit{Whitney stratification} 
		if each stratum of $X$ is a smooth submanifold of Euclidean space 
		and Whitney's condition B is satisfied at every point in each pair of strata. 
	\end{itemize}
\end{definition}
	
\begin{theorem} \label{strat thm}
	The Fox-Neuwirth cell stratification $\nu: \Conf_A\left(\R^n\right) 
	\lra nOrd\left(A\right)$ is a Whitney stratification. 
\end{theorem}

We need the following lemma to prove Theorem~\ref{strat thm}.

\begin{lemma} \label{strat lemma}
	Let the cardinality of $A$ be $k$.
	Let $S$ be a tree in $nOrd\left(A\right)$ and $x$ be a 
	configuration in the $S$-stratum $\Conf_A\left(\R^n\right)_S$. 
	\begin{enumerate}
	\item \label{uno} The tangent space $T_x\Conf_A\left(\R^n\right)$ to $x$ 
	along the $S$-stratum is canonically identified with
	Euclidean space $\R^{r_S} \subset \R^{nk}$ for some integer $r_S \leq nk$. 
	\item \label{dos} There is a canonical closed embedding 
		$$\Conf_A\left(\R^n\right)_{\leq S} \hookrightarrow T_x\Conf_A\left(\R^n\right)$$
	of the fiber over the downward closure of $S$ in $nOrd\left(A\right)$ 
	into the tangent space at the point $x$ in the $S$-stratum.
	\end{enumerate}
\end{lemma}

\begin{proof}
	Let $S$ be a tree in $nOrd\left(A\right)$ and let $x \in \Conf_A\left(\R^n\right)_S$ 
	be a point in the $S$-stratum.
	We can describe the $S$-stratum explicitly as all those injections $\phi: A \hra \R^n$ 
	with the following properties: 
	For each pair $a,b \in A$ with $a <_{S} b$, 
	\begin{itemize}
		\item $\phi\left(a\right)_i = \phi\left(b\right)_i$ for all $i=1,...,b_{S}\left(a,b\right),$
		\item $\phi\left(a\right)_i < \phi\left(b\right)_i$ for $i=b_{S}\left(a,b\right)+1.$
	\end{itemize}
	This evidently describes an open subset of some Euclidean space $\R^{r_S} \subset \R^{nk}$
	whose dimension $r_S$ is determined by the branching numbers of each pair $a <_S b$. 
	Thus, the tangent space $T_x\Conf_A\left(\R^n\right)_S$ to the point $x$ 
	in the $S$-stratum is canonically identified with $\R^{r_S}$, verifying (\ref{uno}). 
	
	Observe that there is an equality 
		$$\overline{\Conf_A\left(\R^n\right)}_S = \Conf_A\left(\R^n\right)_{\leq S}$$
	between the closure of the $S$-stratum and 
	the fiber $\Conf_A\left(\R^n\right)_{\leq S}$ over the downward closure of $S$ in $nOrd\left(A\right)$.
	This equality is evident from their explicit descriptions -
	see the proof of Lemma~\ref{strat lem} for the description of $\Conf_A\left(\R^n\right)_{\leq S}$.
	In particular this means that the downward closure 
	$\Conf_A\left(\R^n\right)_{\leq S}$ is a closed subset of $\R^{r_S}$.
	The canonical equivalence $T_x\Conf_A\left(\R^n\right) = \R^{r_S}$ yields a canonical closed embedding 
	$\Conf_A\left(\R^n\right)_{\leq S} \hookrightarrow T_x\Conf_A\left(\R^n\right)$, which verifies (\ref{dos}). 
\end{proof}

\begin{observation} \label{strat obs}
	Let $V$ be a vector subspace of a finite dimensional vector space $W$. 
	The limit of a convergent sequence of vector subspaces in $V$ 
	(in the $r$-Grassmannian of $V$) will be a subspace of $V$
	(as opposed to being in $V^\perp$). 
\end{observation}

\begin{proof}[Proof of Theorem~\ref{strat thm}]
	In the proof of Lemma~\ref{strat lemma} we identified that each stratum is an open subset of
	$\R^r \subset \R^{nk}$ for some $r$ depending on the stratum. 
	Hence, each stratum is a smooth submanifold of $\R^{nk}$. 
	We will verify Whitney's condition B (Def. \ref{whit}):
	Let $S$ and $T$ be trees in $nOrd\left(A\right)$ with $T<S$, 
	$\left(x_i\right) \subset \Conf_A\left(\R^n\right)_S$ and $\left(y_i\right) \subset \Conf_A\left(\R^n\right)_T$ 
	be sequences in the $S$- and $T$-strata
	converging to a point $y \in \Conf_A\left(\R^n\right)_T$ in the $T$-stratum
	such that the sequence of tangent spaces $T_{x_i}\Conf_A\left(\R^n\right)_S$ is convergent 
	and the sequence of secant lines $x_iy_i$ converges to a line $l$ in the $1$-Grassmannian. 
	
	The assumption $T < S$ implies that $\Conf_A\left(\R^n\right)_T \subset \Conf_A\left(\R^n\right)_{\leq S}$. 
	Thus, the canonical closed embedding $\Conf_A\left(\R^n\right)_{\leq S} \hookrightarrow T_{x_i}\Conf_A\left(\R^n\right)_S$ 
	from Lemma~\ref{strat lemma} tells us that $y_i \in T_{x_i}\Conf_A\left(\R^n\right)_S$ for each $i$.
	Moreover, $y_i \in T_{x_i}\Conf_A\left(\R^n\right)_S$ implies that each secant line $x_iy_i$ is in $T_{x_i}\Conf_A\left(\R^n\right)_S$.
	
	Translating our tangent spaces $T_{x_i}\Conf_A\left(\R^n\right)_S$ 
	to the origin via their canonical identifications with $\R^{r_S}$
	from Lemma~\ref{strat lemma} yields the situation of Observation~\ref{strat obs}: 
	The sequence of secant lines $x_iy_i$ is a sequence of linear subspaces in $\R^{r_S}$
	converging to some $l$, which must also be in $\R^{r_S}$.
	Lastly, the sequence of tangent spaces to the $x_i$ is the constant sequence of $\R^{r_S}$
	which must converge to $\R^{r_S}$. 
	This verifies Whitney's condition B at any point in any stratum of  
	$\nu: \Conf_A\left(\R^n\right) \lra nOrd\left(A\right)$, 
	which implies that $\nu$ is a Whitney stratified space. 
\end{proof}

We now consider the unordered configuration spaces of Euclidean space and 
argue that Theorem~\ref{strat thm} implies that they, too, are Whitney stratified. 

\begin{observation} \label{toe}
	The action of $\Sigma_A$ on the configuration space 
	given by permuting the labels of the points 
	extends to an action
	on the Fox-Neuwirth stratification of the configuration space 
	by permuting the leaves in $nOrd\left(A\right)$.
	We quotient out by this action as indicated in the following diagram 
	\[ \begin{tikzcd}
		\Conf_A\left(\R^n\right)  \ar{r}{\nu} \ar[swap]{d}{/\Sigma_A} 
		& nOrd\left(A\right) \ar{d}{/\Sigma_A} \\
		\Conf_A\left(\R^n\right)_{\Sigma_A} \ar{r}{\nu_{\Sigma_A}} 
		& nOrd\left(A\right)_{\Sigma_A}
	\end{tikzcd} \]
	to obtain the \textit{Fox-Neuwirth stratification of the unordered 
	configuration space of $\R^n$}, $\Conf_A\left(\un{\R}^n\right)_{\Sigma_A}$.
\end{observation}

The quotient map from the ordered to the unordered configuration space 
is a smooth covering space and is thus, a local diffeomorphism.
 Since the Whitney condition is local, the following corollary to 
 Theorem~\ref{strat thm} is immediate. 

\begin{cor} 
	The unordered Fox-Neuwirth stratification 
	$\nu_{\Sigma_A}: \Conf_A\left(\R^n\right)_{\Sigma_A} \lra nOrd\left(A\right)_{\Sigma_A}$ 
	is a Whitney stratification. 
\end{cor}

\subsection{The Ran space}
There is one other important example of a stratified space
which is central to this work.
Recall the unordered configuration space of $r$ points of a manifold $M$, $\Conf_r\left(M\right)_{\Sigma_r}$.
As the cardinality $r$ varies, the unordered configuration spaces naturally organize together into one space. 
Following that of Beilinson-Drinfeld in \S3.4.1 of \cite{BD}, we define this space as follows.
Let $M$ be a nonempty smooth connected manifold. 

\begin{definition} \label{Ran} 
The \emph{Ran space} of $M$ is the topological space whose underlying set is
$$\ds \Ran\left(M\right) := \{ S \subset M \mid S \shs \tx{is finite and nonempty} \}$$  
and whose topology is the finest for which the maps
$$ \ds  M^r \to \text{Ran}\left(M\right)$$ 
given by $\ds \left(m_1,...,m_r\right) \mapsto \{m_1,...,m_r\}$  are continuous for all $\ds r \geq 1$.
\end{definition}

Famously, the Ran space of a connected manifold is weakly contractable \cite{CN}  
(see \cite{BD} for a pleasantly simple proof). 
However, we can remember the interesting topology of the unordered 
configuration spaces as subspaces of the Ran space by stratifying the Ran space by cardinality. 

\begin{example} \label{natural} 
The Ran space of a connected manifold $M$ 
admits a stratification over the natural numbers by cardinality
$$ \Ran\left(M\right) \ra \N$$ 
assigns to each point $S \subset M$ in $\Ran\left(M\right)$ the cardinality of the underlying set $S$. 

Observe that for $r \in \N$, the $r$-stratum of $\Ran\left(M\right)$ is 
precisely the unordered configuration space $\Conf_r\left(M\right)_{\Sigma_r}$. 
\end{example}

\subsubsection{The unital Ran space} \label{unitalRan}
We call the configuration space of all finite subsets of a manifold $M$, 
including the empty subset, the \textit{unital Ran space}. 
We union the empty set to each open set of $\Ran\left(M\right)$  
to define the topology on the unital Ran space of $M$, 
making this a very pathological space.
It is not even Hausdorff, for example, 
which means it is not a stratified space. 

\begin{remark}
There are three other notable, inequivalent topologies on the unital Ran space \cite{CL},
none of which relate to this work.
Indeed, with any of these other topologies, our main result relating configuration 
spaces of $\R^n$ to the categories $\mathbf{\Theta}_n$ does not hold. 
Our topology on the unital Ran space is
more akin to what Lurie mentions in Remark~5.5.4.12 of \cite{Lu1} in that 
the empty set is in every open set.
\end{remark}

\subsubsection{The Fox-Neuwirth Ran space of $\R^n$} \label{FNRan}
Consider a refinement (Def. \ref{refinement}) 
of the stratification of the Ran space of $\R^n$ by cardinality (Ex. \ref{natural})
by additionally stratifying each unordered configuration space living
as a subset of $\Ran(\R^n)$ by its Fox-Neuwirth cells.
More precisely, let $\un{r}$ denote the set $\{1,...,r\}$.
The \emph{Fox-Neuwirth Ran space of $\R^n$} is the stratified space 
$$\Ran(\R^n) \to  \coprod_{r \geq 1} nOrd\left(\un{r}\right)_{\Sigma_{\un{r}}}$$
determined by both cardinality and the map $\nu_{\Sigma_{\un{r}}}$
which is determined by coordinate-coincidence (Obs. \ref{toe}). 
We denote this stratified space $\Ran(\un{\R}^n)$.

\subsubsection{The Fox-Neuwirth unital Ran space of $\R^n$ does not exist} \label{FNRanu}
Extending the stratification of the Fox-Neuwirth Ran space of $\R^n$ to include the unit
faces the same challenge that we saw with the unital Ran space in \S\ref{unitalRan}.
Since the underlying topological space $\Ranu(\R^n)$ 
is not even a topologically stratified space, we cannot encode 
the stratification of the configuration spaces by the Fox-Neuwirth cells within it.

\section{Exit-path \texorpdfstring{$\infty$}{Lg}-categories} \label{exit sec}
MacPherson's `exit-path category' construction was first introduced 
as a way to classify constructible (locally constant on each stratum) 
sheaves on stratified spaces \cite{Treu}. 
This notion has since evolved into the world of $\infty$-categories.
Notable to this work is that of Lurie \cite{Lu1} and Ayala-Francis-Rozenblyum-Tanaka \cite{AFT, AFR},
both of which are the focus of this section.
Also important in this section are a few non-examples of exit-path $\infty$-categories
which motivate

\subsection{Lurie's exit-path \texorpdfstring{$\infty$}{Lg}-category}
Our first definition is based on and is equivalent to Definition~A.6.2 of \cite{Lu1}.

\begin{definition} \label{exitdef} 
For a stratified space $X \xra{f} P$, 
the \emph{exit-path $\infty$-category} 
of $X$, denoted $\Exit\left(X\right)$, 
is the simplicial set whose value on $[k]$ is the set 
$$ \{ \Delta^k \xrightarrow{\sigma} X \mid \sigma \hspace{.1cm} \text{is a stratified map} \}. $$ 
\end{definition}

Informally, an object of the exit-path $\infty$-category is a point in $X$ 
and a morphism is an `exit-path' in $X$, i.e., 
paths that are allowed to witness points moving from a stratum $X_p$ to a stratum $X_q$ 
if $p \leq q$, but not if $p>q$. 
In other words, allowable paths agree with the ordering of the poset $P$. 
Moreover, these paths must exit immediately 
due to the standard stratification of the topological 1-simplex $\Delta^1$ (Ex. \ref{standard}). 
Indeed, each morphism, as a stratified map from $\Delta^1$ to $X$, 
must respect the underlying stratifications of each space. 
Likewise, 2-morphisms are certain kinds of homotopies 
between paths in $X$ governed by the standard stratification of $\Delta^2$ (Ex. \ref{standard}) 
respecting the stratification of $X$. 

The simplicial set $\Exit\left(X\right)$ is guaranteed to be an $\infty$-category 
if the stratified space $X$ is conically stratified (Theorem A.6.4, \cite{Lu1}),
which, roughly speaking, means that the stratified space 
locally looks like the Cartesian product of a cone of a stratified space 
and a topological space.

\begin{remark} 
In Theorem~A.9.3 of \cite{Lu1}, Lurie shows that the exit-path $\infty$-category 
also classifies constructible sheaves, so long as the stratified space is conical.
\end{remark}

\subsection{Ayala-Francis-Rozenblyum-Tanaka's exit-path \texorpdfstring{$\infty$}{Lg}-category} \label{AFRT}
Our second definition of an exit-path $\infty$-category, 
formulated as a small cog in the larger machine of factorization homology, 
is presented as a functor of $\infty$-categories
from \textit{conically smooth stratified spaces} to $\infty$-categories 
(modeled as complete Segal spaces).
Conically smooth stratified spaces are introduced in \cite{AFT} and are not easy to define.
Roughly speaking, a conically smooth stratified space 
locally looks like a cone of a stratified space cross Euclidean space,
and each of these local cone-like neighborhoods can be glued together smoothly.

\begin{definition}[Def. 3.3.1, \cite{AFR}] \label{AFRexit}
For a conically smooth stratified space $X$, 
the value $\Exit\left(X\right)$ is a simplicial space 
assigning the $p$-simplex $[p]$ to the hom-space 
$\Strat\left(\Delta^p, X\right)$, where $\Strat$ is 
the $\infty$-category of conically smooth stratified spaces. 
Here, $\Delta^p$ has the standard stratification (Ex. \ref{standard}). 
\end{definition}

Definition~\ref{AFRexit} is equivalent to (Def. \ref{exitdef}) for 
conically smooth stratified spaces (Lem. 3.3.9, \cite{AFR}). 
As the authors point out, the difference is model specific. 

\begin{remark}
In Corollary~3.3.11 of \cite{AFR}, Ayala-Francis-Rozenblyum 
show that their exit-path $\infty$-category
(Def. \ref{AFRexit}) classifies constructible sheaves on conically smooth stratified spaces.
\end{remark}

\begin{remark}
Ayala-Francis-Tanaka formulate yet another type of exit-path $\infty$-category,
called the `enter-path $\infty$-category' (Def. 1.1.5, \cite{AFT}).
In Lemma~3.3.9 of \cite{AFR}, Ayala-Francis-Rozenblyum show that it is 
equivalent to the opposite of Definition~\ref{AFRexit}.
\end{remark}

\subsubsection{The exit-path $\infty$-category of the 
Fox-Neuwirth cell stratification of $\Conf_r\left(\R^n\right)_{\Sigma_r}$} \label{exitFN}
We apply Ayala-Francis-Tanaka-Rozenblyum's exit-path $\infty$-category construction 
to the Fox-Neuwirth cell stratification of $\Conf_r(\R^n)_{\Sigma_r}$
to obtain an $\infty$-category which is used in the proof of our main result in \S\ref{step2}.

Theorem~1.1 of \cite{NV} states that Whitney stratifications are conically smooth.
Thus, we have the following corollary to Theorem~\ref{strat thm}.

\begin{cor} \label{conically smooth}
	The unordered Fox-Neuwirth cell stratification 
	$\Conf_A\left(\un{\R}^n\right)_{\Sigma_A}$ is conically smooth. 
\end{cor}

We may now input $\Conf_r\left(\un{\R}^n\right)_{\Sigma_r}$ 
into Ayala-Francis-Rozenblyum-Tanaka's 
exit-path $\infty$-construction to obtain an $\infty$-category
$\Exit\left(\Conf_r\left(\un{\R}^n\right)_{\Sigma_r}\right)$.

Heuristically, the objects of $\Exit\left(\Conf_r\left(\R^n\right)_{\Sigma_r}\right)$ 
are unlabeled subsets of $r$ points in $\R^n$.
A morphism can be described by a path 
in $\Conf_r\left(\R^n\right)_{\Sigma_r}$ (or a disjoint union 
of $r$ paths in $\R^n$) 
that witnesses the $r$ points move to a configuration 
with less coordinate coincidence.

\subsection{Motivation for the definitions of our exit-path \texorpdfstring{$\infty$}{Lg}-categories}
We end this section by discussing three non-examples of exit-path $\infty$-categories
obtained by the methods discussed in this section.
We hope that these non-examples provide context to the reader 
for our choices of Definitions \ref{unital Ran}, \ref{little exit}, and \ref{underline}.

\subsubsection{The exit-path $\infty$-category of the unital Ran space of $M$} \label{nottrivial}
Recall from \S\ref{unitalRan} that the unital Ran space of $M$ is not a stratified space.
Therefore, it cannot be input for either exit-path $\infty$-category construction discussed in this section. 
This motivates Definition~\ref{exitdef2}, wherein 
we directly define the exit-path $\infty$-category 
of the unital Ran space of $M$ as a simplicial space 
and then show that it is a complete Segal space.
Our route to obtaining this $\infty$-category allows us to 
avoid working with the unital Ran space itself while still
encoding the pathological behavior of the empty set within it,
a feature key to the relationship between the unital Ran space of $\R^n$ and the 
category $\mathbf{\Theta}_n^\Act$.

\subsubsection{The exit-path $\infty$-category of the Ran space of $M$} \label{nottrivial2}
Our second non-example of an exit-path $\infty$-category 
obtained by the methods discussed in this section is the exit-path $\infty$-category of the Ran space of $M$. 
Recall that the Ran space of $M$ is a stratified space (Ex. \ref{natural}).
This stratification, however, is neither conical (2.7 \& \S4.3, \cite{CL}),
nor, therefore, conically smooth.
This means neither Lurie's nor Ayala-Francis-Rozenblyum-Tanaka's 
exit-path $\infty$-category construction applied to $\Ran\left(M\right)$ is guaranteed to be an $\infty$-category.
We define the exit-path $\infty$-category of the Ran space of $M$ (Def.~\ref{little exit})
as a certain $\infty$-subcategory of the exit-path $\infty$-category 
of the unital Ran space of $M$. 

\begin{remark}
The Ran space of $M$ with a topology courser than that of (\ref{Ran})
is conically stratified (Thm. 4.12, \cite{CL}) and can therefore 
be input to Lurie's construction and will be guaranteed to be an $\infty$-category. 

However, we conjecture that these two different topologies on the Ran space 
do not effect the resulting exit-path $\infty$-category. 
In other words, our exit-path $\infty$-category of the Ran space of $M$
is equivalent to Lurie's construction applied to this courser Ran space of $M$,
as well as our Ran space of $M$. 
\end{remark}

\subsubsection{The exit-path $\infty$-category of the Fox-Neuwirth 
unital Ran space of $\R^n$} \label{exitFNRanu}
For our final non-example, recall the discussion 
about the nonexistent Fox-Neuwirth unital Ran space of $\R^n$ (\S\ref{FNRanu}).
Because the unital Ran space of $\R^n$ is not even Hausdorff,
we cannot stratify it at all. 
Thus, the methods discussed in this section cannot
encode the unital Ran space of $\R^n$ stratified by cardinality and Fox-Neuwirth cells,
which is something that this work seeks to do.
Our work-around is Definition~\ref{underline} wherein we define 
our desired $\infty$-category from scratch as a simplical space and then show that it 
satisfies the completeness and Segal conditions.

\begin{remark} 
However elusive the exit-path $\infty$-category of the Fox-Neuwirth unital Ran space of $\R^n$
seems, at the end of the day, it is just an ordinary category (Thm.~\ref{inteq}).
\end{remark}

\section{The categories \texorpdfstring{$\mathbf{\Theta}_n$}{Lg}} \label{Theta} \label{Theta sec}
In this section, we review \cite{Ber} which defines the category $\mathbf{\Theta}_n$ 
as the $n$-fold wreath product of the simplex category $\mathbf{\Delta}$ with itself.
Along the way, we define the wreath product of categories, 
Segal's assembly functors, active morphisms, and basic categorical notions
relevant to this work, to make for a self-contained section.

\begin{definition} \label{simplex} 
The \emph{simplex category} $\mathbf{\Delta}$ is the category 
in which an object is a nonempty, finite, linearly ordered set 
and a morphism is a non-decreasing map of sets. 
Composition is composition of maps between sets. 
\end{definition}

\begin{terminology} \label{linear}
For each object $S$ of $\mathbf{\Delta}$, there is a unique non-negative integer 
$p$ such that $S$ is canonically isomorphic to the linearly ordered set 
$[p] := \{0<\cdots < p\}$. We call $[p]$ the \emph{$p$-simplex} 
and will henceforth refer to the objects of $\mathbf{\Delta}$ as $p$-simplicies. 
\end{terminology}

\begin{terminology} \label{poset}
The category of posets ${\sf Poset}$ has an evident fully faithful 
embedding into the category of categories $ {\sf Cat}$. 
In light of this fully faithful functor, we refer to $[p]$ (Term. \ref{linear}) as either 
a linearly ordered set or as the category  whose objects are 
$\{0,1,...,p\}$ and in which there is a unique morphism from 
$i$ to $j$ precisely when $i \leq j$, and no morphism otherwise. 
\end{terminology}

\begin{definition} 
The category of pointed, finite sets $\ds \fin$ is the category 
in which an object is a finite, pointed set and a morphism is a pointed map; 
composition is evident.  
\end{definition}

\begin{notation} 
Given a finite set $S$, let $S_*$ denote the finite, pointed set $ \ds S \coprod \{*\}$. 
\end{notation}

\begin{definition} \label{wreath product} 
The \emph{wreath product} $\tx{Fin}_* \wr \mathcal{D}$ for an arbitrary category $\mathcal{D}$ 
is the category defined as follows: 
An object is a symbol $\ds S\left(d_s\right)$ where $S$ is a finite set and $\ds \left(d_s\right)_{s \in S}$ 
is a tuple of objects in $ \mathcal{D}$ indexed by $S$. 
A morphism  $ \ds S\left(d_s\right) \rightarrow T\left(e_t\right)$ consists of a pair of data:
\begin{enumerate}
\item  A morphism $\ds S_* \xrightarrow{\delta} T_* $ in $\tx{Fin}_*$  
\item  For each pair $\left(s \in S, t \in T\right)$ such that $\ds \delta\left(s\right)=t$, 
a morphism $\ds d_s \xrightarrow{\delta_{st}} e_t$ in $\mathcal{D}$.
\end{enumerate}  
Composition is given by composition in $\ds \tx{Fin}_*$ and $\mathcal{D}$. 
\end{definition}

\begin{observation} \label{frgt} 
There is a forgetful functor $\ds \tx{Fin}_* \wr \mathcal{D} \ra \tx{Fin}_* $ 
given by $\ds S\left(d_s\right) \mapsto S_*$; 
its value on morphisms is evident. 
\end{observation}

\begin{definition} \label{wreath2}
Given a category $\mathcal{C} \ra \fin$  over the category 
of based finite sets and a category $\mathcal{D}$, 
the \emph{wreath product} $\ds \mathcal{C} \wr \mathcal{D}$ is the pullback of categories
\[ \ds
\begin{tikzcd}
\mathcal{C} \wr \mathcal{D} \arrow{r} \arrow{d}  \ar[dr, phantom, "\lrcorner", 
very near start] & \fin \wr \mathcal{D} \arrow{d} \\  
\mathcal{C} \arrow{r} & \tx{Fin}_*
\end{tikzcd}
\] 
where the vertical arrow on the right is the forgetful functor from Observation~\ref{frgt}. 
\end{definition}

\begin{note} \label{inftwreath}
The formulation of Definition \ref{wreath2} as a universal object 
naturally allows for its extension to higher categories 
and, in particular for this work, to $\infty$-categories. 
\end{note}

We take advantage of the previous observation 
and define Joyal's category $\mathbf{\Theta}_n$ over $\fin^\op$
inductively as the $n$-fold wreath product of the simplex category $\mathbf{\Delta}$ with itself.

\begin{definition} 
The \emph{assembly functor} 
$$\ds {\sf Fin}_* \wr {\sf Fin}_*\xrightarrow{\nu} {\sf Fin}_*$$ 
is given by the wedge sum. Explicitly, the value of $\nu$ on an object 
$ \ds S\left(\left(T_s\right)_*\right)$ is the wedge sum $\ds \bigvee_{s \in S} \left(T_s\right)_*$. 
Its value on a morphism $ \ds S\left(\left(T_s\right)_*\right) \ra S'\left(\left(T'_{s'}\right)_*\right)$ given by 
\begin{enumerate} 
\item  A morphism $\ds S \xra{\delta} S'$ 
\item  For each pair $\left(s, s'\right)$ such that $\delta\left(s\right)=s'$, 
a morphism $\ds \left(T_s\right)_\ast \xra{\delta_{ss'}} \left(T_s'\right)_\ast$, is 
$$ \ds \bigvee_{s \in S} \left(T_s\right)_* \ra  \bigvee_{s' \in S'} \left({T}_{s'}'\right)_*$$
 defined by $\ds t \in T_s \mapsto \delta_{ss'}\left(t\right)$ 
 for every pair $\ds \left(s,s'\right)$ such that $\ds \delta\left(s\right)=s'$. 
\end{enumerate} 
\end{definition}

\begin{definition} 
The \emph{simplicial circle} is the functor 
$$\ds \mathbf{\Delta} \xra{\gamma} \fin^\op$$
the value of which on an object $\ds [p] $ is the quotient morphism set 
$\ds \mathbf{\Delta}\left([p],[1]\right)/\big\{\{0\}, \{1\} \big\}$, 
where $\ds \{i\}$ denotes the constant map at $i$; 
$\{0\} \sim \{1\}$ is the evident basepoint of the image. 
The value of $\gamma$ on a morphism $[p] \xra{f} [q]$  is precomposition with $f$. 
\end{definition}

\begin{observation} \label{injective} 
The map induced by $\ds \gamma$ between each hom-set is injective. 
This observation comes down to the fact that on morphisms $\gamma$ 
is given by precomposition and composition is unique in categories. 
\end{observation}

\begin{observation} 
There is an evident isomorphism $\gamma\left([p]\right) \overset{\cong}\ra \{1,...,p\}_*$ 
in $\fin$. Let $\nu_j$ denote the morphism that assigns  each $0 \leq i \leq j-1$ to $0$ 
and each $j \leq i \leq p$ to $1$ (i.e., a unique composite of degeneracy maps). 
The value of $\nu_j$ under the isomorphism is $j$. 
Its assignment on morphisms, Then is evident. 
\end{observation}

\begin{terminology} 
In light of the previous observation, we will freely refer to a non-basepoint value 
in the pointed set $\gamma\left([p]\right)$ by $j$ for some $1 \leq j \leq p$. 
\end{terminology}

\begin{definition} \label{theta} 
For each integer $\ds n \geq 1$, the categories $\ds \mathbf{\Theta}_n$ 
are defined inductively by setting $$\ds \mathbf{\Theta}_1 := \mathbf{\Delta} 
\hspace{.4cm}\tx{and} \hspace{.4cm} \mathbf{\Theta}_n := \mathbf{\Theta}_{1} \wr \mathbf{\Theta}_{n-1}$$ 
where the \emph{assembly functors} $\ds \mathbf{\Theta}_n \xra{\gamma_n} \fin^\op$ 
are also defined inductively by setting 
$$\ds \gamma_1 := \gamma \hspace{.4cm}\tx{and} \hspace{.4cm} 
\gamma_n := \nu \circ \left(\gamma_{1} \wr\gamma_{n-1}\right).$$ 
\end{definition}

Recall that a planar level tree of height $n$ is a finite rooted tree 
with an $n$-order on its set of leaves (Def.~\ref{tree}). 
These trees, in fact, naturally describe the objects of $\mathbf{\Theta}_n$, 
as we will explain in the following observation. 

\begin{observation}[\S4.5, \cite{AH}] 
Finite rooted planar level trees of height $n$ naturally 
describe the objects of $\mathbf{\Theta}_n$ as follows: 
When $n=1$, the object $[p]$ corresponds to the tree of height $1$ that has $p$ leaves. 
For $n>1$, the object $[p]\left(T_1,...,T_p\right)$ 
of $\mathbf{\Theta}_n$ corresponds to the tree described as follows: 
The tree which has $p$ vertices at level $1$; the $i$th vertex at level $1$ (according to the linear order) 
is the root of the tree which corresponds to $T_i$, 
i.e., it is the subtree consisting of those vertices for which the unique directed path 
from each one to the root intersects the $i$th vertex at level $1$. 
In terms of trees, Then the assembly functor $\gamma_n$ assigns a tree to its set of leaves. 
\end{observation}

In this paper, we will use this description 
of the objects of $\mathbf{\Theta}_n$ whenever convenient. 
In fact, we often prefer it since it makes the objects of $\mathbf{\Theta}_n$ so accessible.

\begin{example} 
The object, $[3]\left([1],[3],[0]\right)$ in $\mathbf{\Theta}_2$, 
corresponds to the following planar level tree of height $2$.

\centering
\begin{tikzpicture}[scale=.5]
\tikzstyle{every node}=[draw,shape=circle, fill=black, inner sep=1.5pt]
\path 
(-2,0) node (t0) {}
(-1,0) node (t1) {}
(0,0) node (t2) {}
(1,0) node (t3) {}
(-2,-2) node (m1) {}
(0,-2) node (m2) { }
(2,-2) node (m3) { }
(0,-4) node (r) { };
\draw (t0) -- (m1)
(t1) -- (m2)
(t2) -- (m2)
(t3) -- (m2)
(m1) -- (r)
(m2) -- (r)
(m3) -- (r);
\end{tikzpicture}
\end{example}

\begin{remark} 
Alternatively (but equivalently), Rezk defined the category $\mathbf{\Theta}_n$ 
as the full sub-category of the category of strict $n$-categories, ${\sf Cat}_n$, 
in which an object is a \emph{pasting diagram} \cite{Rezk}. 
For example, the object $[3]\left([1],[3],[0]\right)$ in $\mathbf{\Theta}_2$ 
corresponds to the pasting diagram

\[
\begin{tikzcd}[column sep=huge]
0
  \arrow[bend left=50]{r}[name=U]{}
  \arrow[bend right=50]{r}[name=D]{} &
1
  \arrow[shorten <=10pt,shorten >=10pt,Rightarrow,to path={(U) -- node[label=right:{  }] {} (D)}]{}
 \arrow[bend left=60]{r}[name=U]{}
\arrow[bend left=20]{r}[name=MU]{}
\arrow[bend right=20]{r}[name=MD]{}
  \arrow[bend right=60]{r}[name=D]{} &
2
  \arrow[shorten <=3pt,shorten >=0pt,Rightarrow,to path={(U) -- node[label=right:{  }] {} (MU)}]{}
  \arrow[shorten <=3pt,shorten >=0pt,Rightarrow,to path={(MU) -- node[label=right:{  }] {} (MD)}]{}
  \arrow[shorten <=3pt,shorten >=0pt,Rightarrow,to path={(MD) -- node[label=right:{  }] {} (D)}]{}
 \arrow{r} &
3
\end{tikzcd} 
\]
\end{remark}

\subsection{The subcategory \texorpdfstring{$\mathbf{\Theta}_n^\Act$}{Lg} 
of active morphisms of \texorpdfstring{$\mathbf{\Theta}_n$}{Lg}}
We review the definition of active morphisms in general
in order to recall the subcategory of active morphisms of $\mathbf{\Theta}_n$.
This object is central to our work in that we use it to 
encode the unordered configuration spaces of $\R^n$.

\begin{definition} \label{Fin}
The category of finite sets $\ds \Fin$ is the category in which an object is a finite set 
and a morphism is a map of sets; composition is evident.  
\end{definition}

\begin{definition}  \label{actdef} 
Given a category $ \ds \mathcal{C} \xra{F} \Fin_\ast$ over pointed, finite sets, 
a morphism $\sigma$ in $\mathcal{C}$ is \emph{active} if 
$\ds \left(F\left(\sigma\right)\right)^{-1}\left(\{*\}\right)=\{*\}$. 
The subcategory $\C^\Act$ of $\C$ is defined to be the pullback
\[ \ds \begin{tikzcd}
\C^{\Act} \ar[hookrightarrow]{r} \pb \ar{d} &  \C \ar{d}{F} \\
\Fin \ar[hookrightarrow]{r} & \fin.
\end{tikzcd} \]
\end{definition}

For a category $\mathcal{D}$, there is a monomorphism 
$\ds \Fin \wr \mathcal{D} \to \fin \wr \mathcal{D}$ and thus, 
the category $ \Fin \wr \D$ has an explicit description of objects 
and morphisms similar to that of $\ds \fin \wr \mathcal{D}$ 
as given in Definition~\ref{wreath product}. 
We will use this description of $ \Fin \wr \D$ whenever convenient.
Additionally, we would like to point out that given a category 
$ \ds \mathcal{C} \xra{F} \Fin_\ast$ over pointed, finite sets, 
$F$ restricts to a functor $\C^\Act \to \Fin_\ast$ which, by definition, 
factors through $\Fin$. Evidently Then the category $\C^\Act \wr \Fin$ 
is canonically isomorphic to $\C^\Act \wr \Fin_\ast$. 

The following alternative definition of $\mathbf{\Theta}_n^\Act$ 
coincides with that of Definition~\ref{actdef}. 

\begin{definition} \label{act}
For each integer $n \geq 1$, the categories $\ds \mathbf{\Theta}_n^{\Act}$ 
are the subcategories of $\ds \mathbf{\Theta}_n$ defined inductively by setting 
$$ \mathbf{\Theta}_1^{\Act} := \mathbf{\Delta}^{\Act}$$ and defining 
$\ds \mathbf{\Theta}_n^{\Act} := \mathbf{\Theta}_1^\Act \wr \mathbf{\Theta}_{n-1}^\Act$, 
i.e., the pullback
\begin{equation} 
\begin{tikzcd} 
\mathbf{\Theta}_n^{\Act} \ar{r} \pb \ar{d} & \Fin^\op \wr \mathbf{\Theta}_{n-1}^{\Act} \ar{d}{\tx{frgt}} \\
\mathbf{\Theta}_1^{\Act} \ar{r}{\gamma_1} & \Fin^\op.
\end{tikzcd} 
\end{equation} 
\end{definition}

The benefit of this particular formulation of $\mathbf{\Theta}_n^\Act$ 
is that it makes the following facts straightforward.

\begin{observation} 
Because the wreath product is associative, equivalently 
$\mathbf{\Theta}_n^\Act := \mathbf{\Theta}_{n-1}^\Act \wr \mathbf{\Theta}_1^\Act$. 
\end{observation}

\begin{observation} \label{tr} 
For each $n$, there is a natural forgetful functor 
$$\ds \mathbf{\Theta}_n^{\Act} \xra{{\sf tr}} \mathbf{\Theta}_{n-1}^{\Act}.$$  
The value on an object $T_{n-1}\left([m_k]\right) \in \mathbf{\Theta}_{n-1}^{\Act} 
\wr \mathbf{\Theta}_1^{\Act} =: \mathbf{\Theta}_n^{\Act}$ is $T_{n-1}$.\\  
Let $T_{n-1}\left([m_k]\right) \xra{\sigma} W_{n-1}\left([p_l]\right)$ be a morphism in 
$\mathbf{\Theta}_n^{\Act}$ defined by 
\begin{itemize} 
\item  a morphism $\ds T_{n-1} \xra{\sigma'} W_{n-1}$ in $\mathbf{\Theta}_{n-1}^{\Act}$ and, 
\item  a morphism $\ds [m_k] \xra{\sigma_k} [p_l]$ in $\mathbf{\Theta}_1^{\Act}$ 
for each pair $\left(k,l\right)$ such that $\ds \gamma_{n-1}\left(l\right)=k$.
\end{itemize} 
The value of $\sigma$ under ${\sf tr}$ is $\sigma'$. 
\end{observation}

\begin{example} 
Let $T$ be the object of $\mathbf{\Theta}_3$ depicted as the far left tree 
in the figure below. We depict two iterations of the truncation functor ${\sf tr}$ on $T$:

\centering
\begin{tikzpicture}[scale=.5]
\draw[|->] (-4.5,-2) -- (-3.5, -2);
\draw[|->] (3.5,-2) -- (4.5, -2);
\draw (-4, -2) node[anchor=south]() {{\sf tr}};
\draw (4, -2) node[anchor=south]() {{\sf tr}};
\tikzstyle{every node}=[draw,shape=circle, fill=black, inner sep=1.5pt]
\path 
(-2,0) node (t0) {}
(-1,0) node (t1) {}
(0,0) node (t2) {}
(1,0) node (t3) {}
(-2,-2) node (m1) {}
(0,-2) node (m2) { }
(2,-2) node (m3) { }
(0,-4) node (r) { }

(-10.5,2) node () {}
(-10, 2) node () {}
(-8.5,2) node () {}
(-8,2) node () {}

(-10,0) node () {}
(-9,0) node () {}
(-8,0) node () {}
(-7,0) node () {}
(-10,-2) node () {}
(-8,-2) node () {}
(-6,-2) node () {}
(-8,-4) node () {}

(6,-2) node () {}
(8,-2) node () {}
(10,-2) node () {}
(8,-4) node () {};

\draw 
(t0) -- (m1)
(t1) -- (m2)
(t2) -- (m2)
(t3) -- (m2)
(m1) -- (r)
(m2) -- (r)
(m3) -- (r)

(-10.5,2) -- (-10,0)
(-10,2) -- (-10,0)
(-8.5,2) -- (-8,0)
(-8,2) -- (-8,0)

(-10,0) -- (-10,-2)
(-9,0) -- (-8,-2)
(-8,0) -- (-8,-2)
(-7,0) -- (-8,-2)
(-10,-2) -- (-8,-4)
(-8,-2) -- (-8,-4)
(-6,-2) -- (-8,-4)

(6,-2) -- (8,-4)
(8,-2) -- (8,-4)
(10,-2) -- (8,-4);
\end{tikzpicture}
\end{example}

\begin{notation} 
For each $1 \leq i \leq n-1$, denote the $\left(n-i\right)$-fold composite 
of the truncation functor $\tr$ by $\tr_i: \mathbf{\Theta}_n^\Act \ra \mathbf{\Theta}_i^\Act$.  
\end{notation}

\begin{observation} \label{tr-factor} 
For each $1 \leq i \leq n-1 $, there is a natural transformation from 
$\mathbf{\Theta}_n^\Act \xra{\gamma_n} \Fin^\op$ to the composite  $ \gamma_i \circ \tr_i$, 

\[ \begin{tikzcd} 
\mathbf{\Theta}_n^\Act \ar{rd}{\gamma_n} \ar{d}[swap]{\tr_i} &  \\
\mathbf{\Theta}_i^\Act \ar{r}{\gamma_i}   & \Fin^\op.
 \end{tikzcd} \]
For each tree $T$, the natural transformation $\epsilon$ is given by 
the natural map  $\gamma_n\left(T\right) \xra{\epsilon_T}\gamma_i\left(\tr_i\left(T\right)\right)$ 
from the leaves of $T$ to the leaves of the truncation of $T$ to height $i$, 
the assignment of which is the evident one given by the structure of the tree $T$. 
It is straightforward to verify that $\epsilon$ does indeed define a natural transformation. 
\end{observation}

\begin{definition} 
Given a category $\C$, we define ${\sf Fun}\left(\{1<\cdots<n\}, \C\right)$ 
to be the category in which an object is a functor $ \{1<\cdots<n\} \ra \C$ 
which selects out a sequence of composable morphisms in $\C$:
$$c_1 \to c_2 \to \cdots \to c_n$$
and a morphism from $c_1 \to c_2 \to \cdots \to c_n$ to $d_1 \to d_2 \to \cdots \to d_n$ 
is a commutative diagram in $\C$:
\[ \begin{tikzcd}
c_1 \ar{r} \ar{d} & d_1 \ar{d} \\
c_2 \ar{r} \ar{d} & d_2 \ar{d} \\
\vdots \ar{d} & \vdots \ar{d} \\
c_n \ar{r} & d_n. 
\end{tikzcd} \] 
Composition is evident. 
\end{definition}

\begin{observation} \label{layers} 
We use Observation~\ref{tr-factor} to define the natural functor
$$\mathbf{\Theta}_n^\Act \xra{\tau_n} \Fun\left(\{1<\cdots<n\}, \Fin^\op\right)$$
the value of which on an object $T$ is the functor which 
selects out the sequence of composable maps of sets
$$ \gamma_n\left(T\right) \xra{\epsilon_T} \gamma_{n-1}\left(\tr\left(T\right)\right) 
\xra{\epsilon_{\tr\left(T\right)}} \gamma_{n-2}\left(\tr_{n-2}\left(T\right)\right) 
\ra \cdots \ra \gamma_1\left(\tr_1\left(T\right)\right)$$
and on a morphism $T \xra{f} S$ is the diagram of finite sets
 \[ \begin{tikzcd}
 \gamma_n\left(S\right) \ar{rr}{\gamma_n\left(f\right)} \ar{d}[swap]{\epsilon_S} 
 && \gamma_n\left(T\right) \ar{d}{\epsilon_T} \\
\gamma_{n-1}\left(\tr_{n-1}\left(S\right)\right) \ar{rr}{\gamma_{n-1}\left(\tr_{n-1}\left(f\right)\right)} \ar{d} 
&& \gamma_{n-1}\left(\tr_{n-1}\left(T\right)\right) \ar{d} \\
\vdots \ar{d} && \vdots \ar{d} \\
 \gamma_1\left(\tr_1\left(S\right)\right) \ar{rr}{\gamma_1\left(\tr_1\left(f\right)\right)} 
 &&  \gamma_1\left(\tr_1\left(T\right)\right) 
 \end{tikzcd} \] 
which is guaranteed to commute in $\Fin$ because 
the downward arrows in the diagram are given by 
the natural transformation $\epsilon$ from Observation~\ref{tr-factor}.
\end{observation}

\end{document}